\documentclass[a4paper, final]{amsart}

\usepackage[numbers, square]{natbib}
\usepackage[utf8]{inputenc}
\usepackage[T1]{fontenc}
\usepackage[english]{babel}
\usepackage{amsmath,amssymb,amsthm,amsfonts,amsopn}
\RequirePackage[colorlinks=true,citecolor=blue,urlcolor=blue]{hyperref}
\usepackage[author={Freddie}]{fixme}
\FXRegisterAuthor{mw}{amw}{Winkel}
\usepackage{graphicx, tikz, tikz-cd, enumerate}
\usetikzlibrary{arrows, decorations.pathreplacing}
\fxsetup{theme=color}
\fxusetargetlayout{colorcb}
\definecolor{oxfordblue}{RGB}{4,30,66}
\usepackage[shortlabels]{enumitem}
\setlist[enumerate]{label=(\roman*)}
\usetikzlibrary{shapes,backgrounds,positioning}

\theoremstyle{plain}
\newtheorem{thm}{Theorem}[section]
\newtheorem{lemma}[thm]{Lemma}
\newtheorem{prop}[thm]{Proposition}
\newtheorem{cor}[thm]{Corollary}

\theoremstyle{definition}
\newtheorem{defi}[thm]{Definition}

\newtheorem{rem}[thm]{Remark}
\newtheorem{nota}[thm]{Notation}

\newtheorem*{conjecture*}{Conjecture}

\newcommand{\tree}[1][t]{\boldsymbol{#1}}
\newcommand{\that}[1][t]{\hat{\boldsymbol{#1}}} % Deterministic

\newcommand{\treesigma}[1][\sigma]{\boldsymbol{#1}}
\newcommand{\That}[1][T]{\widehat{#1}}
\newcommand{\Thatspace}[1][\T]{\widehat{\boldsymbol{#1}}} % Space
\newcommand{\T}{\mathbb{T}}
\DeclareMathOperator{\edge}{edge}

\DeclareMathOperator{\vertices}{vert}

\DeclareMathOperator{\insertable}{ins}
\DeclareMathOperator{\branchpoints}{bp}
\DeclareMathOperator{\id}{id}
\DeclareMathOperator{\bin}{bin}
\DeclareMathOperator{\intstruct}{\textsc{int}}
\DeclareMathOperator{\rank}{\textsc{rank}}

\newcommand{\insertablef}[1][\tree]{\insertable({\tree[#1]})}
\DeclareMathOperator{\tildei}{\tilde{\textit{\i}}}
\DeclareMathOperator{\crp}{CRP}
\DeclareMathOperator{\ocrp}{oCRP}
\DeclareMathOperator{\dirmult}{DM}

\DeclareMathOperator{\betabin}{BetaBin}

\newcommand\independent{\protect\mathpalette{\protect\independenT}{\perp}}
\def\independenT#1#2{\mathrel{\rlap{$#1#2$}\mkern2mu{#1#2}}}

\newcommand{\nin}{{n \in \mathbb{N}}}

\newcommand{\deq}{\stackrel{D}{=}}

\newcommand{\dcon}{\xrightarrow{D}}

\renewcommand{\SS}{\mathbb{S}}

\newcommand{\XX}{\mathbb{X}}
\newcommand{\YY}{\mathbb{Y}}

\renewcommand{\P}{\mathrm{P}}

\newcommand{\N}{\mathbb{N}}

\newcommand{\R}{\mathbb{R}}

\DeclareMathOperator{\Unif}{Unif}

\begin{document}

\title{A down-up chain with persistent labels on multifurcating trees}

\address{\hspace{-0.42cm}Frederik~S{\o}rensen\\ Department of Statistics\\ University of Oxford\\ 24--29 St Giles'\\ Oxford OX1 3LB, UK\\ Email: frederik.soerensen@seh.ox.ac.uk}             

\author{Frederik S{\o}rensen}    



\date{\today}

\begin{abstract}
  In this paper, we propose to study a general notion of a down-up Markov chain for multifurcating trees with $n$ labelled leaves.
  We study in detail down-up chains associated with the $(\alpha, \gamma)$-model of Chen et al. (2009), generalising and further developing previous work by Forman et al. (2018, 2020) in the binary special cases.
  The technique we deploy utilizes the construction of a growth process and a down-up Markov chain on trees with planar structure.
  Our construction ensures that natural projections of the down-up chain are Markov chains in their own right.
  We establish label dynamics that at the same time preserve the labelled alpha-gamma distribution and keep the branch points between the $k$ smallest labels for order $n^2$ time steps for all $k \geq 2$.
  We conjecture the existence of diffusive scaling limits generalising the ``Aldous diffusion’’ by Forman et al. (2018+) as a continuum-tree-valued process and the ``algebraic $\alpha$-Ford tree evolution’’ by L\"{o}hr et al. (2018+) and by Nussbaumer and Winter (2020) as a process in a space of algebraic trees.
\end{abstract}

\ \vspace{-24pt}

\maketitle

\section{Introduction}\label{sec:introduction}
Let $[n] := \{1, \ldots, n\}$.
An $[n]$-tree is a (non-planar) rooted, multifurcating tree with $\nin$ leaves labelled by $1, \ldots, n$, an additional vertex of degree $1$ called the root, and in which the degree of any other vertex, called a branch point, will be at least $3$.
We will denote the space of such trees by $\T_{[n]}$.
Occasionally, we will only allow binary trees, also called $n$-leaf cladograms, and will denote this space by $\T_{[n]}^{\bin}$.
The \textit{Aldous chain}~\cite{RefWorks:doc:5b4cbc14e4b04428cc72cf41,RefWorks:doc:5b4cbc43e4b0185a9132aee9} is a Markov chain on $\T_{[n]}^{\bin}$ with a transition kernel consisting of a down-step followed by an up-step defined as follows:
\begin{itemize}[leftmargin=1in]
  \item[\textit{Down-step}:] Delete a leaf uniformly at random and contract the parent branch point away.
  \item[\textit{Up-step}:] Insert a new leaf with the same label by selecting an edge uniformly at random, inserting a new branch point on that edge and attaching the new leaf onto that branch point.
\end{itemize}
It is natural to extend the above type of Markov chain by introducing alternative ways of deleting and inserting leaves as well as by extending it to the space of multifurcating trees.
We will refer to any Markov chain on $\T_{[n]}$ (or a subset thereof) as a \textit{down-up chain} if its transition kernel can be decomposed into a down-step where a random (but not necessarily uniformly distributed) leaf is deleted, and an up-step where a new leaf is inserted into a randomly picked edge or branch point.

Noting that the up-step described above is identical to one step of R\'{e}my's growth process~\cite{RefWorks:doc:5b71b380e4b06c0731a629f4}, a specific problem is to define new down-up chains where the up-step is governed by other growth processes from the literature.
For our purposes Ford's $\alpha$-model for binary trees~\cite{RefWorks:doc:5b76ce32e4b0820c421f301d}, extending R\'{e}my's growth process~\cite{RefWorks:doc:5b71b380e4b06c0731a629f4}, and the $(\alpha, \gamma)$-growth rule for multifurcating trees~\cite{RefWorks:doc:5b4cbb5fe4b02dc0c79270af}, of which both Ford's $\alpha$-model and Marchal's stable tree growth process~\cite{RefWorks:doc:5b6c561fe4b06c0731a5c558} are special cases, are of highest importance.
The $(\alpha, \gamma)$-growth process for $0 \leq \gamma \leq \alpha \leq 1$, is characterized by, for each $\nin$, obtaining an element of $\T_{[n+1]}$ from an element $\tree \in \T_{[n]}$ by inserting a new leaf labelled by $n+1$ into an \textit{insertable part} of $\tree$, denoted by $\insertablef$, where an insertable part $x \in \insertablef$ is picked with probability proportional to an associated weight, $w_x$, with
\begin{align}
\label{eq:alphagamma_weights}
    w_x
    = 
    \begin{cases}
      1 - \alpha & \text{if $x$ is a leaf edge,} \\
      \gamma & \text{if $x$ is an internal edge,} \\
      \left( c_x - 1 \right) \alpha - \gamma & \text{if $x$ is a branch point with $c_x$ children.} \\
    \end{cases}
  \end{align}

It is rather natural to construct a down-step such that the stationary distribution of the down-up chain is that induced by the growth process that governs the up-step.
To be more precise, let ${\left( T_n \right)}_{\nin}$ denote a growth process.
We are then looking to devise a down-step, such then when carrying out the down-step from $T_n$ we obtain $T_n^\downarrow$ with the property that $T_n^\downarrow \deq T_{n-1}$ for each $n \geq 2$.
Combinining the down-step with the up-step from the growth process yields a stationary down-up Markov chain, ${\left( T_n(m) \right)}_{m \in \N_0}$, if $T_n(0) \deq T_n$.
If the leaf labels are exchangeable, i.e.\ if for any permutation $\sigma$ of $[n]$ the probability of obtaining an $[n]$-tree, $\tree$, is the same as the probability of obtaining $\tree$ where leaf $i$ is labelled by $\sigma(i)$ for each $i \in [n]$, a down-step with a simple uniform deletion of the leaves will suffice.
However, in the absence of exchangeability the situation is more complicated.

In~\cite{forman2018projections, RefWorks:doc:5b4cbc93e4b07f5746e47014} a modified version of the Aldous chain, the $\alpha$-chain, was introduced, where the up-step is governed by Ford's $\alpha$-model.
For $\alpha = \frac{1}{2}$, Ford's $\alpha$-model is the same as R\'{e}my's growth process, and so the leaf labels are exchangeable, but this is not true for $0 < \alpha \neq \frac{1}{2} < 1$, so a more complicated down-step was introduced:
\begin{enumerate}
  \item Down-step (selection): Select leaf $i$ uniformly at random.
  \item Down-step (local search): Defining $a$ and $b$ to be the smallest leaf labels in the first and second subtree on the ancestral line from $i$, respectively, let $\tildei := \max\left\{ i, a, b \right\}$.
  \item Down-step (swap and delete): Swap $i$ and $\tildei$, delete leaf $\tildei$, and contract away the parent branch point.
  \item Down-step (relabel): Relabel the leaves using the increasing bijection $[n] \setminus \{\tildei\} \to [n-1]$ to obtain an element of $\T_{n-1}^{\bin}$.
  \item Up-step: Insert leaf $n$ according to Ford's $\alpha$-growth rule~\cite{RefWorks:doc:5b76ce32e4b0820c421f301d} (the binary specification of the $(\alpha, \gamma)$-growth rule obtained for $\gamma = \alpha$).
\end{enumerate}
In the exchangeable case where $\alpha = \frac{1}{2}$, the relabelling in step (iv) is unnecessary and one can therefore adjust (v) accordingly and just insert the deleted leaf again, which yields the uniform chain from~\cite{RefWorks:doc:5b4cbc93e4b07f5746e47014}.

Irrespective of the exchangeability of the leaf labels, the above down-step serves a key purpose in a broader context which we will now outline.
It is well known that the stationary distribution of the Aldous chain, the uniform distribution on the space of binary trees~\cite{RefWorks:doc:5b4cbc14e4b04428cc72cf41}, converges when suitably scaled to that of the Brownian Continuum Random Tree~\cite{RefWorks:doc:5b72ecc3e4b0bc0f31737130}, as the number of leaves tends to infinity.
Aldous and Schweinsberg~\cite{RefWorks:doc:5b4cbc14e4b04428cc72cf41,RefWorks:doc:5b4cbc43e4b0185a9132aee9} showed that the relaxation time of the Aldous chain is of order $n^2$.
Aldous conjectured that running the chain $n^2$ times faster than the number of leaves tend to infinity would result in a “diffusion on continuum trees”~\cite{aldousweb}.

In a sequence of papers~\cite{forman2019diffusions, RefWorks:doc:5b720899e4b0fd36f5bb93a6, RefWorks:doc:5b720899e4b0874a74e5a439, forman2018aldous} the authors have carried out a programme to contruct this “Aldous diffusion”.
Parts of this construction arise by studying projections of the uniform chain onto decorated trees.
To be more precise, consider for an $[n]$-tree the reduced subtree spanned by the leaves labelled by $1$ and $2$, where all vertices of degree $2$ have been contracted away, and consider how many of the $n$ leaves are sitting in subtrees of the branch point and on each of the edges of that reduced tree (for a more precise description of this, see Definition~\ref{def:decoratedtrees}).
This will, for a binary tree, yield three masses summing to $n$.
For each $m \in \N_0$, let $T_n(m)$ denote the $m$th step of the uniform chain on $\T_{[n]}$, let $Y_n^0(m), Y_n^1(m), Y_n^2(m)$ denote the proportion of these masses when projecting $T_n(m)$ to a decorated $[2]$-tree of mass $n$.
In~\cite{RefWorks:doc:5b4cbb92e4b0bc982fe42f3a} a notion of a Wright-Fisher diffusion with negative mutation rates was developed, and it was shown that
\begin{align}\label{eq:WrightFisher-diffusion}
  {\left( Y_n^0(\lfloor n^2 t \rfloor), Y_n^1(\lfloor n^2 t \rfloor ) , Y_n^2(\lfloor n^2 t \rfloor ) \right)}_{t \geq 0}
  \dcon
  {\left( Y^0(t), Y^1(t), Y^2(t) \right)}_{t \geq 0}
\end{align}
as $n \to \infty$, stopped the first time one of the last two coordinates hits 0, where $\dcon$ denotes convergence in distribution and $(Y^0, Y^1, Y^2)$ is a Wright-Fisher diffusion with mutation rates $\left( \frac{1}{2}, -\frac{1}{2}, -\frac{1}{2} \right)$.
If we delete leaves uniformly at random, we will end up deleting the branch point due to one of the masses associated with the leaf edges dropping to $0$.
In the limit, the Wright-Fisher process does not provide any guidance as to how to continue the process.
But with the down-step of the $\alpha$-chain we have a mechanism by which we can identify another branch point and continue the process.
The two key aspects that make this possible, are the label swapping as well as the relabelling of the remaining leaf labels.

The aim of this paper is to study and generalize down-up chains with label swapping to the space of multifurcating trees.
Specifically, we will study chains where the up-step is governed by the $(\alpha, \gamma)$-growth process~\cite{RefWorks:doc:5b4cbb5fe4b02dc0c79270af}, which also encompasses the prime special case of the stable tree growth~\cite{RefWorks:doc:5b6c561fe4b06c0731a5c558}.
Whilst any $(\alpha, \gamma)$-growth process with $0 < \gamma \leq \alpha < 1$ has a continuum random tree as its scaling limit, the stable tree growth is of particular motivational interest due to the universality of the stable continuum random tree~\cite{RefWorks:doc:5b6c561fe4b06c0731a5c558}, and the fact that the stable tree growth induces exchangeable leaf labels.
Hence a down-up chain with an up-step governed by the $(\alpha, \gamma)$-growth process will be analogous to the Aldous and uniform chains, in terms of devising a “diffusive limit on the continuum trees”.
To envisage the construction of such a process on continuum trees it is of key importance to devise a down-step that allows the construction beyond a disappearing branch point, when projecting it to the space of decorated trees.
This is the key purpose of this paper.

%% ATTEMPT AT REWRITING
\begin{defi}%
  \label{def:alphagamma_chain}
  Fix $\nin$.
  The $(\alpha, \gamma)$-\textit{chain on} $\T_{[n]}$, {${\left( T_n(m) \right)}_{m \in \N_0}$}, is a Markov chain on $\T_{[n]}$ with a transition kernel characterized as follows:
  \begin{enumerate}
    \item Conditional on $T_n(m) = \tree$, select leaf $I = i$ uniformly at random.
    \item Let $c_v$ denote the number of children of $v$, the parent of $i$ in $\tree$.
      \begin{itemize}[leftmargin=.5cm]
        \item If $c_v = 2$, let $\tilde{I} = \max \{i, a, b\}$ where 
          \begin{align*}
            &a = \min
            \begin{Bmatrix}
              \text{leaf labels in the first spinal bush} \\ \text{on the ancestral line from leaf}\ i
            \end{Bmatrix}, \\
            &b = 
            \min
            \begin{Bmatrix}
              \text{leaf labels in the second spinal bush} \\ \text{on the ancestral line from leaf}\ i
            \end{Bmatrix}.
          \end{align*}
        \item If $c_v > 2$, let $i_1, \ldots, i_{c_v}$ denote the smallest leaf labels in the $c_v$ subtrees of $v$, enumerated in increasing order, and say that $i = i_j$ for some $j \in [c_v]$.
          Define the conditional distribution of $\tilde{I}$ by
          \begin{align*}
            \P \left( \tilde{I} = i_{j^\prime} \ \middle \vert \ T_n(m) = \tree, I = i \right)
            =
            \begin{cases}
              \frac{\alpha - \gamma}{(c_v - 1)\alpha - \gamma} & \text{if}\ j \leq 2, j^\prime = 3 \\[0.1cm]
              \frac{\alpha}{(c_v - 1)\alpha - \gamma} & \text{if}\ j \leq 2, 3 < j^\prime \leq c_v \\[0.1cm]
              \frac{\alpha}{(c_v - 1)\alpha - \gamma} & \text{if}\ 2 < j < j^\prime \leq c_v \\[0.15cm]
              \frac{(c_v - 1 - j)\alpha - \gamma}{(c_v - 1)\alpha - \gamma} & \text{if}\ 2 < j = j^\prime \leq c_v
            \end{cases}
          \end{align*}
      \end{itemize}
    \item Swap $i$ and $\tilde{I}$, delete $\tilde{I}$.
    \item Relabel $\{\tilde{I} + 1, \ldots, n\}$ by $\{\tilde{I}, \ldots, n-1\}$ using the increasing bijection.
    \item Perform an up-step according to the $(\alpha, \gamma)$-growth process.
  \end{enumerate}
\end{defi}
Formal definitions of the terminology and operations used in Definition~\ref{def:alphagamma_chain} can be found in Sections~\ref{sec:treegrowthprocesses} and~\ref{sec:downupchains}.
\begin{thm}\label{thm:nonplanarstationarity}
  Fix $\nin$.
  Let ${\left( T_n(m) \right)}_{m \in \N_0}$ denote the $(\alpha, \gamma)$-chain and let $T_n$ denote the $n$th step of the $(\alpha, \gamma)$-growth process.
  Then $T_n(m) \dcon T_n$ as $m \to \infty$, where $\dcon$ denotes convergence in distribution.
\end{thm}
The significance of Theorem~\ref{thm:nonplanarstationarity} is that the unique invariant distribution for the $(\alpha, \gamma)$-chain is exactly the distribution induced by the $(\alpha, \gamma)$-growth process, proving that the down-step does indeed counteract the up-step governing the $(\alpha, \gamma)$-growth process.

We now turn our focus to the investigation of projections of the $(\alpha, \gamma)$-chain.
Specifically, we will project the $(\alpha, \gamma)$-chain on $\T_{[n]}$ to the space of decorated $[k]$-trees of mass $n$, a special case of which was introduced prior to~\eqref{eq:WrightFisher-diffusion}.
We will show that the projected chain is Markovian and characterize its transition kernel without referencing the $(\alpha, \gamma)$-chain.
There are several well-known results on when a function of a Markov chain is again Markovian.
As in~\cite{forman2018projections}, we will utilize both the Kemeny-Snell criterion~\cite{MR0115196} and the intertwining criterion~\cite{MR624684}, and we will need to use them in conjunction with one another by considering an intermediary Markov chain.

In general, a decorated $[k]$-tree of mass $n$ is a $[k]$-tree $\tree[s] \in \T_{[k]}$, called the \textit{tree shape}, where all branch points and edges have an affiliated non-negative integer mass such that any leaf edge has mass at least $1$ and the sum of all masses is $n$.
An $[n]$-tree naturally gives rise to such a decorated tree, by letting the $[k]$-tree be the reduced subtree of the $[n]$-tree spanned by the leaves labelled by $[k]$ where all branch points of degree $2$ are contracted away.
Comparing this reduced subtree to the original $[n]$-tree, the mass associated with an insertable part, i.e.\ an edge or a branch point, now arise by counting all the leaves for which the insertable part is the first insertable part of the reduced $[k]$-tree that is reached on the ancestral line from the leaf.
We denote the surjective map that yields a decorated $[k]$-tree of mass $n$ from an $[n]$-tree by $\pi_{[k]}^{\bullet n} \colon \T_{[n]} \to \T_{[k]}^{\bullet n}$.
We will discuss this in more detail in Section~\ref{sec:MC_decorated}.

In the following, let $K_n$ denote the transition kernel of the $(\alpha, \gamma)$-chain from Definition~\ref{def:alphagamma_chain}.
Define the following natural Markov kernel from the space $\T_{[k]}^{\bullet n}$ of decorated $[k]$-trees of mass $n$ to $\T_{[n]}$:
\begin{align}
  \label{eq:markovkernel_decorated_nonplanar}
  \Pi_k^{\bullet n}(\tree^\bullet, \cdot) = \P \left( T_n \in \cdot \ \middle \vert \ \pi_{[k]}^{\bullet n} \left( T_n \right) = \tree^\bullet \right),
\end{align}
where $T_n$ is the $n$th step of the $(\alpha, \gamma)$-growth process.

In order to introduce a down-up chain on $\T_{[k]}^{\bullet n}$, let us abuse notation so that $K_n$ and $\Pi_k^{\bullet n}$ denote both the kernel and the corresponding Markovian matrix, whilst $\pi_{[k]}^{\bullet n}$ is both the projection map and the matrix $\pi_{[k]}^{\bullet n} = {\left( 1_{\pi_{[k]}^{\bullet n} ( \tree ) = \tree^\bullet} \right)}_{\tree \in \T_{[n]}, \tree^\bullet \in \T_{[k]}^{\bullet n}}$.
\begin{defi}[Decorated $(\alpha, \gamma)$-chain]%
  \label{def:decorated_alphagamma_chain_intro}
  Fix $n > k \in \N$.
  The \textit{decorated $(\alpha, \gamma)$-chain on} $\T_{[k]}^{\bullet n}$, is a Markov chain on $\T_{[k]}^{\bullet n}$ with transition kernel
  \begin{align}
    K_k^{\bullet n} := \Pi_k^{\bullet n} K_n \pi_{[k]}^{\bullet n}.
    \label{eq:transitionkernel_decorated_intro}
  \end{align}
\end{defi}
The decorated $(\alpha, \gamma)$-chain on $\T_{[k]}^{\bullet n}$ can be described in an autonomous way without referencing the transition kernel of the $(\alpha, \gamma)$-chain.
This is done in Proposition~\ref{prop:decorated_transition_kernel}.
By combining the Kemeny-Snell and intertwining criteria mentioned earlier we obtain another main result of this paper:
\begin{thm}\label{thm:projection_decorated_chain}
  Fix $n \geq 2$.
  Let ${\left( T_n(m) \right)}_{m \in \N_0}$ denote an $(\alpha, \gamma)$-chain.
  For $1 \leq k < n$, let $T_{[k]}^{\bullet n}(m) := \pi_{[k]}^{\bullet n} \left( T_n(m) \right)$, $m \in \N_0$, denote the projection of the chain onto decorated $[k]$-trees of mass $n$.
  If $T_n(0) \sim \Pi_n (\tree^\bullet, \cdot)$ for some $\tree^\bullet \in \T_{[k]}^{\bullet n}$, then ${\left( T_{[k]}^{\bullet n}( m ) \right)}_{m \in \N_0}$ is a decorated $(\alpha, \gamma)$-chain on $\T_{[k]}^{\bullet n}$ started from $\tree^\bullet$.
\end{thm}
Theorem~\ref{thm:projection_decorated_chain} makes it possible to vary $k$ and obtain a consistent system of projected Markov chains all running in stationarity, provided that the $(\alpha, \gamma)$-chain itself is running in stationarity.
\begin{cor}[Consistency in stationarity]%
  \label{cor:consistency_stationarity}
  Fix $n \geq 2$.
  Let ${\left( T_n(m) \right)}_{m \in \N_0}$ denote an $(\alpha, \gamma)$-chain running in stationarity.
  For $1 \leq k < n$, let $T_{[k]}^{\bullet n}(m) := \pi_{[k]}^{\bullet n} \left( T_n(m) \right)$, $m \in \N_0$, denote the projection of the chain onto decorated $[k]$-trees of mass $n$.
  Then each of the ${\left( T_{[k]}^{\bullet n}(m) \right)}_{m \in \N_0}$ is a decorated $(\alpha ,\gamma)$-chain on $\T_{[k]}^{\bullet n}$ running in stationarity.
\end{cor}

%%%%% FIRST BIT OF MATTHIAS COMMENT
In the context of the $(\alpha, \gamma)$-chain projected to the decorated $[k]$-tree spanned by the leaves labelled by $[k]$, let us for each $n \geq k$ consider a decorated $(\alpha, \gamma)$-chain started from $\tree_n^\bullet \in \T_{[k]}^{\bullet n}$ all with tree shape $\tree[s] \in \T_{[k]}$.
Let $Y_n^x(m)$ denote the proportion of mass associated with the insertable part $x \in \insertablef[s]$ of the decorated tree at time $m$, stopped the first time $Y_n^e (\cdot)$ hits $0$ for any leaf edge $e$, and assume that ${\left( Y_n^x(0) \right)}_{x \in \insertablef[s]}$ converges in distribution as $n \to \infty$.
Then
\begin{align}
  {\left( {\left( Y_n^x ( \lfloor n^2 t \rfloor ) \right)}_{x \in \insertablef[s]} \right)}_{t \geq 0}
  \dcon {\left( {\left( Y^x(t) \right)}_{x \in \insertablef[s]} \right)}_{t \geq 0}
    \quad \text{as} \quad n \to \infty,
  \label{eq:WF_general}
\end{align}
where ${\left( Y^x \right)}_{x \in \insertablef[s]}$ is a Wright-Fisher diffusion with weights ${\left( \tilde{w}_x \right)}_{x \in \insertablef[s]}$, where $\tilde{w}_x = w_x - 1$ whenever $x$ is an external edge, $\tilde{w}_x = w_x$ otherwise, and $w_x$ is defined by~\eqref{eq:alphagamma_weights}.
This follows exactly as in~\cite{RefWorks:doc:5b4cbb92e4b0bc982fe42f3a}.

%%%%% SECOND BIT OF MATTHIAS COMMENT
Let $0 < \gamma \leq \alpha < 1$.
Let $\T_n^\circ$ denote the set of rooted multifurcating trees with $n$ unlabelled leaves and consider the natural projection from $\T_{[n]}$ to $\T_n^\circ$.
Then the $(\alpha, \gamma)$-chain satisfies the Kemeny-Snell criterion, and the projected $\T_n^\circ$-valued Markov chain evolves by simply (i) selecting a uniform leaf, (ii) deleting it, and (iii) using the growth rule of the unlabelled $(\alpha, \gamma)$-Markov branching model~\cite{RefWorks:doc:5b4cbb5fe4b02dc0c79270af} to insert a new leaf.
\fxnote{Check enumerations.}
The stationarity of this can also be seen directly from the sampling consistency of this Markov branching model.
In Corollary 3 of~\cite{RefWorks:doc:5b4cbb5fe4b02dc0c79270af} it was shown that, when suitably represented as discrete $\R$-trees with edge lengths $n^{-\gamma}$, these stationary distributions converges weakly, in the Gromov-Hausdorff topology to the distribution of a continuum random tree $\mathcal{T}^{\alpha, \gamma}$.
By Theorem 11 of~\cite{haas2012} this can be strengthened to Gromov-Hausdorff-Prokhorov convergence of weighted $\R$-trees.
In the light of~\eqref{eq:WF_general} we make the following conjecture:
\begin{conjecture*}
  The unlabelled $(\alpha, \gamma)$-chain with $n^2$ steps per time unit and edges scaled by $n^\gamma$ has a scaling limit that is a path-continuous Markov process in the Gromov-Hausdorff-Prokhorov space of weighted $\R$-trees, whose stationary distribution is that of $\mathcal{T}^{\alpha, \gamma}$.
\end{conjecture*}
In the binary case where $\gamma = \alpha$, this conjecture strengthens results by L\"{o}hr et al.~\cite{RefWorks:doc:5b7200a0e4b06af4d4ef714a}, in the case $\alpha = \frac{1}{2}$, and by Nussbaumer and Winter~\cite{nussbaumer2020algebraic}, who use a state space of algebraic trees with a weaker topology that, in some sense, disregards the edge lengths. 

The main new technique of this paper is the \textit{semi-planar $(\alpha, \gamma)$-growth process} and a down-up chain on what we will call \textit{semi-planar trees}, which will be two auxiliary processes that we use to derive the results stated above.
Crucially, both the $(\alpha, \gamma)$-growth process and the $(\alpha, \gamma)$-chain arise as projections of their semi-planar counterpart. 
We define the space of semi-planar trees to be trees where branch points are equipped with a left-to-right ordering of its subtrees (as a planar tree), such that we cannot distinguish the order of the two leftmost subtrees, which we define to be the two subtrees of the branch point with the smallest leaf labels.
The processes we define on the space of semi-planar trees are furthermore of mathematical interest in their own right.
The semi-planar $(\alpha, \gamma)$-growth process is defined similarly to the $(\alpha, \gamma)$-growth process, but the weight $(c-1)\alpha - \gamma$ in a branch point with $c$ children is split up so that there is weight $\alpha$ to the left of any subtree not leftmost, and weight $\alpha - \gamma$ to the right of the rightmost subtree.
$(\alpha, \gamma)$-growth processes are known to have intimate links to random partitions induced by Chinese restaurant processes~\cite{MR883646, RefWorks:doc:5b644f76e4b0a3434e490fc8} and ordered Chinese restaurant processes~\cite{RefWorks:doc:5b6c5580e4b0a3935d3436d8}, see e.g.~\cite{MR2606082, RefWorks:doc:5b4cbb5fe4b02dc0c79270af, RefWorks:doc:5b6c5580e4b0a3935d3436d8}, and the semi-planar trees provide exactly the right structure to further hook onto this link efficiently.

Compared to a non-planar tree, the additional structure will allow us to define a local search (as in step (ii) of the $\alpha$-chain) that coincides with the local search from the $\alpha$-chain in the cases where the leaf $i$ selected for deletion in step (i) is located in a binary branch point.
The key aspect of the down-step is that projecting the $(\alpha, \gamma)$-chain onto the space of (non-planar) $[n]$-trees using the natural surjective projection yields exactly the $(\alpha, \gamma)$-chain of Definition~\ref{def:alphagamma_chain}, if the initial distribution is picked according to an appropriate Markov kernel similar to~\eqref{eq:markovkernel_decorated_nonplanar}.

The paper is structured as follows:

Section 2 contains a brief overview of notation and structures relied on throughout the exposition.
Most notable are urn schemes, (un)ordered Chinese Restaurant Processes, basic tree growth processes and a general setup for tree-valued down-up chains.
We will also cover notation for operations on trees such as deletion, insertion and swapping of leaves.
These fundamental notions will be updated to various tree spaces in later sections.

In Section 3 we develop the notion of semi-planar trees, construct and investigate properties of a semi-planar $(\alpha, \gamma)$-growth process, and in Section 4 we go on to defining the semi-planar down-up chain and prove that it has a suitable stationary distribution.

The projection of the semi-planar down-up chain onto non-planar trees is carried out in Section 5, and Theorem~\ref{thm:nonplanar_alphagamma_chain} yields that this is the $(\alpha, \gamma)$-chain of Definition~\ref{def:alphagamma_chain}.
This section outlines the methodology that will be used in later sections.

In Section 6, an in-depth description of the behaviour of the semi-planar down-up chain projected to the space of decorated trees is provided.

Section 7 provides a selection of key results characterizing the link between projections of the semi-planar down-up chain and the semi-planar down-up chain, and Section 8 pulls together the threads from Section 5, 6, and 7, and proves Theorem~\ref{thm:projection_decorated_chain}. \\

\textbf{Acknowledgements.}
This research is funded by the Lundbeck Foundation (R263-2017-3677) and the Engineering and Physical Sciences Research Council (1930434).
%
%%%%%%%%%%%%%%%%%%%%%%%%%%%%%%%%%%%%%%%%%%%%%%%%%%%%%%%%%%%%%%%%%%%%%%%%%%%%%%%%%%%%
%%%%%%%%%%%%%%%%%%%%%%%%%%%%%%%%%% PRELIMS %%%%%%%%%%%%%%%%%%%%%%%%%%%%%%%%%%%%%%%%%
%%%%%%%%%%%%%%%%%%%%%%%%%%%%%%%%%%%%%%%%%%%%%%%%%%%%%%%%%%%%%%%%%%%%%%%%%%%%%%%%%%%%

\section{Preliminaries}\label{sec:prelims}
\subsection{Urn schemes}
Urn schemes are intimately connected with Chinese Restaurant Processes and other exchangeable processes, and this connection will play a crucial role throughout this exposition.
We will here give a brief account of properties of the generalized \textit{P\'{o}lya urn}.
\begin{defi}[Generalized P\'{o}lya urn]
  Fix $k \in \N$ and $\mathbf{w} = {\left(w_j\right)}_{j \in [k]}$ where $w_j \geq 0$ for each $j \in [k]$ and $\sum_{j=1}^k w_j > 0$. 
  Let ${(X_n)}_\nin$ be a sequence of random variables taking values in $[k]$ such that
  \begin{align}
    \forall j \in [k] \colon \quad \P \left( X_{n+1} = j \ \middle \vert \ N_1^n = n_1, \ldots, N_k^n = n_k \right) = \frac{w_j + n_j}{\sum_{i \in [k]} w_i + n},
    \label{eq:polyaurn}
  \end{align}
  where $N_j^n = \sum_{j^\prime = 1}^n 1_{\{X_{j^\prime} = j\}}$ for each $j \in [k]$ and $n \in \N_0$.
  We call ${(X_n)}_\nin$ a \textit{generalized P\'{o}lya urn with initial weights $\mathbf{w}$}.
  \label{def:polyaurn}
\end{defi}
More loosely, we will use the term \textit{urn scheme} for any countable sequence of random variables governed by~\eqref{eq:polyaurn}.
Before summarizing some of the properties of the above urn scheme, let us briefly recall a few distributions related to urn schemes:
\begin{defi}[Dirichlet-multinomial distribution]
  Fix $k \in \N$, $\nin$, and $\mathbf{w} = {(w_j)}_{j \in [k]}$ where $w_j > 0$. 
  Let $\mathbf{N} = (N_1, \ldots, N_k)$ be a random variable taking values on $\{(n_1, \ldots, n_k) \in \N_0^k \mid \sum_{j=1}^k n_j = n\}$.
  If
  \begin{align}
    \P \left( \mathbf{N} = (n_1, \ldots, n_k) \right)
    = \frac{n!\ \Gamma \left( \sum_{j=1}^k w_j \right)}{\Gamma \left( n + \sum_{j=1}^k w_j \right)} \prod_{j = 1}^k \frac{\Gamma (n_j + w_j)}{n_j!\ \Gamma ( w_j)}
    \label{eq:dirmult}
  \end{align}
  we say that $\mathbf{N}$ has a \textit{Dirichlet-multinomial distribution with $n$ trials and weights $\mathbf{w}$}, and denote this by $\mathbf{N} \sim \dirmult^n(\mathbf{w})$.
  For $k = 2$, we will write $N_1 \sim \dirmult^n(w_1,w_2)$, since $N_2 = n - N_1$.
  This distribution is also known as the \textit{Beta-binomial distribution}, denoted by $\betabin^n(w_1, w_2)$.
\end{defi}
\noindent
We now summarize a couple of properties which will be of value later:
\begin{prop}
  Fix $k \in \N$ and ${\mathbf{w} = (w_1, \ldots, w_k)}$ where $w_j > 0$ for each $j \in [k]$, and let $X = {(X_n)}_\nin$ be a generalized P\'{o}lya urn with weights $\mathbf{w}$.
  Furthermore, for each $j \in [k]$ and $\nin$ let $N_j^n$ be as in Definition~\ref{def:polyaurn}.
  Then
  \begin{enumerate}
    \item $X$ is exchangeable,
    \item $(N_1^n, \ldots, N_k^n) \sim \dirmult^n(\mathbf{w})$,
    \item $N_j^n \sim \dirmult^n \left( w_j, \sum_{i \neq j} w_i \right) = \betabin^n \left( w_j, \sum_{i \neq j} w_i \right)$ for each $j \in [k]$,
    \item\label{prop:polyaurn_conditioning} $\P \left( X_1 = j \ \middle \vert \ N_1^n = n_1, \ldots, N_k^n = n_k \right) = \frac{n_j}{n}$ for any $j \in [k]$.
  \end{enumerate}
  \label{prop:polyaurn}
\end{prop}
\subsection{Chinese Restaurant Processes}\label{sec:chineserest}
We briefly recall the concept of Chinese Restaurant Processes.
\begin{defi}
  Fix $\alpha$ and $\theta$ such that either $0 \leq \alpha \leq 1$ and $\theta > -\alpha$, or $\alpha < 0$ and $\theta = l_\text{max}\alpha$ for some $l_\text{max} \in \N$.
  The seating plan for the customers in a \textit{Chinese Restaurant Process} with parameters $(\alpha, \theta)$, an $(\alpha,\theta)$-$\crp$, is as follows:
  \begin{itemize}
    \item The first customer sits at the first table.
    \item If the first $n$ customers are seated at $L$ tables (where $L \leq l_\text{max}$ in the case where $\alpha < 0$ and $\theta = l_\text{max} \alpha$) with $N_1, \ldots, N_L$ customers seated at each table, customer $n+1$ will sit at
      \begin{itemize}
        \item table $l$ with probability $\frac{N_l - \alpha}{n + \theta}$, for each $l \in [L]$, and
        \item at a new table, $L+1$, with probability $\frac{L\alpha + \theta}{n + \theta}$.
      \end{itemize}
  \end{itemize}
  \label{def:unordered_chinese_restaurants}
\end{defi}
The table-seating of the first $n$ customers naturally induces a partition of $[n]$ by saying that $j$ and $j^\prime$ belong to the same block of the partition if and only if customers $j$ and $j^\prime$ are seated at the same table in the Chinese Restaurant.
It is well-known that in an $(\alpha, \theta)$-$\crp$ the number of customers seated at the first table, $N_1$, after having seated $n$ customers in total, satisfies that $N_1 \sim \betabin^n(1-\alpha, \alpha + \theta)$.
For a more extensive account of this process, see~\cite{RefWorks:doc:5b644f76e4b0a3434e490fc8}. \\

\noindent
We can define an \textit{ordered Chinese Restaurant Process} with parameters $(\alpha, \theta)$ for $0 \leq \alpha \leq 1$ and $\theta \geq 0$, denoted by $\ocrp(\alpha, \theta)$, in a similar fashion.
The seating dynamic is the same but we now associate a left-right ordering to the tables.
For any number of tables $L \in \N$, conditional on customer $n+1$ sitting at a new table, $L+1$, place this table
\begin{itemize}
  \item to the left of the leftmost table with probability $\frac{\alpha}{L\alpha + \theta}$,
  \item between any pair of neighbouring tables with probability $\frac{\alpha}{L\alpha + \theta}$, and
  \item to the right of the rightmost table with probability $\frac{\theta}{L\alpha + \theta}$.
\end{itemize}
This construction not only gives rise to a random partition of $[n]$ but also a random permutation of $[L]$, $\Sigma$, corresponding to the ordering of the $L$ tables seating $n$ customers.
It holds that
\begin{align}
  \P \left( \Sigma = \sigma \right)
  = {\left( \frac{\theta}{\alpha} \right)}^{R(\sigma)} \frac{\Gamma\left( \frac{\theta}{\alpha} \right)}{\Gamma\left( \frac{\theta}{\alpha} + L \right)}
  \label{eq:ocrp_tableorder}
\end{align}
for every permutation $\sigma$ of $[L]$, where
\begin{align*}
  R(\sigma) = \# \left\{ j \in [L] \ \middle \vert \ \sigma(j) > \sigma(j^\prime)\ \text{for all}\ 1 \leq j^\prime < j \right\}
\end{align*}
is the number of record values in $\sigma$.
Whenever we wish to refer to the distribution of a random partition induced by an $(\alpha, \theta)$-$\crp$ or $(\alpha, \theta)$-$\ocrp$ with $n$ customers, we will denote these distributions by $\crp^n(\alpha, \theta)$ and $\ocrp^n(\alpha, \theta)$, respectively. \\

We have the following result:
\begin{prop}[Proposition 6 of~\cite{RefWorks:doc:5b6c5580e4b0a3935d3436d8}]\label{prop:decrementmatrix}
  For each $\nin$, let ${\mathbf{N}}_n = (N_1^n, \ldots, N_k^n)$ denote the vector of the number of customers seated in an $(\alpha, \theta)$-$\ocrp$ with $n$ customers, enumerated in the table order of appearance.
  It holds that the corresponding sequence ${({\mathbf{N}}_n)}_\nin$ is regenerative, i.e.\ that given the first part of $\mathbf{N}_n$ is $n_1$, then the remaining parts of $\mathbf{N}_n$ sum to $n - n_1$ and have the same distribution as $\mathbf{N}_{n - n_1}$.
  Furthermore, 
  \begin{align}
    \P \left( \mathbf{N}_n = (n_1, \ldots, n_k) \right)
    = \prod_{i=0}^{k-1} q_{\alpha, \theta} \left(  n - \sum_{i' = 1}^{i} n_{i'},  n_{i + 1} \right) 
    \label{eq:decrementmatrixdist}
  \end{align}
  where the \textit{decrement matrix}, $q_{\alpha, \theta}$, is given by
  \begin{align}
    q_{\alpha, \theta}(n,m) =
    \begin{pmatrix}
      n \\ m
    \end{pmatrix}
    \frac{(n - m)\alpha + m\theta}{n} \frac{\Gamma(m-\alpha)}{\Gamma(1-\alpha)} \frac{\Gamma(n-m+\theta)}{\Gamma(n+\theta)}.
    \label{eq:decrementmatrix}
  \end{align}
\end{prop}
\noindent
For more information regarding $(\alpha, \theta)$-$\ocrp$ and their role in regenerative tree growth processes, see~\cite{RefWorks:doc:5b6c5580e4b0a3935d3436d8}.
\subsection{Trees}
Recall that $\T_A$ denotes the space of (non-planar) rooted, multifurcating trees with leaves labelled by the elements of $A$, and call an element of $\T_A$ an $A$-tree.
For each $j \in A$ we will refer to the leaf labelled by $j$ as \textit{leaf $j$}, whilst denoting the root vertex by $\emptyset$.
Let $\edge(\tree)$ and $\vertices(\tree)$ denote the set of edges and vertices of $\tree \in \T_A$, respectively.
\begin{nota} Fix a finite set $A$ and $\tree \in \T_A$.
  \begin{enumerate}
    \item Let $e \in \edge(\tree)$ be an edge between $u, v \in \vertices(\tree)$.
      If $u$ is closer to the root, $\emptyset$, than $v$, in the graph distance, we say that $u$ is the \textit{parent} of $v$, and that $v$ is the \textit{child} of $u$.
      For brevity, we will sometimes write $u \prec_{\tree} v$ if $u$ is the parent of $v$ in $\tree$.
      We will write $e_{uv}$ if there is any ambiguity about which edge we are discussing, and in cases where we do not wish to assign a name to the parent of $v$, we will simply denote the edge between $v$ and its parent by $e_{v \downarrow}$.
      For the unique edge connected to root we will write $e_{\emptyset \uparrow}$, and, to ease notation, we shall simply refer to a leaf edge by the label of the associated leaf.
    \item If $v_1, v_2 \in \vertices(\tree)$ have the same parent, we say that they are \textit{siblings}.
%    \item \label{nota:edgelabels} Associate to each $e \in \edge(\tree)$ the set
%%
%      \begin{align*}
%        B_e = \left\{ j \in A \ \middle \vert \e\ \text{is included in the unique path in $\tree$ from $j$ to $\emptyset$} \right\},
%      \end{align*}
%%
%      and take $B_e \in \edge(\tree)$ to mean the unique edge $e \in \edge(\tree)$ with $B_e = B$.
%      We will call $B_e$ the \textit{label} of $e$.
%    \item \label{nota:vertexlabels} Associate to each vertex $\emptyset \neq v \in \vertices(\tree)$, with children $w_1, \ldots, w_m \in \vertices(\tree)$, the unique set $B_v = \cup_{i=1}^m B_i$ such that $B_i \in \edge(\tree)$ is the edge between $v$ and $w_i$ for each $i \in [m]$.
%    \item For any $B \subseteq A$ we define
%      %
%      \begin{align*}
%        \siblings(B) = \left\{ C \subseteq A \ \middle \vert \C\ \text{is the label of a sibling of}\ B \right\} 
%      \end{align*}
%      %
%      to be the collection of siblings of $B$.
    \item We say that a vertex that does not have any children is a \textit{leaf}, and will refer to an edge connected to a leaf as a \textit{leaf edge} or an \textit{external edge}.
      It is worth noting that the root is not a leaf, and importantly, we do not consider the edge $e_{\emptyset \uparrow}$ external.
      We will refer to any vertex apart from the root that is not a leaf as a \textit{branch point}, and refer to the collection of these by $\branchpoints (\tree)$.
      To ease notation later on, we will define \textit{insertable parts} of the tree as $\insertablef[t] := \branchpoints(\tree) \cup \edge(\tree)$.
    \item The \textit{ancestral line} from a vertex $v \in \vertices(\tree)$ is the collection of vertices $v_0, \ldots, v_l$, and edges, $e_1, \ldots, e_l$, such that $v_0 = v$, $v_l = \emptyset$, whilst for all $j \in [l]$ it holds that $v_j$ is the parent of $v_{j-1}$ with $e_{j} = e_{v_{j-1}v_j}$.
  \end{enumerate}
\end{nota}

Now fix $\tree \in \T_A$ and some $i \in A$.
If $v_0, \ldots, v_l$ and ${\left( e_{v_{j-1}v_j} \right)}_{j \in [l]}$ denotes the ancestral line from leaf $i$, this naturally splits $\tree$ into the ancestral line and a collection of \textit{spinal bushes} each consisting of a collection of \textit{subtrees}.
To make this precise, note how for each $j \in [l-1]$ there is a collection of subtrees, ${\left( \tree_{j,j^\prime} \right)}_{j^\prime \in [c_{v_j} - 1]}$, rooted at $v_j$, such that each of them only contains descendants of $v_j$ that are not also descendants of $v_{j-1}$.
Here $c_{v_j}$ denotes the number of children of $v_j$.
We call this collection of subtrees a \textit{spinal bush}, and will refer to the spinal bush located at $v_j$ as the \textit{$j$th spinal bush on the ancestral line from leaf $i$} for each $j \in [l-1]$.
Note how the ancestral line from $i$ and the collection of the spinal bushes are disjoint apart from the root-vertices of the subtrees in the spinal bushes, and exhaust $\tree$, and are thus referred to as the \textit{spinal decomposition of $\tree$ from $i$}.

\subsection{Tree growth processes}%
\label{sec:treegrowthprocesses}
In order to define tree growth processes we first need to introduce various operations on trees:
\begin{defi}[Insertion of a leaf]
  \label{def:insertion_nonplanar}
  Let $A \subset \N$ be a finite set, fix $\tree[s] \in \T_A$, $j \in \N \setminus A$, and $x \in \insertablef[s]$.
  The operation of inserting leaf $j$ to $x$ in $\tree[s]$ is the map $\left( \tree[s], x, j \right) \overset{\varphi}{\mapsto} \tree \in \T_{A \cup \{j\}}$ where $\tree$ is defined as follows.
  \begin{itemize}
    \item If $x = e_{vw} \in \edge(\tree[s])$, $\tree$ is obtained by adding new vertices $v'$ and $j$ to  $\vertices ( \tree[s] )$, and replacing the edge $e_{vw}$ with the three edges $e_{vv'}, e_{v'w},$ and $e_{jv'}$.
    \item If $x = v \in \branchpoints( \tree[s] )$, $\tree$ is obtained by adding a new vertex, $j$, to $\vertices ( \tree[s] )$, and an edge, $e_{jv}$, to $\edge ( \tree[s] )$.
  \end{itemize}
\end{defi}
\begin{defi}[Deletion of a leaf]
  \label{def:leaf_deletion}
  Let $A$ be a finite set and fix $\tree \in \T_A$.
  The operation of deleting leaf $j \in A$ from $\tree$ is the map
  \begin{align*}
    \left( \tree, j \right) \overset{\mathring{\rho}}{\mapsto} \tree[s] \in \T_{A \setminus \{j\}}
  \end{align*}
  where $\tree[s]$ is obtained in the following way.
  Let $v$ be the parent of $j$ and let $c_v$ be the number of children of $v$.
  If $c_v \geq 3$, $\tree[s]$ is obtained by removing $j$ from $\vertices ( \tree )$ and $e_{jv}$ from $\edge (\tree)$.
  If $c_v = 2$, let $w$ be the sibling of $j$, $v'$ the parent of $v$, and obtain $\tree[s]$ by removing $j$ and $v$ from $\vertices (\tree)$ and $e_{wv}, e_{vv'}$, and $e_{jv}$ from $\edge (\tree)$ and subsequently adding $e_{wv'}$ to the latter.
  We denote by $\rho$ the map that, after applying $\mathring{\rho}$, subsequently relabels $A \setminus \{j\}$ by $[\# A - 1]$ using the increasing bijection.
\end{defi}
\fxnote{Make figure showing deletion and insertion.}
If we wish to delete all leaves of $\tree \in \T_A$ labelled by the elements of $B \subsetneq A$, we will denote this by $\mathring{\rho} (\tree, B)$ and $\rho (\tree, B)$, respectively, obtaining the resulting tree by successively applying the above definition to the elements of $B$.

Throughout this exposition we will need the following notion from~\cite{RefWorks:doc:5b6c5580e4b0a3935d3436d8}:
\begin{defi}[Tree growth process]\label{def:growthprocess}
  A Markov chain, ${\left( T_n \right)}_\nin$, is a \textit{tree growth process} if $T_n$ takes values in $\T_{[n]}$ and $\rho \left( T_{n+1}, n+1 \right) = T_n$ for each $\nin$. 
\end{defi}
Equivalently, one could define a tree growth process in terms of a sequence of Markov kernels, $G = {\left( G_n \right)}_\nin$, where $G_n$ is a Markov kernel from $\T_{[n]}$ to $\T_{n+1}$.
We will refer to $G$ as a \textit{growth rule}.
The main tree growth process of interest to us is the following:
\begin{defi}[$(\alpha, \gamma)$-growth process]\label{def:growthprocess_alphagamma}
  Fix $0 \leq \gamma \leq \alpha \leq 1$.
  Let $T_1$ be the unique element of $\T_1$.
  For any $\nin$ construct $T_{n+1}$ conditional on $T_n = \tree[s] \in \T_{[n]}$, by, for each $x \in \insertablef[s]$, setting
  \begin{align*}
    w_x
    = 
    \begin{cases}
      1 - \alpha & \text{if $x \in \edge \left( \tree[s] \right)$ is external,} \\
      \gamma & \text{if $x \in \edge \left( \tree[s] \right)$ is internal,} \\
      \left( c_x - 1 \right) \alpha - \gamma & \text{if $x \in \branchpoints (\tree[s])$ with $c_x$ children,} \\
    \end{cases}
  \end{align*}
  and then defining the conditional distribution of $T_{n+1}$ by
  \begin{align*}
    \P \left( T_{n+1} = \varphi \left( \tree[s], x, n+1 \right) \ \middle \vert \ T_n = \tree[s] \right)
    = \frac{w_x}{n - \alpha}. 
  \end{align*}
  The sequence ${\left( T_n \right)}_\nin$ is referred to as the $(\alpha, \gamma)$-\textit{growth process}.
\end{defi}
The $(\alpha, \gamma)$-growth process was first introduced in~\cite{RefWorks:doc:5b4cbb5fe4b02dc0c79270af}, and encompasses several other well-known growth processes.
For $\gamma = 1-\alpha$ and $\frac{1}{2} \leq \alpha < 1$ the above growth process is Marchal's stable tree growth~\cite{RefWorks:doc:5b6c561fe4b06c0731a5c558}.
For $0 \leq \gamma = \alpha \leq 1$ it is identical to Ford's $\alpha$-model~\cite{RefWorks:doc:5b76ce32e4b0820c421f301d}.
And, as mentioned previously, a special case of Ford's $\alpha$-model is R\'{e}my's uniform binary tree growth process~\cite{RefWorks:doc:5b71b380e4b06c0731a629f4}, which arises from choosing $\gamma = \alpha = \frac{1}{2}$.
One property of the latter growth process is that $T_n$ will be uniformly distributed on $\T_{[n]}^{\bin}$, the space of binary trees with $n$ leaves~\cite{RefWorks:doc:5b4cbc14e4b04428cc72cf41}, for each $\nin$.
We will refer to the growth rule governing the $(\alpha, \gamma)$-growth process as the $(\alpha, \gamma)$-growth rule.

\subsection{Down-up chains}%
\label{sec:downupchains}
Let us now introduce a general notion of down-up chain for trees that contains not only the Aldous down-up chain~\cite{RefWorks:doc:5b4cbc14e4b04428cc72cf41} but also the down-up chain proposed in~\cite{RefWorks:doc:5b4cbc93e4b07f5746e47014}.
The definition of down-up chains given here is in accordance with down-up and up-down chains on more general state spaces studied in detail in~\cite{MR2480792,MR2095623,MR2480787,MR2596654}.
\begin{defi}[Down-up chains]\label{downupchain_general}
  Fix $n \in \N$, $\tree \in \T_{[n]}$, some growth rule, $G$, and a \textit{transformation function} $f \colon \T_{[n]} \times [n] \to \T_{[n]} \times [n]$.
  The $(f, G)$-\textit{down-up chain, ${\left( T_n(m) \right)}_{m \in \N_0}$,  started at $\tree$} is a Markov chain on $\T_{[n]}$ with $T_n(0) = \tree$ and with the following transition kernel.
  For each $m \in \N_0$ obtain $T_n(m+1)$ from $T_n(m)$ by
  \begin{enumerate}
    \item\label{transformation_down} applying the transformation function, $f$, to $T_n(m)$ and a random variable $I \sim \Unif\left( [n] \right)$, thereby defining
      \begin{align*}
        \left( T_n'(m), \tilde{I} \right)
        = f \left( T_n(m), I \right),
      \end{align*}
    \item\label{delete_relabel} deleting the leaf labelled $\tilde{I}$ from $T_n'(m)$, and relabel $\{\tilde{I} + 1, \ldots, n\}$ by $\{\tilde{I}, \ldots, n-1\}$ using the increasing bijection, i.e.\ defining
      \begin{align*}
        T_n^\downarrow(m)
        = \rho\left( T_n'(m), \tilde{I} \right) \in \T_{[n-1]},
      \end{align*}
    \item\label{insertion} and inserting leaf $n$ to $T_n^{\downarrow}(m)$ according to $G_{n-1}$ to obtain $T_n(m+1)$.
  \end{enumerate}
  % OLD VERSION OF NOTATION
  % 17-06-2020
  %
  %\begin{enumerate}
  %  \item\label{transformation_down} applying the transformation function, $f$, on $T_n(m)$ and a random variable $U \sim \Unif\left( [n] \right)$, thereby defining
  %    \begin{align*}
  %      \left( T_n'(m), U' \right)
  %      = f \left( T_n(m), U \right),
  %    \end{align*}
  %  \item\label{step1down} deleting the leaf labelled $U'$ from $T_n'(m)$, i.e.\ defining
  %    %
  %    \begin{align*}
  %      T_n^\downarrow(m)
  %      = \rho\left( T_n'(m), U' \right),
  %    \end{align*}
  %    %
  %  \item\label{relabel} relabel $T_n^\downarrow(m)$ by the unique increasing bijection $[n] \setminus \{U'\} \to [n-1]$ to obtain $T_n^{\downarrow, \text{re}}(m)$.
  %  \item\label{insertion} inserting leaf $n$ to $T_n^{\downarrow, \text{re}}(m)$ according to $G$ to obtain $T_n(m+1)$.
  %\end{enumerate}
\end{defi}
We will only work with the above type of down-up chains, but there are examples (see e.g.~\cite{RefWorks:doc:5b4cbc93e4b07f5746e47014}) of down-up chains where there is no relabelling.
In that case, $\rho$ is replaced by $\mathring{\rho}$ in~\ref{delete_relabel}, and~\ref{insertion} is replaced with a step where the leaf $\tilde{I}$ is inserted directly into $T_n^\downarrow(m)$.
This is particularly useful when the growth rule induces exchangeable leaf labels, which is the case for R\'{e}my's uniform tree growth and Marchal's stable tree growth, but not for other $(\alpha, \gamma)$-growth processes.

Let us focus on the role of the transformation map $f$.
In the above we require $I$ to be uniform on $[n]$.
As in step~\ref{transformation_down}, say that $\left( \tree[s]', \tilde{I} \right) := f \left( \tree[s], I \right)$.
Even though $I$ is uniformly distributed on the leaves of $\tree[s]$, it will be dependent on $f$ whether or not $\tilde{I}$ is uniformly distributed on $\tree[s]'$.
This will allow us to make non-uniform deletions, which will be of vital importance.

In the basic cases, $f$ will simply be the identity function, in which case we will call the chain an $(\id,G)$-down-up chain and then specify $G$.
For our purposes, it will be important to use more complicated function, in which case we will characterize the output of $f$ explicitly in terms of operations on $\tree[s]$ and the value of $I$.
Furthermore, if $G$ is the $(\alpha, \gamma)$-growth rule we will denote the associated chains by $\left(f, (\alpha, \gamma)\right)$-down-up chain, and then specify $f$.

We note that the $(\id,G)$-down-up chain without relabelling, where $G$ is the \textit{uniform growth rule}, corresponds to Aldous' down-up chain~\cite{RefWorks:doc:5b4cbc14e4b04428cc72cf41}.
To prepare a generalization of down-up chains related to the $(\alpha, \gamma)$-growth process we need the following definition:
\begin{defi}[Label swapping]\label{def:labelswapping}
  Let $A$ be a finite set.
  The operation of \textit{swapping labels $i \in A$ and $j \in A$ in $\tree[s] \in \T_{A}$} is the map
  \begin{align*}
    \left( \tree[s], i, j \right) \stackrel{\tau}{\mapsto} \tree \in \T_A
  \end{align*}
  where, if we let $v_i$ and $v_j$ denote the parents of $i$ and $j$, respectively, $\tree$ is obtained from $\tree[s]$ by replacing $e_{iv_i}$ and $e_{jv_j}$ with $e_{iv_j}$ and $e_{jv_i}$, respectively, in $\edge (\tree[s])$.
\end{defi}
Let us discuss two examples in the binary setting studied in~\cite{forman2018projections} and~\cite{RefWorks:doc:5b4cbc93e4b07f5746e47014}.
To this end, let $\tree[s] \in \T_{[n]}^{\bin}$, and define a transformation function, $f$, by setting $f(\tree[s], i) = \left(\tau\left( \tree[s], i, \tildei \right), \tildei \right)$, where
\begin{align*}
  &a \colon = \min \left\{ \text{leaves in 1st subtree on the ancestral line from}\ i \ \text{in}\ \tree[s] \right\}, \nonumber \\
  &b \colon = \min \left\{ \text{leaves in 2nd subtree on the ancestral line from}\ i \ \text{in}\ \tree[s] \right\},
\end{align*}
and $\tildei \colon = \max\left\{ i, a, b \right\}$.
Then the $(f, G)$-down-up chain where $G$ is the uniform growth rule is exactly the \textit{uniform chain} on $\T_{[n]}^\text{bin}$ without relabelling described in~\cite{forman2018projections}. 
If instead $G$ is the $\alpha$-growth rule, then the $(f, G)$-down-up chain with relabelling is the \textit{$\alpha$-chain} of~\cite{RefWorks:doc:5b4cbc93e4b07f5746e47014}.
\begin{thm}[Theorem 1 of~\cite{RefWorks:doc:5b4cbc93e4b07f5746e47014}]
  Let ${\left( T_n(m) \right)}_{m \in \N_0}$ denote the uniform chain on $\T_{[n]}^{\bin}$ and let $T_n$ denote the $n$th step of the uniform growth process.
  Then the unique invariant distribution of the uniform chain is that of $T_n$, i.e.\ the uniform distribution on $\T_{[n]}^{\bin}$.
\end{thm}
%
%%%%%%%%%%%%%%%%%%%%%%%%%%%%%%%%%%%%%%%%%%%%%%%%%%%%%%%%%%%%%%%%%%%%%%%%%%%%%%%%%%%%
%%%%%%%%%%%%%%%%%%%%%%%%%%%% Planar growth process %%%%%%%%%%%%%%%%%%%%%%%%%%%%%%%%%
%%%%%%%%%%%%%%%%%%%%%%%%%%%%%%%%%%%%%%%%%%%%%%%%%%%%%%%%%%%%%%%%%%%%%%%%%%%%%%%%%%%%
%
\section{A semi-planar growth process}
\label{sec:semiplanar_alphagamma}
In order to introduce an appropriate down-step in the multifurcating case we start by introducing a refined version of the $(\alpha, \gamma)$-growth process.
In the following we will, for any $c \in \N$, let $\SS_{[c]}$ denote the space of permutations of $[c]$.
We will require the existence of an element of $\SS_\emptyset$, and we will use the convention that this element is denoted by $0$.
Additionally, we will use the notation that $\SS_{[c]} \ni \sigma = {(l_j)}_{j \in [c]}$ if $\sigma(j) = l_j$ for each $j \in [c]$.
Recall one of the many ways of defining a planar tree:
\begin{defi}[Planar Trees]
  Let $A$ be a finite set and fix $\tree \in \T_A$.
  Let $\treesigma^* = {\left( \sigma_v^* \right)}_{v \in \branchpoints(\tree)}$ be a collection of permutations, such that for each $v \in \branchpoints (\tree)$, $\sigma_v^* \in \SS_{[c_v]}$, where $c_v$ is the number of children of $v$.
  We call $\that^* = (\tree, \treesigma^*)$ a \textit{planar tree} with \textit{tree shape} $\tree$, and denote by $\Thatspace_A^*$ the space of planar trees with leaves labelled by $A$.
\end{defi}
Let us introduce some vocabulary for planar trees that will enable us to work efficiently with these objects.
Let $\hat{\tree}^* = \left( \tree, \treesigma^* \right) \in \Thatspace_A^*$ be given, fix $v \in \branchpoints(\tree)$, and let $\that_1^{*v}, \ldots, \that_{c_v}^{*v}$ denote the planar subtrees rooted at $v$, enumerated in the order of least elements of the associated sets of leaf labels.
We say that $\that_i^{*v}$ is to the \textit{right} of $\that_j^{*v}$ if $\sigma_v^*(i) > \sigma_v^*(j)$ for any $1 \leq i \neq j \leq c_v$, and use the term \textit{right neighbour} if $\sigma_v^*(i) = \sigma_v^*(j) + 1$.
From the above it is also clear what we mean by the terms \textit{rightmost}, \textit{left} and \textit{leftmost} which are defined in the obvious fashion. 

Any planar version of the $(\alpha, \gamma)$-growth process will have to have a total weight of $\alpha - \gamma$ in a branch point with two children, and then iteratively adding a weight of $\alpha$ for every subsequent leaf being inserted into that branch point.
Hence, in a branch point with $c$ children there are only $c-1$ natural weights to use.
One way to incorporate this dynamic, is to designate the original two subtrees of the branch point a special role, and only allow insertions to the right of those subtrees.
These two subtrees are, by definition of a growth process, easily identifiable as the two subtrees with the lowest leaf labels.
If we insist that these two subtrees are located leftmost in the branch point, we can limit our attention to the following state space:
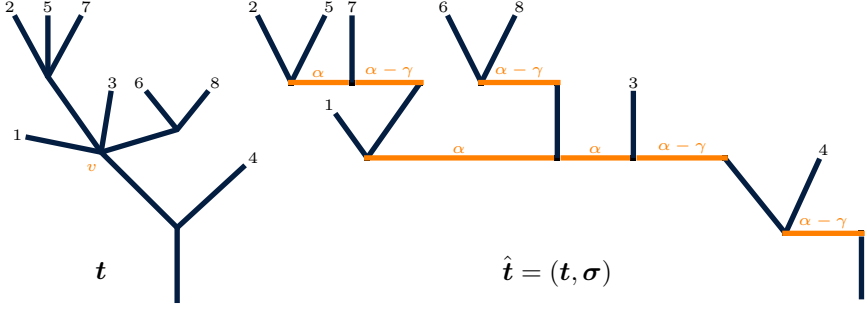
\begin{figure}[t]
  \centering
  \begin{tikzpicture}[
    bp/.style={rectangle, fill, inner sep=0pt, minimum width=2pt, minimum height=2pt},
    comment/.style={circle, draw, color=orange},
    leaf/.style={inner sep=1pt},
    weights/.style={auto, swap, orange, inner sep = 1pt},
    label/.style={text=oxfordblue},
    line/.style={-, line width = 2pt, color = oxfordblue},
    arrowline/.style={->, line width = 1pt, color = oxfordblue, >=stealth},
    scale=1
  ]
  \tikzstyle{every node}=[font=\tiny]

  %
  % Tree shape
  %

  \begin{scope}[xscale = -1]
    \node (0) at (0,0) {};
    \node (bp) at (0,1) {};
    \node[leaf] (4) at (-1,2) {$4$};
    \node (v) at (1,2) {};

  % 1 subtree (from left)
    \node (bp2) at (0,2.3) {};
    \node[leaf] (8) at (-0.5, 3) {$8$};
    \node[leaf] (6) at (0.5, 3) {$6$};

  % 2 subtree (from left)
    \node[leaf] (3) at (0.85,3) {$3$};

  % 3 subtree (from left)
    \node (bp3) at (1.7,3) {};
    \node[leaf] (7) at (1.2, 4) {$7$};
    \node[leaf] (5) at (1.7, 4) {$5$};
    \node[leaf] (2) at (2.2, 4) {$2$};

  % 4 subtree (from left)
    \node[leaf] (1) at (2.1,2.3) {$1$};

    \draw[line] (0.mid) to (bp.mid) to (4);
    \draw[line] (bp.mid) to (v.mid) to (bp2.mid) to (8);
    \draw[line] (bp2.mid) to (6);
    \draw[line] (v.mid) to (3);
    \draw[line] (v.mid) to (bp3.mid) to (7);
    \draw[line] (bp3.mid) to (5);
    \draw[line] (bp3.mid) to (2);
    \draw[line] (v.mid) to (1);

    \node[orange] at (v.south west) {$v$};

    \node at (1,0.5) {\normalsize $\tree$};
  \end{scope}
  %
  % Semi-planar version
  %

  \begin{scope}[shift = {(9,0)}, xscale = -1]

    \node (0) at (0,0) {};

    \node[bp] (bp_1) at (0,1) {};
    \node[bp] (bp_2) at (1,1) {};
    \node[leaf] (4) at (0.5,2.1) {$4$};

    \node[bp] (v_1) at (1.8,2) {};
    \node[bp] (v_2) at (3,2) {};
    \node[bp] (v_3) at (4,2) {};
    \node[bp] (v_4) at (6.5,2) {};

  % 1 subtree (from left) (v_1)
    \node[bp] (bp2_1) at (4,3) {};
    \node[bp] (bp2_2) at (5,3) {};
    \node[leaf] (8) at (4.5, 4) {$8$};
    \node[leaf] (6) at (5.5, 4) {$6$};

  % 2 subtree (from left) (v_2)
    \node[leaf] (3) at (3,3) {$3$};

  % 3 subtree (from left) (v_3)
    \node[bp] (bp3_1) at (5.8,3) {};
    \node[bp] (bp3_2) at (6.7,3) {};
    \node[bp] (bp3_3) at (7.5,3) {};
    \node[leaf] (7) at (6.7, 4) {$7$};
    \node[leaf] (5) at (7, 4) {$5$};
    \node[leaf] (2) at (8, 4) {$2$};

  % 4 subtree (from left)
    \node[leaf] (1) at (7,2.7) {$1$};

    \draw[line] (0) to (bp_1);
    \draw[line] (bp_2.center) to (4);

    \draw[line] (bp_2.center) to (v_1.center);
    \draw[line] (v_2.center) to (3);
    \draw[line] (bp2_2.center) to (8);
    \draw[line] (bp2_2.center) to (6);

    \draw[line] (v_3.center) to (bp2_1.center);

    \draw[line] (v_4.center) to (bp3_1.center);
    \draw[line] (bp3_2.center) to (7);
    \draw[line] (bp3_3.center) to (5);
    \draw[line] (bp3_3.center) to (2);

    \draw[line] (v_4.center) to (1);

    \draw[line, orange] (bp_1.east) to node[weights] {$\alpha - \gamma$} (bp_2.west);
    \draw[line, orange] (v_1.east) to node[weights] {$\alpha - \gamma$} (v_2) to node[weights] {$\alpha$} (v_3) to node[weights] {$\alpha$} (v_4.west);
    \draw[line, orange] (bp2_1.east) to node[weights] {$\alpha - \gamma$} (bp2_2.west);
    \draw[line, orange] (bp3_1.east) to node[weights] {$\alpha - \gamma$} (bp3_2) to node[weights] {$\alpha$} (bp3_3.west);
    
    \node at (4,0.5) {\normalsize $\that = (\tree, \treesigma)$};
  \end{scope}
\end{tikzpicture}
  \caption{Illustration of a semi-planar tree with the allocation of weights within the branch points.
    The horizontal orange lines are expansions of the branch points illustrating the $c_v - 1$ insertable parts associated with branch point $v$.
    The only non-trivial element of $\treesigma$ is $\sigma_v$ characterized by $\sigma_v(1) = 2$ and $\sigma_v(2) = 1$, and the only other semi-planar tree with tree shape $\tree$, would come from swapping the position of the subtrees labelled by $\{3\}$ and $\{6,8\}$, such that $v$ would have the associated permutation $\sigma_v^\prime = (1,2)$.}
  \label{fig:semiplanar_element}
\end{figure}
\begin{defi}[Semi-planar Tree]\label{def:semiplanartree}
  Let $A$ be a finite set and fix $\tree \in \T_A$.
  Let $\treesigma = {\left( \sigma_v \right)}_{v \in \branchpoints(\tree)}$ be a collection of permutations, such that for each $v \in \branchpoints (\tree)$, $\sigma_v$ is a permutation of $\left[ c_v - 2 \right]$ where $c_v$ is the number of children of $v$.
  We call $\that = (\tree, \treesigma)$ a \textit{semi-planar tree}, and denote by $\Thatspace_A$ the space of semi-planar trees with leaves labelled by $A$.
\end{defi}
For an illustration of an element of $\Thatspace_{[8]}$, see Figure~\ref{fig:semiplanar_element}.
The space of semi-planar trees, $\Thatspace_A$, arises from $\Thatspace_A^*$ as follows:
\begin{enumerate}
  \item consider only the elements ${\that[s]}^* \in \Thatspace_A^*$ such that ${\that[s]}^* = (\tree[s], \treesigma^*)$ where for all $v \in \branchpoints(\tree[s])$ and $l \in \{1,2\}$ we have $\sigma_v^*(l) \in \{1, 2\}$, and
  \item identify $\that_1^* = (\tree, \treesigma^*)$ with $\that_2^* = (\tree, \treesigma^{**})$ if $\sigma_v^*(l) = \sigma_v^{**}(l)$ for all $v \in \branchpoints(\tree)$ and $3 \leq l \leq c_v$, where $c_v$ is the number of children of $v$.
\end{enumerate}
For each element $\that = (\tree, \treesigma) \in \Thatspace_A$ we will systematically pick a representative $\that^* = (\tree, \treesigma^*) \in \Thatspace_A^*$ by setting $\sigma_v^*(1) = 1$, $\sigma_v^*(2) = 2$ and $\sigma_v^*(l) = \sigma_v(l-2)$ for all $3 \leq l \leq c_v$ for each $v \in \branchpoints(\tree)$, where $c_v$ denotes the number of children of $v$.

It is clear that there is a surjective projection map from the space of semi-planar (or planar) trees to the space of trees, $\hat{\pi} \colon \Thatspace_A \to \T_A$ for any $A \subseteq \N$, defined by
\begin{align}\label{projectionmap}
  \hat{\pi} \left( \hat{\tree} \right) = \hat{\pi} \left( \tree, \treesigma \right) = \tree.
\end{align}
We will continue to use all the previous terminology such as edge set, set of branch points, and ancestral path for a semi-planar tree, $\that$, by which we will mean the corresponding versions in $\hat{\pi} ( \that )$. 
Before introducing a semi-planar version of the $(\alpha, \gamma)$-growth process, let us update the notion of inserting a leaf to the semi-planar setting, recalling that $\varphi$ denotes the insertion map of Definition~\ref{def:insertion_nonplanar}.
%
%\begin{defi}[Deletion of leaf in $\Thatspace$]\label{deletionplanar}
%    Let $A$ be a finite set and fix $\that \in \Thatspace_A$.
%    The operation of deleting a leaf, $j \in A$, from $\that$ is a map
%    %
%    \begin{align}
%        \left( \that, j \right)
%        = \left( \tree, \treesigma, j \right)
%        \overset{\hat{\rho}}{\mapsto}
%        \left( \rho \left( \tree, j \right), \treesigma^\rho \right)
%        \in \Thatspace_{A \setminus \{\tildei\}}
%        \label{deletionmapplanar}
%    \end{align}
%    %
%    where $\treesigma^\rho = {\left( \sigma_B^\rho \right)}_{B \in \internal \left( \rho\left( \tree, j \right) \right)}$ is defined by setting 
%    %
%    \begin{align}
%      \sigma_{B \setminus \{j\}}^\rho
%        =
%        \begin{cases}
%          {\left( \sigma_B(i) - 1\left( \sigma_B(i) \geq l' \right) \right)}_{i \neq l'} & \text{if}\ B\ \text{is the parent of $j$ in}\ \tree \\
%        \sigma_B & \text{otherwise} 
%        \end{cases}
%    \end{align}
%    %
%    for each $B \in \branchpoints(\tree)$, where $l' = \max \left\{ l - 2, 1 \right\}$ and $l$ is uniquely determined by $C_l = \{ j \}$ for some $l \in [k]$ where $C_1, \ldots, C_k$ are the children of $B$ enumerated in the order of least element. 
%\end{defi}
%
\begin{defi}[Insertion of a leaf in $\Thatspace_A$]\label{def:insertionplanar}
  Let $A \subset \N$ be a finite set, $\that[s] = \left( \tree[s], \treesigma \right) \in \Thatspace_A$, $j \in \N$, $x \in \insertablef[s]$ and $l \in \N_0$.
  The operation of inserting leaf $j$ to $x$ in location $l$ in $\tree[s]$ is the map $\left( \tree[s], x, l, j \right) \overset{\hat{\varphi}}{\mapsto} \that \in \Thatspace_{A \cup \{j\}}$ defined as follows. 
    \begin{itemize}
      \item If $x = e \in \edge ( \tree[s] )$, only $l = 0$ is allowed and $\that = \left( \varphi \left( \tree[s], e, j \right), \treesigma^\varphi \right)$ where $\sigma_{v'}^\varphi = 0$ if $v'$ is the parent of leaf $j$ in $\varphi ( \tree[s], e, j )$, and $\sigma_v^\varphi = \sigma_v$ for all other $v \in \branchpoints( \varphi ( \tree[s], e, j ) )$;
      \item If $x = v^\prime \in \branchpoints( \tree[s] )$ with $c \geq 2$ children, $l \in [c-1]$ is allowed and $\that = \left( \varphi \left( \tree[s], v^\prime, j \right), \treesigma^\varphi \right)$, where 
    \begin{align*}
      \sigma_{v^\prime}^\varphi ( l^\prime )
      =
      \begin{cases}
        \sigma_{v^\prime}(l^\prime) + 1 \left( \sigma_{v'}(l^\prime) \geq l \right) & \text{for all}\ l^\prime \in [c-2] \\
        l & \text{for}\ l^\prime = c-1,
      \end{cases}
    \end{align*}
    and $\sigma_v^\varphi = \sigma_v$ for all $v \in \branchpoints( \tree[s] ) \setminus \{v^\prime\}$.
    \end{itemize}
  \end{defi}
  Just as in the previous section, we will write $\hat{\varphi} ( \that[s], x, j, l )$ in both of the above cases, and then specify $x \in \edge (\tree[s])$ or $x \in \branchpoints(\tree[s])$ for the leaf insertion into the edge or vertex, respectively.
  Definition~\ref{def:insertionplanar} simply states, that whenever we insert a new leaf into a vertex, we can insert it to the right of the rightmost or in between any two children, with the exeption of the two leftmost, of that vertex.
  We are now ready to define a semi-planar version of the $(\alpha, \gamma)$-growth process:
\begin{defi}[Semi-planar $(\alpha, \gamma)$-growth process]\label{def:planaralphagamma}
  Fix $0 \leq \gamma \leq \alpha \leq 1$.
  Let $\That_1$ be the unique element of $\Thatspace_1$.
  For any $\nin$ construct $\That_{n+1}$ conditional on $\That_n = \that$, by setting 
    \begin{align}
        w_{x, l}
        =
        \begin{cases}
          1 - \alpha & \text{if}\ x \in \edge ( \that )\ \text{is external and}\ l = 0, \\
          \gamma & \text{if}\ x \in \edge ( \that )\ \text{is internal and}\ l = 0, \\
          \alpha & \text{if}\ x \in \branchpoints( \that )\ \text{with $c_x$ children and}\ l \in [c_x-2], \\
          \alpha - \gamma & \text{if}\ x \in \branchpoints( \that )\ \text{with $c_x$ children and}\ l = c_x-1, \\
        \end{cases}
        \label{weightsplanar}
    \end{align}
    and then defining the conditional distribution of $\That_{n+1}$ by
    \begin{align}
        \P \left( \That_{n+1} = \hat{\varphi} \left( \hat{\tree}, x, j, l \right) \Big| \ \That_n = \hat{\tree} \right)
        = \frac{w_{x,l}}{n - \alpha}
        \label{InsertionPlanarAlphaGamma}
    \end{align}
    for each $x \in \insertable \left( \that \right)$ and $l \in [c_x - 1]$ for $x \in \branchpoints (\that)$ and $l=0$ for $x \in \edge (\that)$.
    The sequence ${\left( \That_n \right)}_\nin$ is referred to as the \textit{semi-planar $(\alpha, \gamma)$-growth process}.
\end{defi}
The weights specified in~\eqref{weightsplanar} are such that the projection of the semi-planar $(\alpha, \gamma)$-growth process from $\Thatspace_{[n]}$ onto $\T_{[n]}$ exactly yields the (non-planar) $(\alpha, \gamma)$-growth process from Definition~\ref{def:growthprocess_alphagamma}.

In the semi-planar growth process, one can think of the additional structure of a branch point as a small comb tree rooted at the right end and with a leaf for every subtree, where naturally the two leftmost leaves connect to the same branch point, and where internal edges between branch points have weight $\alpha$ and the edge to the root has weight $\alpha - \gamma$ (see Figure~\ref{fig:semiplanar_element}).
In particular, we do not allow any insertions between or to the left of the two leftmost subtrees.
If $\That_n = \that \in \Thatspace_{[n]}$, it thus holds for any $v \in \branchpoints(\that)$ that if $i_1, \ldots, i_c$ denotes the least label in each of the subtrees rooted at $v$, enumerated in the order of least element, then $i_1$ and $i_2$ are always located in the two leftmost subtrees.

This is not the first definition of a binary tree growth process within the $(\alpha, \gamma)$-chain.
In~\cite{RefWorks:doc:5b4cbb5fe4b02dc0c79270af}, the authors use a colouring scheme of the edges to make the translation between the multifurcating and the binary setting, whereas we work with a semi-planar structure for the multifurcating trees.

Let us now investigate the left-to-right ordering of the subtrees in a branch point further.
For each $v \in \branchpoints(\that)$, let $i_1, \ldots, i_{c_v}$ denote the smallest leaf label in each of the $c_v$ subtrees rooted at $v$, enumerated in the order of least element, and let $\that_1, \ldots, \that_{c_v}$ denote the corresponding subtrees of $\that$.
Then, by definition of semi-planar trees, $\that_1$ and $\that_2$, are located leftmost in $v$.
The semi-planar $(\alpha, \gamma)$-growth process assigns a random location to each of the remaining $c_v - 2$ subtrees within the branch point, by inducing a permutation of $[c_v - 2]$, $\sigma_v$, associated with $v$.
Letting $\Sigma_v$ denote the random permutation induced by the growth process associated with $v$, and disregarding the two leftmost subtrees, we have that $\that_l$ is the $\Sigma_v(l - 2)$ leftmost subtree in $v$, for each $3 \leq l \leq c_v$. \\

\noindent
We can characterize $\Sigma_v$ in the following way:
\begin{lemma}
  \label{lemma:internalstructure}
  Let $\That_n$ be the $n$th step of the semi-planar $(\alpha, \gamma)$-growth process, and fix $\tree \in \T_{[n]}$.
  Then ${\left( \Sigma_v \right)}_{v \in \branchpoints(\tree)}$ is conditionally independent given $\hat{\pi} \left( \That_n \right) = \tree$, and 
  \begin{align}
    \forall v \in \branchpoints(\tree) \colon \P \left( \Sigma_v = \sigma \ \middle \vert \ \hat{\pi} \left( \That_n \right) = \tree \right) 
    = p^{c_v - 2}_{\alpha, \alpha - \gamma}( \sigma ),
    \label{eq:internalstructure}
  \end{align}
  where $p^{c_v - 2}_{\alpha, \alpha - \gamma}$ is the probability function of the random permutation of $[c_v-2]$ induced by the random order of the tables in an $(\alpha, \alpha - \gamma)$-$\ocrp$ with $c_v-2$ tables, described by~\eqref{eq:ocrp_tableorder}.
\end{lemma}
\begin{proof}
  The assertion that ${\left( \Sigma_v \right)}_{v \in \branchpoints{\tree}}$ is conditionally independent given $\hat{\pi} \left( \That_n \right) = \tree$ is a consequence of the independence of the leaf insertions.
  Now fix $v \in \branchpoints (\tree)$, and observe how the subtrees rooted at $v$ are analogous to the tables in an ordered Chinese Restaurant.
  Firstly, note that there is nothing to prove for $c_v = 2$, so fix $c_v = 3$.
  By construction the two subtrees with the smallest labels will be located leftmost in $v$ and there will be a single subtree to the right of them.
  That subtree will have weight $\alpha - \gamma$ to the right and $\alpha$ to the left.
  This corresponds exactly to the $(\alpha, \alpha - \gamma)$-oCRP with $1$ table.
  For any other $c_v > 3$, let us disregard the two leftmost subtrees, only focus on the remaining $c_v-2$, and consider how these have been placed by the growth process.
  Each insertion corresponds to a new customer in the ordered Chinese Restaurant analogy, and each subtree corresponds to a table.
  Amongst these subtrees, there will be weight $\alpha$ to the left of the leftmost subtree as well as in between any neighbouring pair.
  Furthermore there will be weight $\alpha - \gamma$ to the right of the rightmost subtree.
  This is exactly the same weight specification as in an $(\alpha, \alpha - \gamma)$-oCRP with $c_v-2$ tables, which finishes the proof.
\end{proof}
Lemma~\ref{lemma:internalstructure} naturally characterizes a conditional distribution on $\Thatspace_{[n]}$.
Fix some $\tree \in \T_{[n]}$ and, for each $v \in \branchpoints(\tree)$, let $c_v$ denote the number of children of $v$.
Then
\begin{align}
  \label{eq:condprop_planargivenshape}
  \P \left( \That_n = \that \ \bigg\vert \ \hat{\pi} \left( \That_n \right) = \tree \right) 
  &= \prod_{v \in \branchpoints(\tree)} \P \left( \Sigma_v = \sigma_v \ \bigg\vert \ \hat{\pi} \left( \That_n \right) = \tree \right) \\
  &= \prod_{v \in \branchpoints(\tree)} p^{c_v-2}_{\alpha, \alpha - \gamma} \left( \sigma_v \right) \nonumber.
\end{align}
%
%%%%%%%%%%%%%%%%%%%%%%%%%%%%%%%%%%%%%%%%%%%%%%%%%%%%%%%%%%%%%%%%%%%%%%%%%%%%%%%%%%%%
%%%%%%%%%%%%%%%%%%%%%%%%%%%%% Semi-planar down-up chain %%%%%%%%%%%%%%%%%%%%%%%%%%%%%%%%%
%%%%%%%%%%%%%%%%%%%%%%%%%%%%%%%%%%%%%%%%%%%%%%%%%%%%%%%%%%%%%%%%%%%%%%%%%%%%%%%%%%%%
%
\section{A semi-planar down-up chain}
The motivation for introducing the semi-planar version of the $(\alpha, \gamma)$-growth model was fundamentally to get an ordering of the subtrees within a multifurcating branch point.
This was done to enable a generalization of the $\alpha$-chain described in the introduction, by generalizing the local search from a designated leaf to a multifurcating setting where there might be more than one subtree to search in the first branch point on the ancestral line.

To avoid any ambiguity, we will update the notions of swapping and deleting leaves, respectively, before defining a semi-planar down-up chain.
The sole reason for defining the label swapping map $\hat{\tau}$ on the space of planar trees, $\Thatspace_A^*$, rather than directly on $\Thatspace_A$, is that certain label swaps will bring us outside of the latter state space.
However, the definition can easily be applied to elements of $\Thatspace_A$ by using the identification described just after Definition~\ref{def:semiplanartree}.
See Figure~\ref{fig:labelswapping} for an illustration.
Recall that $\tau$ denotes the label swapping map on $\T_{[n]}$ from Definition~\ref{def:labelswapping}.
\begin{defi}[Label swapping in $\Thatspace_A^*$]\label{def:labelswappingplanar}
  Let $A$ be a finite set and fix $\that^* = (\tree, \treesigma^*) \in \Thatspace_A^*$.
  The operation of \textit{swapping labels} is a map $\left( \that^*, i, j \right) \stackrel{\hat{\tau}}{\mapsto} \left( \tau(\tree, i, j), \treesigma^{* \tau} \right)$ for any pair $i,j \in [n]$ that satisfies that if $v \in \branchpoints (\tree)$ is the parent of $i$, and $i_1, \ldots, i_{c_v}$ denotes the smallest leaf label of the subtrees of $v$ in $\tree$, enumerated in increasing order, then $j \in \{i_1, \ldots, i_{c_v}\}$ if $c_v > 2$ or $j \in \{i_1, i_2, b\}$ if $c_v = 2$ where $b$ is the smallest leaf label in the second spinal bush on the ancestral line from $i$. 
  If $c_v > 2$, $i = i_{l_i}$ and $j = i_{l_j}$ some $l_i, l_j \in [c_v]$, then for each $l \in [c_v]$ let
  \begin{align*}
    \sigma_v^{* \tau} (l)
    = \begin{cases}
      \sigma_v^* (l_i) & \text{if}\ l = l_j, \\
      \sigma_v^* (l+1) & \text{if}\ l_i \leq l < l_j, \\
      \sigma_v^* (l-1) & \text{if}\ l_j < l \leq l_i, \\
      \sigma_v^* (l) & \text{otherwise},
    \end{cases}
  \end{align*}
  and $\sigma_{v^\prime}^{* \tau} = \sigma_{v^\prime}^*$ for all other $v^\prime \in \branchpoints (\tree)$.
  If $c_v \leq 2$, let $w$ denote the grandparent of $i$, and define $\sigma_w^*$ analogously to the above, with the slight alteration that $l_i$ and $l_j$ denotes the ranks of the subtrees containing $i$ and $j$, respectively, in $w$, in the order of least elements.
\end{defi}
The above definition seems technical, but it simply states that all subtrees in $v_i$ and $v_j$ keep the positions, with the subtree containing $j$ in $v_j$ taking the position of the subtree containing $i$ in $v_i$ and vice versa.
  The label swapping will be used prior to deleting a leaf in the semi-planar down-up chain defined in this section.
  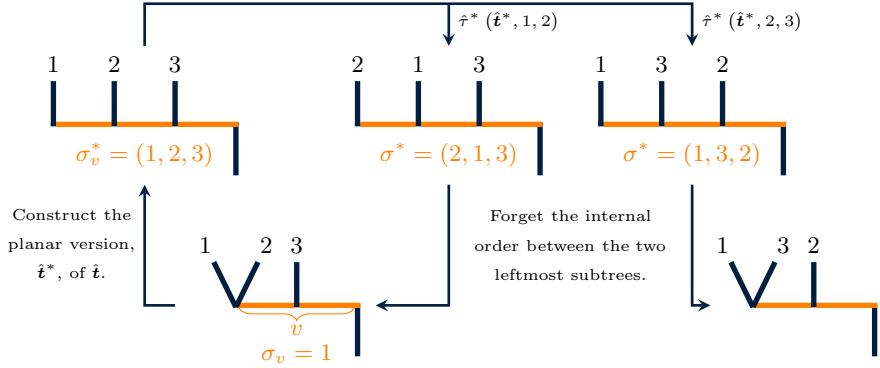
\begin{figure}[t]
    \centering
    \begin{tikzpicture}[%
    bp/.style={rectangle, draw, color=oxfordblue, inner sep=0pt, minimum width=1.6pt},
    comment/.style={circle, draw, color=orange},
    label/.style={text=oxfordblue},
    line/.style={-, line width = 2pt, color = oxfordblue},
    arrowline/.style={->, line width = 1pt, color = oxfordblue, >=stealth},
    scale=0.8
  ]
  \tikzstyle{every node}=[font=\small]

    %%%%%%% T1
  \begin{scope}[shift = {(1,0)}, xscale = -1]
    \node (0) at (0,0) {};
    \node[bp] (bp_start) at (0,1) {};
    \node[bp] (bp_mid) at (1, 1) {};
    \node[bp] (bp_end) at (2, 1) {};
    \node (1) at (2.5, 2) {$1$};
    \node (2) at (1.5, 2) {$2$};
    \node (3) at (1, 2) {$3$};

    \draw[line] (0) to (bp_start);
    \draw[line, color = orange] (bp_start.east) to (bp_mid) to (bp_end.east);
    \draw[line] (bp_mid.south) to (3);
    \draw[line] (bp_end.south) to (1);
    \draw[line] (bp_end.south) to (2);

    \draw [decorate,decoration={brace,amplitude=5pt}, orange] (bp_start.south west) to (bp_end.south east);
    \node [orange] at (1, 0.6) {$v$};
    \node [orange] at (1, 0.2) {$\sigma_v = 1$};
  \end{scope}
  \begin{scope}[shift = {(-1,3)}, xscale = -1]
    \node (0) at (0,0) {};
    \node[bp] (bp_start) at (0,1) {};
    \node[bp] (bp_mid) at (1, 1) {};
    \node[bp] (bp_end) at (2, 1) {};
    \node[bp] (bp_end2) at (3, 1) {};
    \node (1) at (3, 2) {$1$};
    \node (2) at (2, 2) {$2$};
    \node (3) at (1, 2) {$3$};

    \draw[line] (0) to (bp_start);
    \draw[line, color = orange] (bp_start.east) to (bp_mid) to (bp_end) to (bp_end2);
    \draw[line] (bp_mid.south) to (3);
    \draw[line] (bp_end2.south) to (1);
    \draw[line] (bp_end.south) to (2);

    \node [orange] at (1.5, 0.5) {$\sigma_v^* = (1,2,3)$};
  \end{scope}
  \begin{scope}[shift = {(4,3)}, xscale = -1]
    \node (0) at (0,0) {};
    \node[bp] (bp_start) at (0,1) {};
    \node[bp] (bp_mid) at (1, 1) {};
    \node[bp] (bp_end) at (2, 1) {};
    \node[bp] (bp_end2) at (3, 1) {};
    \node (2) at (3, 2) {$2$};
    \node (1) at (2, 2) {$1$};
    \node (3) at (1, 2) {$3$};

    \draw[line] (0) to (bp_start);
    \draw[line, color = orange] (bp_start.east) to (bp_mid) to (bp_end) to (bp_end2);
    \draw[line] (bp_mid.south) to (3);
    \draw[line] (bp_end2.south) to (2);
    \draw[line] (bp_end.south) to (1);

    \node [orange] at (1.5, 0.5) {$\sigma^* = (2,1,3)$};
  \end{scope}
  \begin{scope}[shift = {(8,3)}, xscale = -1]
    \node (0) at (0,0) {};
    \node[bp] (bp_start) at (0,1) {};
    \node[bp] (bp_mid) at (1, 1) {};
    \node[bp] (bp_end) at (2, 1) {};
    \node[bp] (bp_end2) at (3, 1) {};
    \node (1) at (3, 2) {$1$};
    \node (3) at (2, 2) {$3$};
    \node (2) at (1, 2) {$2$};

    \draw[line] (0) to (bp_start);
    \draw[line, color = orange] (bp_start.east) to (bp_mid) to (bp_end) to (bp_end2);
    \draw[line] (bp_mid.south) to (2);
    \draw[line] (bp_end2.south) to (1);
    \draw[line] (bp_end.south) to (3);

    \node [orange] at (1.5, 0.5) {$\sigma^* = (1,3,2)$};
  \end{scope}
  \begin{scope}[shift = {(9.5,0)}, xscale = -1]
    \node (0) at (0,0) {};
    \node[bp] (bp_start) at (0,1) {};
    \node[bp] (bp_mid) at (1, 1) {};
    \node[bp] (bp_end) at (2, 1) {};
    \node (1) at (2.5, 2) {$1$};
    \node (3) at (1.5, 2) {$3$};
    \node (2) at (1, 2) {$2$};

    \draw[line] (0) to (bp_start);
    \draw[line, color = orange] (bp_start.east) to (bp_mid) to (bp_end.east);
    \draw[line] (bp_mid.south) to (2);
    \draw[line] (bp_end.south) to (1);
    \draw[line] (bp_end.south) to (3);
  \end{scope}

  \draw[arrowline] (-2, 1) to (-2.5, 1) to (-2.5, 3); % Arrow up
  \draw[arrowline] (-2.5, 5.3) to (-2.5, 6) to (6.5, 6) to (6.5, 5.3); % Long top arrow
  \draw[arrowline] (2.5, 6) to (2.5, 5.3); % Short top arrow
  \draw[arrowline] (2.5, 3) to (2.5, 1) to (1.3, 1); % Arrow down left
  \draw[arrowline] (6.5, 3) to (6.5, 1) to (6.8, 1); % Arrow down right

  \node [anchor = west] (tau1) at (2.5, 5.7) {\tiny $\hat{\tau}^* \left( \that^*, 1, 2 \right)$};
  \node [anchor = west] (tau2) at (6.5, 5.7) {\tiny $\hat{\tau}^* \left( \that^*, 2, 3 \right)$};
  \node [align = center] at (4.5, 2) {\tiny Forget the internal \\ \tiny order between the two \\ \tiny leftmost subtrees.};
  \node [align = center, anchor = east] at (-2.5, 2) {\tiny Construct the \\ \tiny planar version, \\ \tiny $\that^*$, of $\that$.};
\end{tikzpicture}    
    \caption{Swapping labels in $\Thatspace_{[3]}^*$.
      The orange “edges” is an illustrative expansion of the branch point $v$ so that the left-to-right ordering of the subtrees of $v$ becomes apparent.
      Thus, all semi-planar and planar trees depicted above have the same tree shape, but have different permutations associated to $v$.
      Label swapping in $\that$ (bottom left) occurs by constructing the planar version, $\that^*$ (top left), and then swapping the labels in the planar version.
      Swapping labels $1$ and $2$ yields a planar tree corresponding to a semi-planar tree since the two subtrees with the lowest labels are located leftmost.
      This is not the case if we swap labels $2$ and $3$, and thus the bottom right is not a semi-planar tree, i.e.\ not an element of $\Thatspace_{[3]}$.
    }
    \label{fig:labelswapping}
  \end{figure}
  \fxnote{Look at Matthias comments.}
  Swaps like this will occur in the down-up chain, but the subsequent deletion of a leaf will make sure we end up with a semi-planar tree again.
  % Before proving this claim, we will update the definition of deleting a leaf to the semi-planar setting.
%
\begin{defi}[Deletion of a leaf in $\Thatspace_A^*$]\label{deletionplanaraug}
    Let $A$ be a finite set and fix $\that^* \in \Thatspace_A^*$.
    The operation of deleting leaf $j \in A$ from $\that^*$ is the map
    \begin{align}
        \left( \that^*, j \right)
        = \left( \tree, \treesigma^*, j \right)
        \overset{\mathring{\hat{\rho}}^*}{\mapsto}
        \left( \mathring{\rho} \left( \tree, j \right), \treesigma^{* \rho} \right)
        \in \Thatspace_{A \setminus \{j\}}^*
        \label{eq:deletionmapplanaraug}
    \end{align}
    where $\sigma_u^{* \rho} = \sigma_u^*$ for all $u \neq v$, $v$ is the parent of $j$ in $\that$, and, if $c_v > 2$,
    \begin{align}
      \label{eq:deletionpermutation}
      \sigma_v^{* \rho} (i)
      =
      \sigma_v^*(i) - 1 \left( \sigma_v^*(i) > \sigma_v^* (i_j) \right)
    \end{align}
    for all $i \in [c_v - 1]$, where $c_v$ is the number of children of $v$ in $\that$, and $i_j$ is the rank of $j$ amongst the smallest leaf labels of the subtrees rooted at $v$ in $\that$.
    As in Definition~\ref{def:leaf_deletion}, we denote by $\hat{\rho}^*$ the map that, after applying $\mathring{\hat{\rho}}^*$, subsequently relabels $A \setminus \{j\}$ by $[\# A - 1]$ using the increasing bijection.
\end{defi}
Of course, \eqref{eq:deletionpermutation} is not needed in the case where $c_v = 2$ since the branch point will completely disappear upon the deletion of leaf $j$, and the corresponding permutation will thus not be included in $\treesigma^{* \rho}$.
The above definition states, that if $\that_{l}^*$ and $\that_{r}^*$ are the neighbouring subtrees to the left and right, respectively, of leaf $j$ in $\that^*$ before deletion, then $\that_l^*$ is the left neighbour of $\that_r^*$ after the deletion.
\begin{defi}[Semi-planar $(\alpha, \gamma)$-chain]\label{def:planaralphagammachain}
  Fix $n \geq 2$.
  The \textit{semi-planar $(\alpha, \gamma)$-chain} on $\Thatspace_{[n]}$ is an $(f, (\alpha, \gamma))$-down-up chain where $f \colon \Thatspace_{[n]} \times [n] \to \Thatspace_{[n]}^* \times [n]$ is defined by setting $f(\that, i) = \left( \hat{\tau}^* \left( \that, i, \tildei \right), \tildei \right)$, with $\tildei = \max \left\{ i, a, b \right\}$, where we by $l_i$ denote the rank of $i$ amongst the least leaf labels of the $c_v$ subtrees of the parent $v$ of $i$ in $\that$, set
  \begin{align*}
    &a = \min
    \begin{Bmatrix}
      \text{leaf labels in the first spinal bush} \\ \text{on the ancestral line from leaf}\ i
    \end{Bmatrix},
  \end{align*}
  and
  \begin{align*}
    &b = 
    \begin{cases}
      \min
      \begin{Bmatrix}
        \text{leaf labels in the second spinal bush} \\ \text{on the ancestral line from leaf}\ i
      \end{Bmatrix}
      & \text{if $c_v = 2$}, \\
      \min
      \begin{Bmatrix}
        \text{leaf labels in the first subtree not} \\ \text{containing $a$ to the right of leaf $i$}
      \end{Bmatrix}
      & \text{if $c_v > 2$, $l_i \leq 2$}, \\
      \min
      \begin{Bmatrix}
        \text{leaf labels in the first subtree not} \\ \text{containing $a$ to the left of leaf $i$}
      \end{Bmatrix}
      & \text{if $c_v > 2$, $l_i > 2$}, \\
    \end{cases}
  \end{align*}
  with the convention that $b = 0$ if $c_v = 2$ and the ancestral line from leaf $i$ has only one spinal bush.
\end{defi}
Firstly, recall that we can apply $\hat{\tau}$, and hence $f$, to elements of $\Thatspace_{[n]}$ (instead of $\Thatspace_{[n]}^*$) by using the canonical way of translating a semi-planar tree to a planar tree described just after Definition~\ref{def:semiplanartree}. 
However, it is not clear a priori that carrying out the down-step of the above chain started from a semi-planar tree again yields one such.
This is needed for the above to be well-defined as a Markov chain on $\Thatspace_{[n]}$.
The following result shows that this is indeed the case:
\begin{lemma}
  Let $A$ be a finite set and fix $\that \in \Thatspace_A$.
For any $i \in A$ it holds that $\mathring{\hat{\rho}}^* \left( \hat{\tau}^* (\that, i, \tildei), \tildei \right) \in \Thatspace_{A \setminus \{\tildei\}}$, where $\tildei$ is defined as in Definition~\ref{def:planaralphagammachain}.
  \label{lemma:downstepsemiplanaroutput}
\end{lemma}
\begin{proof}
  We have to check that for every branch point in $\mathring{\hat{\rho}} \left( \hat{\tau}^* (\that, i, \tildei), \tildei \right)$, the two subtrees with the smallest leaf labels, respectively, are located leftmost.
  From the definition of the semi-planar $(\alpha, \gamma)$-chain it is clear that the only two branch points in $\that$ where any changes can take place are the parent $v$ and potentially the grandparent $v'$ of leaf $i$ in $\that$, since all other branch points will remain unchanged as $i \leq \tildei$ and $\tildei$ is the smallest leaf in its subtree.
  We now split up into cases based on the location of leaf $i$ in $v \in \branchpoints (\that^*)$, where $\that^*$ is the canonical representative of $\that$ in $\Thatspace_A^*$ described just after Definition~\ref{def:semiplanartree}.

  If $i$ is not one of the two leftmost subtrees in $v$, there is no option of swapping with the smallest leaf label in either of the two leftmost subtrees, as they are both smaller than $i$ by definition of a semi-planar tree.

  Now consider the case where $i$ is placed leftmost in $v$, with $v$ having more than two children.
  In $\that^*$, $i$ will then be placed leftmost or second leftmost in $v$.
  In either case $\tildei = b$ will be the smallest label in the third leftmost subtree of $v$ in $\that^*$.
  $\hat{\tau}^* (\that, i, \tildei) \in \Thatspace_A^*$ will not correspond to an element of $\Thatspace_A$, but consider the ordering of the subtrees of $v$ in this planar tree.
  $i$ will have retained its rank ($1$ or $2$) amongst the least labels of the subtrees of $v$, and will now be located in the third leftmost subtree.
  Upon the deletion of $\tildei$ in $\hat{\tau}^* (\that, i, \tildei)$ the two leftmost subtrees will again be the ones containing the two smallest of the least leaf labels of the subtrees of $v$.

  %%%%% REWRITE
%  Either, $\tildei = b$ is now placed leftmost in $v$, in which case the third subtree from the left will contain the lowest label of all the subtrees of $v$, and will upon the deletion of $\tildei$ be the second subtree from the left in $v$.
%  Otherwise, $\tildei = b$ is the second subtree from the left, in which case the third subtree from the left contains the second smallest label of all the subtrees of $v$.
%  After the deletion of $\tildei$ this subtree will again be located second from the left, and the previous arguments shows that this subtree is indeed a semi-planar tree itself.

  The last case is when $i$ is placed leftmost in $v$, with $v$ only having two children.
  Here, $b$ will always be found in $v'$, the parent of $v$, and no matter which swap occurs, the branch point $v$ will disappear upon deletion of $\tildei$.
  Hence we only need to check that the new ordering of subtrees in $v'$ is acceptable.
  If $\tildei \in \{a, i\}$ there is nothing to prove, as this means that the two subtrees with the lowest labels in $v'$ will not be affected by the swapping of $i$ and $\tildei$ and subsequent deletion of $\tildei$.
  If $\tildei = b$, this means that $b$ was the smallest leaf label (not in a subtree of $v$) in $v'$ and that $a$ was even smaller.
  Hence, by definition of the $(\alpha, \gamma)$-growth rule, the smallest labels in the two leftmost subtrees of $v'$ after swapping and deleting, respectively, will be $a$ and $i$.
  
  In all cases we conclude that $\mathring{\hat{\rho}}^* \left(  \hat{\tau} (\that, i, \tildei), \tildei \right) \in \Thatspace_{A \setminus \{\tildei\}}$.
\end{proof}
We will approach the proof of Theorem~\ref{thm:planarstationarity} via the following technical result.
The statement of the lemma as well as the proof technique is an adaptation of Lemma $4$ and the proof of Theorem $1$ in~\cite{RefWorks:doc:5b4cbc93e4b07f5746e47014}, designed to work in the case where $\alpha = \gamma = \frac{1}{2}$.
It is possible to prove that the stationary distribution of a down-up with the transition kernel described by Proposition~\ref{prop:nonplanar_alphagamma_chain} is the distribution of the the $n$th step of the $(\alpha, \gamma)$-growth rule without any reference to a semi-planar structure, thus making this lemma obsolete, in the case where $\gamma = 1 - \alpha$ for $\alpha \in [0,1]$.
However, this falls outside the scope of this exposition.
\begin{lemma}\label{lemma:condind}
  Fix $\nin$ and let ${\left( \That_m \right)}_{1 \leq m \leq  n}$ be the first $n$ steps of the semi-planar $(\alpha, \gamma)$-growth process.
    Fix $1 \leq i \leq \tildei \leq n$ and define
    \begin{align}
      E_{i, \tildei} = \left\{ \tildei = \max\left\{i, a, b\right\} \text{in}\ \That_n \right\},
        \label{eq:condindlemma}
    \end{align}
    where $a$ and $b$ are as in Definition~\ref{def:planaralphagammachain}.
    Then $E_{i, \tildei} \independent \That_{\tildei - 1}$ and, conditional on the event $E_{i, \tildei}$, $\hat{\rho}^* \left( \hat{\tau}^* \left( \That_n, i, \tildei \right), \tildei \right)$ and $\That_{n-1}$ have the same distribution.
\end{lemma}
    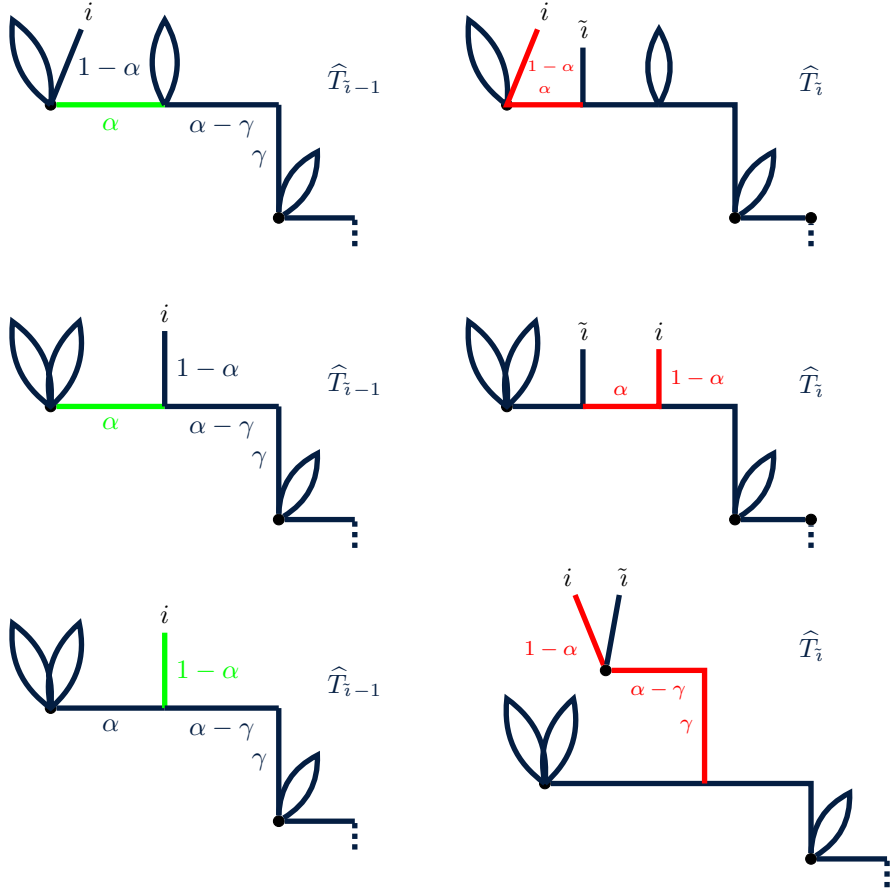
\begin{figure}[t]
      \centering
      \begin{tikzpicture}[%
        table/.style={circle, draw, minimum width=0.7cm},
        bp/.style={circle, draw, fill=black, inner sep=0pt, minimum width=4pt},
        comment/.style={circle, draw, color=red},
        label/.style={text=oxfordblue},
        line/.style={-, line width = 2pt, color = oxfordblue},
        scale=1
    ]

    \begin{scope}[xscale = -1]
      \node (0) at (0,0) {};
      \node (1bp_start) at (0,0.5) {};
      \node [bp] (1bp_end) at (1,0.5) {};
      \node (tree1) at (0.5,1.5) {};
      \node (bp_start) at (1,2) {};
      \node (bp_mid2) at (2.5,2) {};
      \node[anchor=south] (2) at (2.5,3) {};
      \node [bp] (bp_end) at (4,2) {};
      \node [anchor = south] (1) at (4.5,3) {};
      \node [anchor = south] (i) at (3.5,3) {$i$};

      \draw[line, dotted] (0) to (1bp_start.center);
      \draw[line] (1bp_start.center) to (1bp_end);

      \draw[line] (bp_start.center) to node[auto, outer sep = 1pt] {$\alpha-\gamma$} (bp_mid2.center);
      \draw[line, green] (bp_mid2.center) to node[auto, outer sep = 2pt] {$\alpha$} (bp_end);

      \draw[line, bend left] (bp_mid2.center) to (2.center) to (bp_mid2.center);
      \draw[line] (1bp_end) [bend left] to (tree1) [bend left] to (1bp_end);
      \draw[line, bend left] (bp_end.center) to (1.center) to (bp_end.center);
      \draw[line] (bp_end.center) to node[auto, swap, anchor = west] {$1 - \alpha$} (i);

      \draw[line] (1bp_end) to node[auto] {$\gamma$} (bp_start.center);

      \node[label, anchor = south] (T_itilde) at (0,2) {$\That_{\tildei-1}$};

      \begin{scope}[shift = {(-6,0)}]
        \node (0) at (0,0) {};
        \node [bp] (1bp_start) at (0,0.5) {};
        \node [bp] (1bp_end) at (1,0.5) {};
        \node (tree1) at (0.5,1.5) {};
        \node (bp_start) at (1,2) {};
        \node (bp_mid1) at (2,2) {};
        \node (2) at (2,3) {};
        \node (bp_mid2) at (3,2) {};
        \node (itilde) at (3,3) {$\tildei$};
        \node [bp] (bp_end) at (4,2) {};
        \node [anchor = south] (1) at (4.5,3) {};
        \node [anchor = south] (i) at (3.5,3) {$i$};

        \draw[line, dotted] (0) to (1bp_start);
        \draw[line] (1bp_start) to (1bp_end) to (bp_start.center) to (bp_mid1.center);
        \draw[line] (itilde) to (bp_mid2.center) to (bp_mid1.center);
        \draw[line, bend left] (bp_mid1.center) to (2.center) to (bp_mid1.center);
        \draw[line, bend left] (1bp_end) to (tree1) to (1bp_end);
        \draw[line] (bp_end.center) [bend left] to (1.center) [bend left] to (bp_end.center);
        \draw[line, red] (i) to node[auto, swap, anchor = west, inner sep = 1pt] {\tiny $1 - \alpha$} (bp_end.center) to node[auto] {\tiny $\alpha$} (bp_mid2.center);

        \node[label, anchor = south] (T_itilde) at (0,2) {$\That_{\tildei}$};
      \end{scope}

      \begin{scope}[shift = {(0,-4)}]
        \node (0) at (0,0) {};
        \node (1bp_start) at (0,0.5) {};
        \node [bp] (1bp_end) at (1,0.5) {};
        \node (tree1) at (0.5,1.5) {};
        \node (bp_start) at (1,2) {};
        \node (bp_mid2) at (2.5,2) {};
        \node[anchor=south] (i) at (2.5,3) {$i$};
        \node [bp] (bp_end) at (4,2) {};
        \node [anchor = south] (1) at (4.5,3) {};
        \node [anchor = south east] (2) at (3.5,3) {};

        \draw[line, dotted] (0) to (1bp_start.center);
        \draw[line] (1bp_start.center) to (1bp_end);

        \draw[line] (1bp_end) to node[auto] {$\gamma$} (bp_start.center);

        \draw[line, green] (bp_mid2.center) to node[auto] {$\alpha$} (bp_end);

        \draw[line] (bp_mid2.center) to node[auto, swap] {$1-\alpha$} (i);
        \draw[line] (1bp_end) [bend left] to (tree1) [bend left] to (1bp_end);
        \draw[line] (bp_end.center) [bend left] to (1.center) [bend left] to (bp_end.center) [bend left] to (2.center) [bend left] to (bp_end.center);

        \draw[line] (bp_start.center) to node[auto] {$\alpha - \gamma$} (bp_mid2.center);

        \node[label, anchor = south] (T_itilde) at (0,2) {$\That_{\tildei-1}$};
      \end{scope}

      \begin{scope}[shift = {(-6,-4)}]
        \node (0) at (0,0) {};
        \node [bp] (1bp_start) at (0,0.5) {};
        \node [bp] (1bp_end) at (1,0.5) {};
        \node (tree1) at (0.5,1.5) {};
        \node (bp_start) at (1,2) {};
        \node (bp_mid1) at (2,2) {};
        \node (i) at (2,3) {$i$};
        \node (bp_mid2) at (3,2) {};
        \node (itilde) at (3,3) {$\tildei$};
        \node [bp] (bp_end) at (4,2) {};
        \node [anchor = south] (1) at (4.5,3) {};
        \node [anchor = south east] (2) at (3.5,3) {};

        \draw[line, dotted] (0) to (1bp_start);
        \draw[line] (1bp_start) to (1bp_end) to (bp_start.center) to (bp_mid1.center);
        \draw[line] (bp_end) to (bp_mid2.center) to (itilde);
        \draw[line, red] (bp_mid2.center) to node[auto] {\footnotesize $\alpha$} (bp_mid1.center) to node[auto, swap] {\footnotesize $1-\alpha$} (i);
        \draw[line] (1bp_end) [bend left] to (tree1) [bend left] to (1bp_end);
        \draw[line] (bp_end.center) [bend left] to (1.center) [bend left] to (bp_end.center) [bend left] to (2.center) [bend left] to (bp_end.center);

        \node[label, anchor = south] (T_itilde) at (0,2) {$\That_{\tildei}$};
      \end{scope}

      \begin{scope}[shift = {(0,-8)}]
        \node (0) at (0,0) {};
        \node (1bp_start) at (0,0.5) {};
        \node [bp] (1bp_end) at (1,0.5) {};
        \node (tree1) at (0.5,1.5) {};
        \node (bp_start) at (1,2) {};
        \node (bp_mid2) at (2.5,2) {};
        \node[anchor=south] (i) at (2.5,3) {$i$};
        \node [bp] (bp_end) at (4,2) {};
        \node [anchor = south] (1) at (4.5,3) {};
        \node [anchor = south east] (2) at (3.5,3) {};

        \draw[line, dotted] (0) to (1bp_start.center);
        \draw[line] (1bp_start.center) to (1bp_end);

        \draw[line] (1bp_end) to node[auto] {$\gamma$} (bp_start.center);

        \draw[line] (bp_mid2.center) to node[auto] {$\alpha$} (bp_end);

        \draw[line, green] (bp_mid2.center) to node[auto, swap] {$1-\alpha$} (i);
        \draw[line] (1bp_end) [bend left] to (tree1) [bend left] to (1bp_end);
        \draw[line] (bp_end.center) [bend left] to (1.center) [bend left] to (bp_end.center) [bend left] to (2.center) [bend left] to (bp_end.center);

        \draw[line] (bp_start.center) to node[auto] {$\alpha - \gamma$} (bp_mid2.center);

        \node[label, anchor = south] (T_itilde) at (0,2) {$\That_{\tildei-1}$};
      \end{scope}

      \begin{scope}[shift = {(-6,-7.5)}]
        \node (0) at (-1,-1) {};
        \node (0bp_start) at (-1, -0.5) {};
        \node [bp] (0bp_end) at (0, -0.5) {};
        \node (tree1) at (-0.5,0.5) {};
        \node (1bp_start) at (0,0.5) {};
        \node (1bp_mid) at (1.4,0.5) {};
        \node [bp] (1bp_end) at (3.5,0.5) {};
        \node [anchor = south east] (tree2) at (3,1.5) {};
        \node [anchor = south] (tree3) at (4,1.5) {};

        \node (bp_start) at (1.4,2) {};
        \node [bp] (bp_end) at (2.7,2) {};
        \node [anchor = south] (1) at (3.2,3) {$i$};
        \node [anchor = south east] (2) at (2.3,3) {$\tildei$};

        \draw[line, dotted] (0) to (0bp_start.center);
        \draw[line] (0bp_start.center) to (0bp_end) to (1bp_start.center) to (1bp_mid.center) to (1bp_end);
        \draw[line] (0bp_end) [bend left] to (tree1) [bend left] to (0bp_end);
        \draw[line] (1bp_end.center) [bend left] to (tree3.center) [bend left] to (1bp_end.center);
        \draw[line] (1bp_end.center) [bend left] to (tree2.center) [bend left] to (1bp_end.center);
        \draw[line, red] (1bp_mid.center) to node[auto] {\footnotesize $\gamma$} (bp_start.center) to node[auto] {\footnotesize $\alpha - \gamma$} (bp_end) to node[auto] {\footnotesize $1-\alpha$} (1);
        \draw[line] (bp_end) to (2);

        \node[label, anchor = south] (T_itilde) at (0,2) {$\That_{\tildei}$};
      \end{scope} 
    \end{scope}
\end{tikzpicture}
      % \captionsetup{singlelinecheck=off}
      \caption[fig:ForbiddenEdges_1]{Illustration of three of the characterizations of the event $E_{i, \tildei}$ in the case where $i < \tildei$ and $i$ is located in a multifurcating branch point, from the proof of Lemma~\ref{lemma:condind}.
        The horizontal lines in each illustration correspond to the expansion of branch points described just after Definition~\ref{def:planaralphagamma}.
      Each row is an insertion scenario, where the left column is an illustration of $\That_{\tildei - 1}$ where the insertion location of $\tildei$ is colored green, while the right column is an illustration of $\That_{\tildei}$ where the edges colored in red are the ones where insertions of $\tildei + 1, \ldots, n$ are disallowed.}
      \label{fig:ForbiddenEdges_1}
    \end{figure}
\begin{proof}
    We will prove the first assertion by splitting into two cases.
    The overall proof technique will be to translate $E_{i, \tildei}$ to restrictions regarding the growth process and to use this characterization to obtain the result.

    \textbf{Case A:} $i < \tildei$. In this case, $E_{i, \tildei}$ is equivalent to one set of the following disjoint events, some of them illustrated in Figure~\ref{fig:ForbiddenEdges_1}, where we will denote the parent branch point of $i$ in $\That_{\tildei - 1}$ by $v$ and the number of children of $v$ in $\That_{\tildei - 1}$ by $c_v$.
        \begin{itemize}
          \item if $c_v = 2$ then either
            \begin{enumerate}
              \item $\tildei$ is inserted into $v$ (as right neighbour of $i$), and
              \item $\tildei + 1, \ldots, n$ are \textit{not} inserted on the leaf edge $i$ or as right neighbour $i$ in $v$ in $\That_{\tildei}, \ldots, \That_{n-1}$, respectively,
            \end{enumerate}
            or
            \begin{enumerate}
              \item $\tildei$ is inserted on the parent edge of $i$, $e_{v\downarrow}$, and
              \item $\tildei + 1, \ldots, n$ are $\textit{not}$ inserted on the leaf edge $i$, into the branch point $v$, nor on the edge $e_{v\downarrow}$ in $\That_{\tildei}, \ldots, \That_{n-1}$, respectively;
            \end{enumerate}
          \item if $c_v > 2$ and $i$ is leftmost in $v$, then
            \begin{enumerate}
              \item $\tildei$ is inserted as the right neighbour of $i$ in $v$, and
              \item $\tildei + 1, \ldots, n$ are \textit{not} inserted on leaf edge $i$ or between $i$ and the subtree containing $\tildei$ in the branch point $v$ in $\That_{\tildei}, \ldots, \That_{n-1}$, respectively.
            \end{enumerate}
            This situation is depicted in the top row of Figure~\ref{fig:ForbiddenEdges_1}.
          \item if $c_v > 2$ and $i$ is not leftmost in $v$, then
            \begin{enumerate}
              \item $\tildei$ is as the left neighbour of $i$ in $v$, and
              \item $\tildei + 1, \ldots, n$ are \textit{not} inserted on leaf edge $i$ or between $i$ and the subtree containing $\tildei$ in the branch point $v$ in $\That_{\tildei}, \ldots, \That_{n-1}$, respectively.
            \end{enumerate}
            This situation is depicted in the middle row of Figure~\ref{fig:ForbiddenEdges_1}.
        \end{itemize}
        In addition to the above, there is, regardless of which of the above scenarios we are in, the option to
            \begin{enumerate}
              \item insert $\tildei$ on the leaf edge $i$ in $\That_{\tildei - 1}$, and
              \item \textit{not} insert $\tildei + 1, \ldots, n$ on the first two edges nor in the first branch point on the ancestral line from leaf $i$ in $\That_{\tildei}, \ldots, \That_{n-1}$, respectively,
            \end{enumerate}
            which is illustrated in the bottom row of Figure~\ref{fig:ForbiddenEdges_1}. \smallskip

          \textbf{Case B:} $i = \tildei$.
          By the same reasoning as earlier, $E_{i, \tildei}$ is equivalent to
            \begin{itemize}
              \item \textit{not} inserting $\tildei + 1, \ldots, n$ on the first two edges, nor in the parent branch point of $i$, on the ancestral line from $i$ in $\That_{\tildei}, \ldots, \That_{n-1}$, if $i = \tildei$ was inserted into an edge of $\That_{\tildei - 1}$,
            \end{itemize}
            or
            \begin{itemize}
              \item \textit{not} inserting $\tildei + 1, \ldots, n$ on the leaf edge $i = \tildei$ or as the left neighbour of $i = \tildei$ in the first branch point on the ancestral line from $i$ in $\That_{\tildei}, \ldots \That_{n-1}$, respectively, if $i = \tildei$ is inserted into a branch point of $\That_{\tildei - 1}$.
            \end{itemize}
            We have now characterized the event $E_{i,\tildei}$ in terms of making or not making insertions, respectively, to specific edges with weights summing to $1$ in all cases.
            This implies that the probability of making or not making these insertions, respectively, is the same in all cases, and resolving the telescoping products, we find
    \begin{align*}
      &\P \left( \left\{ \That_{\tildei - 1} = \hat{\rho}^* ( \that, [n] \setminus [\tildei - 1] ) \right\} \cap E_{i, \tildei} \right) \\
      =\
      &\begin{cases}
        \frac{1}{n - 1 - \alpha} \P \left( \That_{\tildei - 1} = \hat{\rho}^* \left( \that, [n] \setminus [\tildei - 1] \right) \right) & \text{if $1 \leq i < \tildei \leq n$} \\[0.2cm]
        \frac{\tildei - 2 - \alpha}{n - 1 - \alpha} \P \left( \That_{\tildei - 1} = \hat{\rho}^* \left( \that, [n] \setminus [\tildei - 1] \right) \right) & \text{if $2 \leq i = \tildei \leq n$}
      \end{cases}
    \end{align*}
    for all $\that \in \Thatspace_{[n]}$, showing that $\That_{\tildei - 1} \independent E_{i, \tildei}$ for all fixed $1 \leq i \leq \tildei \leq n$. \\

    For the second assertion we start by noting that
    \begin{align*}
      \P \left( \That_{\tildei - 1} \in \cdot \ \middle \vert \ E_{i,\tildei} \right) = \P \left( \That_{\tildei - 1} \in \cdot \right)
    \end{align*}
    due to the independence proven above.
    Once again using the characterization of $E_{i, \tildei}$, we note that we can construct $\hat{\rho}^* \left( \hat{\tau}^* \left( \That_n, i, \tildei \right), \tildei \right)$ from $\That_{\tildei - 1}$ under the conditional law of $E_{i, \tildei}$:
    The leaf $\tildei$ is never attached and so the ``forbidden edges'', characterized above and coloured red in Figure~\ref{fig:ForbiddenEdges_1}, do not exist.
    But every leaf that was previously labelled $\tildei+1, \ldots, n$ will still be attached according to the $(\alpha, \gamma)$-growth rule since the total weight of the forbidden edges are exactly $1$ in all cases.
    This finishes the proof.
\end{proof}
\begin{thm}%
\label{thm:planarstationarity}
  Fix $\nin$.
  Let ${\left( \That_n(m) \right)}_{m \in \N_0}$ denote the semi-planar $(\alpha, \gamma)$-chain and let $\That_n$ denote the $n$th step of the semi-planar $(\alpha, \gamma)$-growth process.
  Then it holds that $\That_n(m) \dcon \That_n$ for $m \to \infty$, where $\dcon$ denotes convergence in distribution.
  Furthermore, it holds that the projection of $\That_n$ to the space of (non-planar) $[n]$-trees has the same distribution as the $n$th step of the $(\alpha, \gamma)$-growth process.
\end{thm}
\begin{proof}
    Fix $n \in \N$.
    That the semi-planar $(\alpha, \gamma)$-chain on $\Thatspace_{[n]}$ has a unique stationary distribution is easily seen, since the state space is finite and there is a unique, recurrent communicating class for
    \begin{itemize}
        \item $\alpha = 1$ and all $\gamma \in (0,1)$ (only insertions in branch points or internal edges),
        \item $\gamma = 0$ and all $\alpha \in (0,1)$ (only insertions in branch points or leaf edges),
        \item $\gamma = \alpha$ and all $\alpha \in [0,1]$ (binary trees),
        \item $\alpha = 1$ and $\gamma = 0$ (only insertions in branch points),
    \end{itemize}
    whilst the Markov chain is irreducible otherwise.
    Thus we are done if we can show that $\That_n(1) \deq \That_n(0) := \That_n$, which we will do by utilizing Lemma~\ref{lemma:condind}.
    Let $I$ be a uniform random variable on $[n]$, independent of everything else, let $f$ denote the transformation function from Definition~\ref{def:planaralphagammachain} and note that for all $\hat{\tree} \in \Thatspace_{n-1}$ 
    \begin{align*}
      &\P \left( \hat{\rho}^* \left( f \left(\That_n, I \right) \right) = \hat{\tree} \right) \\
      =\quad &\sum_{1 \leq i \leq \tildei \leq n} \P \left( \hat{\rho}^* \left( \hat{\tau}^* \left( \That_n, i, \tildei \right), \tildei \right) = \hat{\tree} \ \middle \vert \ E_{i, \tildei} \cap \{ I = i \} \right) \P \left( E_{i, \tildei} \cap \{ I = i \} \right) \\
      =\quad &\sum_{1 \leq i \leq \tildei \leq n} \P \left( \That_{n-1} = \hat{\tree} \right) \P \left( E_{i, \tildei} \cap \{ I = i \} \right) \\
      =\quad &\P \left( \That_{n-1} = \hat{\tree} \right)
    \end{align*}
    where we have used the independence of $I$, the definition of $f$ as well as Lemma~\ref{lemma:condind}.
    This shows that $\hat{\rho}^* \left( f \left( \That_n, I \right) \right) \deq \That_{n-1}$, and since $\That_n(1)$ is constructed from the former using the same growth rule used to construct $\That_n$ from the latter, this finishes our proof.
\end{proof}
%
%%%%%%%%%%%%%%%%%%%%%%%%%%%%%%%%%%%%%%%%%%%%%%%%%%%%%%%%%%%%%%%%%%%%%%%%%%%%%%%%%%%%
%%%%%%%%%%%%%%%%%%%%%%%%%%%% Non-planar down-up chain %%%%%%%%%%%%%%%%%%%%%%%%%%%%%%
%%%%%%%%%%%%%%%%%%%%%%%%%%%%%%%%%%%%%%%%%%%%%%%%%%%%%%%%%%%%%%%%%%%%%%%%%%%%%%%%%%%%
%
\section{Projecting to a non-planar down-up chain}
In the previous section we constructed a semi-planar down-up chain on $\Thatspace_{[n]}$ with the distribution of $\That_n$ as its stationary distribution, where $\That_n$ denotes the $n$th step of the semi-planar $(\alpha, \gamma)$-growth process of Definition~\ref{def:planaralphagamma}.
Recalling that $\hat{\pi} \colon \Thatspace_{[n]} \to \T_{[n]}$ denotes the projection map from the space of semi-planar trees to that of non-planar trees, we have already noted that $\hat{\pi} ( \That_n ) \deq T_n$, where the latter denotes the $n$th step of the $(\alpha, \gamma)$-growth process of Definition~\ref{def:growthprocess_alphagamma}.
Now let us construct a non-planar down-up chain on $\T_{[n]}$ from the semi-planar version.

Fix $\nin$.
Let $\hat{K}_n$ denote the transition kernel of the semi-planar $(\alpha, \gamma)$-chain on $\Thatspace_{[n]}$ and define a $\hat{\pi}$-induced Markov kernel from $\T_{[n]}$ to $\Thatspace_{[n]}$ by
\begin{align}
  \hat{\Pi}_n \left( \tree, \left\{ \that \right\} \right)
  := \P \left( \That_n = \that \ \middle \vert \ \hat{\pi} \left( \That_n \right) = \tree \right)
  = \prod_{v \in \branchpoints(\tree)} p^{c_v-2}_{\alpha, \alpha - \gamma} \left( \sigma_v \right),
  \label{eq:MarkovKernel}
\end{align}
for each $\that \in \Thatspace_{[n]}$ and $\tree \in \T_{[n]}$, where $c_v$ denotes the number of children of a branch point $v \in \branchpoints (\tree)$ and $\sigma_v$ is the permutation of $[c_v - 2]$ associated with the same branch point in $\that$.
The latter expression stems from Lemma~\ref{lemma:internalstructure} and~\eqref{eq:condprop_planargivenshape}.
Let us abuse notation, so that $\hat{\Pi}_n$ also denotes the induced Markovian matrix $\hat{\Pi}_n := {\left( \hat{\Pi}_n \left( \tree, \left\{ \that \right\} \right) \right)}_{\tree \in \T_{[n]}, \that \in \Thatspace_{[n]}}$, whilst $\hat{\pi} := {\left( 1_{\hat{\pi}(\that) = \tree} \right)}_{\that \in \Thatspace_{[n]}, \tree \in \T}$ denotes the matrix linking a semi-planar tree to its non-planar counterpart.
We now give an alternative definition of the $(\alpha, \gamma)$-chain on $\T_{[n]}$ which we identify with Definition~\ref{def:alphagamma_chain} in Proposition~\ref{prop:nonplanar_alphagamma_chain}:
\begin{defi}[$(\alpha, \gamma)$-chain]
  \label{def:alpha_gamma_chain}
  Fix $\nin$.
  The $(\alpha, \gamma)$-\textit{chain on} $\T_{[n]}$ is a Markov chain on $\T_{[n]}$ with transition kernel
  \begin{align}\label{eq:nonplanartransitionmatrix}
    K_n := \hat{\Pi}_n \hat{K}_n \hat{\pi}.
  \end{align}
\end{defi}
The transition kernel $K_n$ is defined so that the diagram in Figure~\ref{fig:commutativediagram} commutes, and a more detailed explanation of the transitions of the $(\alpha, \gamma)$-chain, ${\left( T_n(m) \right)}_{m \in \N_0}$, is consequently the following:
\begin{enumerate}
  \item Fix $m \in \N_0$.
    If $T_n(m) = \tree \in \T_{[n]}$, pick a semi-planar tree, $\that \in \Thatspace_{[n]}$, according to $\hat{\Pi}_n \left( \tree, \cdot \right)$ defined in~\eqref{eq:MarkovKernel}.
  \item Make one transition of the semi-planar $(\alpha, \gamma)$-chain started from $\that$, to obtain $\that^\prime$.
  \item Project $\that'$ to the space of non-planar trees, i.e.\ set $T_n(m+1) = \hat{\pi} (\that^\prime)$.
\end{enumerate}
The following result ensures that the two ways of defining the $(\alpha, \gamma)$-chain (Definitions~\ref{def:alphagamma_chain} and~\ref{def:alphagamma_chain}) are compatible.
\begin{prop}%
\label{prop:nonplanar_alphagamma_chain}
Definition~\ref{def:alphagamma_chain} and Definition~\ref{def:alpha_gamma_chain} are equivalent.
\end{prop}
\begin{proof}
  Before we consider the down-step itself, note that it does not make any difference if we perform the up-step for the semi-planar $(\alpha, \gamma)$-chain and then project to the non-planar trees, or if we project to the non-planar trees immediately after the down-step, and then perform the up-step by using the non-planar $(\alpha, \gamma)$-growth process, see the brief discussion following Definition~\ref{def:planaralphagamma}.
  This ensures, that we in the following only have to argue that the down-step outlined in Steps 1 and 2 is the correct characterization. 

  So fix $\tree \in \T_{[n]}$ and sample $\That$ from $\hat{\Pi}_n \left( \tree, \cdot \right) = \P \left( \That_n \in \cdot \ \middle \vert \ \hat{\pi} \left( \That_n \right) = \tree \right)$.

  If $i$ is in a binary branch point in $\tree$, the same is the case in $\That$.
  Hence the characterization of $\tilde{I}$ is trivial, as it is exactly the same here as in Definition~\ref{def:planaralphagammachain}.

  If $i$ is in a multifurcating branch point with $c_v > 2$ children in $\tree$, we need to split up into cases based on the location of $i$ within the branch point in $\That$.
  So let $v$ denote the parent branch point of $i$, let $i_1, \ldots, i_{c_v}$ denote the smallest leaf labels in the $c_v$ subtrees of $v$, enumerated in increasing order, and say that $i = i_j$ for some $j \in [c_v]$.
  This is all deterministic based on $\tree$ and $i$.
  Furthermore, recall the definition of $a$ and $b$ from Definition~\ref{def:planaralphagammachain}.

  If $i$ is placed leftmost in $v$ in $\That$, i.e.\ if $j \in [2]$, then the $b$ of Definition~\ref{def:planaralphagamma} is the smallest leaf to the right of the two leftmost subtrees, which will always be the largest of $i, a$ and $b$, implying that $\tildei = \max \{i, a, b\} = b$.
  Hence, conditional on $T_n(m) = \tree$ and $I = i$, the probability of $\tilde{I} = i_{j'}$ for some $2 < j' \leq c_v$ is the probability that $\That$ has the subtree containing $i_{j'}$ placed to the immediate right of $i$ in $v$.
  Considering how these subtrees were formed by the growth rule in the ordered Chinese Restaurant analogy outlined in Lemma~\ref{lemma:internalstructure}, this is the probability having table $j' - 2$, enumerated in the order of least element, placed leftmost in an $(\alpha, \alpha - \gamma)$-$\ocrp$ with $c_v - 2$ tables.

  If $i$ is not places leftmost in $v$ in $\That$, i.e.\ if $2 < j \leq c_v$, $b$ is the smallest leaf in the tree to the left of $i$ in $v$.
  In this case we always find that $\tildei = \max \{i, b\}$ as $a$ by definition is equal to $i_1$.
  Consequently, conditional on $T_n(m) = \tree$ and $I = i$, the probability of $\tilde{I} = i = i_j$ stems from the probability that $\That$ has one of the subtrees containing $i_1, \ldots, i_{j-1}$, respectively, placed to the immediate left of $i$ in $v$.
  Hence in the ordered Chinese Restaurant analogy, this corresponds to not opening a new table to the immediate left of table $j-2$, enumerated in the order of least element, in an $(\alpha, \alpha - \gamma)$-$\ocrp$ with $c_v - 2$ tables.
  Similarly the probability of getting $\tilde{I} = i_{j'}$ for any $j < j' \leq c_v$ corresponds to having table $j'-2$ to the immediate left of table $j-2$ in the same ordered Chinese Restaurant.
  This defines the conditional distribution of $\tilde{I}$ outlined above.
\end{proof}
\begin{figure}[t]
  \centering
  \begin{tikzcd}
    \Thatspace_{[n]} \arrow[r, "\hat{K}_n"] & \Thatspace_{[n]} \arrow[d, "\hat{\pi}"] \\
    \T_{[n]} \arrow[u, "\hat{\Pi}_n"] \arrow[r, red, "K_n" red] & \T_{[n]}
  \end{tikzcd}
  \caption{Commutative diagram to illustrate the construction of the transition kernel for the $(\alpha, \gamma)$-chain.}\label{fig:commutativediagram}
\end{figure}
The following result now follows immediately from Theorem~\ref{thm:planarstationarity}.
\begin{prop}\label{prop:stationaryMCnonplanar}
  Fix $\nin$ and let $T_n$ denote the $n$th step of the $(\alpha, \gamma)$-growth process.
  The $(\alpha, \gamma)$-chain on $\T_{[n]}$ started from $T_n$ is stationary.
\end{prop}
\begin{proof}
  First note that the distribution of $T_n$ is pushed forward to the distribution of $\That_n$ by $\hat{\Pi}_n$, where $\hat{\Pi}_n$ is defined by~\eqref{eq:MarkovKernel} and $\That_n$ denotes the $n$th step of the semi-planar $(\alpha, \gamma)$-growth process.
  By Theorem~\ref{thm:planarstationarity} the semi-planar $(\alpha, \gamma)$-chain started from $\That_n$ is stationary, so if ${\left( \That_n(m) \right)}_{m \in \N_0}$ denotes the semi-planar $(\alpha, \gamma)$-chain started from $\That_n$, then $\That_n(1) \deq \That_n(0) \deq \That_n$.
  Lastly, by construction of the semi-planar $(\alpha, \gamma)$-growth process, we have that $\hat{\pi} ( \That_n ) \deq T_n$, and so $T_n(1) \deq T_n(0) \deq T_n$.
\end{proof}
Proposition~\ref{prop:stationaryMCnonplanar} is a fairly easy result to obtain due to the construction of the $(\alpha, \gamma)$-chain on $\T_{[n]}$ and the results regarding the semi-planar $(\alpha, \gamma)$-chain at our current disposal.
However, the link between the two down-up chains is somewhat deeper than what this result eludes to.
\begin{thm}\label{thm:nonplanar_alphagamma_chain}
  Fix $\nin$.
  Let ${\left( \That_n(m) \right)}_{m \in \N_0}$ be a semi-planar $(\alpha, \gamma)$-chain, and let $T_n(m) := \hat{\pi} \left( \That_n(m) \right)$ denote the projection of the chain onto $\T_{[n]}$.
  If $\That_n(0) \sim \hat{\Pi}_n(\tree, \cdot)$ for some $\tree \in \T_{[n]}$, then ${\left(T_n(m)\right)}_{m \in \N_0}$ is the $(\alpha, \gamma)$-chain started from $\tree$ with stationary distribution being that of $T_n$, the $n$th step of the $(\alpha, \gamma)$-growth process.
\end{thm}
Theorem~\ref{thm:nonplanar_alphagamma_chain} is substantially deeper result than both Propositions~\ref{prop:nonplanar_alphagamma_chain} and~\ref{prop:stationaryMCnonplanar}, but will defer the proof as we have yet to develop the methodology to prove it, and as it is a special case of Proposition~\ref{prop:decorated_transition_kernel} and Theorem~\ref{thm:decorated_chain_projection_theorem}.
%
%We will finish this section with an example of what can happen in the down-step of the $(\alpha,\gamma)$-chain.
%%
%\begin{ex}
%  Fix $t \in \T_{[n]}$ and say that the leaf $i \in [n]$, selected for deletion, is located rightmost in a bush with more than two subtrees in the semi-planar tree $\that$ sampled from $\Pi(\tree, \cdot)$, described by~\eqref{eq:MarkovKernel}.
%  Now, observe the location of $a$ and $b$ from Definition~\ref{def:planaralphagammachain}.
%  By definition, $a$ will be located in the leftmost subtree of the first bush on the ancestral line from $i$, whilst be will be located in one of the two leftmost subtrees of the second bush on the ancestral line from $i$.
%  Now assume that $i = \tildei$, which is the case if and only if all leaves in the second bush on the ancestral line from $i$ have labels larger than $i$.
%  Then carrying it the down-step will relocate $i$ to the second bush, even though there might be larger labels in the first bush.
%  This situation is depicted in Figure~\ref{fig:decoratedtrees_downstep}.
%\end{ex}
%
\section{A Markov Chain on decorated trees}%
\label{sec:MC_decorated}
So far we have introduced a down-up Markov chain on semi-planar trees, and shown that there is a corresponding Markov chain on the space of (non-planar) trees.
We now continue our endeavour to define down-up chains on various spaces of trees that arise as projections of $[n]$-trees.
To make this precise, we introduce the notion of decorated trees, as well as an intermediary known as collapsed trees, which arise from affiliating every element of $\insertablef$, the collection of edges and branch points of the tree $\tree \in \T_{[n]}$, with an integer or a set of integers, respectively.
\begin{defi}[Decorated Trees]\label{def:decoratedtrees}
  Let $A$ be a finite set, and fix some $\# A \leq n \in \N$.
  Let $\tree[s] \in \T_A$ be given, and let $\mathbf{y} = {\left( y_x \right)}_{x \in \insertablef[s]}$ be a collection of non-negative integers such that $\sum_x y_x = n$ where $y_x \geq 1$ for every external $x \in \edge (\tree[s])$.
  We call $\tree^\bullet = (\tree[s], \mathbf{y})$ a \textit{decorated $A$-tree of mass $n$}, and we denote the space of such trees by $\T_A^{\bullet n}$, and will refer to $\tree[s]$ as the \textit{tree shape} of $\tree$.
\end{defi}
For any finite sets $A \subseteq B$ with $\#B =: n$ we can easily obtain a decorated $A$-tree of mass $n$ from any $\tree \in \T_B$ by projection onto the subtree spanned by the leaves with labels in $A$.
\fxnote{Matthias: How does $\tree[s] \subseteq \tree$ notation match with e.g.~\cite{RefWorks:doc:5b6c561fe4b06c0731a5c558}? Relate to the literature in general.}
Let us explain this more formally.
In the remainder of this exposition we will use the notation, $\tree[s] \subseteq \tree \in \T_B$, to mean that there exists a set $A \subseteq B$ such that $\rho (\tree, B \setminus A) = \tree[s]$.
Consequently, we will insist that $\vertices(\tree[s]) \subseteq \vertices(\tree)$ in the usual sense, and hence it is meaningful to talk about a branch point of $\tree[s]$ being on the ancestral line from a leaf in $\tree$.
Note however that, unless we are in the degenerate situation where $A = B$ or where all leaves labelled by $B \setminus A$ have a parent branch point of $\tree[s]$ in $\tree$, we will have $\edge(\tree[s]) \nsubseteq \edge(\tree[t])$.

So fix $\tree \in \T_B$, and define $\tree[s] := \mathring{\rho} \left( \tree, B \setminus A \right)$, where we recall that the latter means that we delete all leaves with labels of $B$ not in $A$ from $\tree$, see Definition~\ref{def:leaf_deletion}.
Recalling that $u \prec_{\tree} v$ denotes that $u$ is the parent of $v$ in $\tree$, we now define the \textit{decoration function} $g \colon B \to \tree[s] \subseteq \tree$ by setting 
\begin{itemize}
  \item $g(i) = v \in \branchpoints (\tree[s])$ if $i$ is in a subtree of $v$ in $\tree$ which only contains leaves labelled by elements of $B \setminus A$, and
  \item $g(i) = e_{v_0v_m} \in \edge (\tree[s])$ if there exist $m \geq 2$, vertices $v_0, v_m \in \vertices (\tree[s])$ and $v_1, \ldots, v_{m-1} \in \branchpoints(\tree[t]) \setminus \branchpoints(\tree[s])$ such that for all $j \in [m]$ it holds that $v_{j-1} \prec_{\tree} v_j$, $v_0 \prec_{\tree[s]} v_m$, and that there is a $j \in [m-1]$ such that $i$ is in a subtree of $v_j$ in $\tree$ which only contains leaves labelled by elements of $B \setminus A$.
\end{itemize}
Then define
\begin{align}
  y_x = \sum_{i = 1}^n \delta_{x g(i)} = \# g^{-1}(\{x\}) = \# \left\{ i \in [n] \ \middle \vert \ g(i) = x \right\}
  \label{eq:decoratedweights}
\end{align}
for each $x \in \insertablef[s]$, where $\delta_{xy}$ is the Kronecker delta, and define $\mathbf{y} = {(y_x)}_{x \in \insertablef[s]}$.
Finally, define $\tree^\bullet = (\tree[s], \mathbf{y})$.
For each $k \leq n \in \N$, we will denote the map $\T_{[n]} \ni \tree \mapsto \tree^\bullet = (\tree[s], \mathbf{y}) \in \T_{[k]}^{\bullet n}$, described above, by $\pi_{[k]}^{\bullet n}$.
This construction yields another family of trees as well:
\begin{defi}[Collapsed Trees]
  \label{def:collapsed_trees}
  Let $A$ be a finite set, and fix some $\# A \leq n \in \N$.
  With $g$ being the decoration function, we denote the map that sends $\tree \in \T_{[n]}$ to its associated \textit{collapsed $A$-tree of mass $n$} by $\pi_{A}^{*n}$:
  \begin{align}
    \pi_A^{*n} ( \tree ) = \left( \tree[s], {g^{-1}(\{x\})}_{x \in \insertablef[s]} \right),
    \label{eq:collapsedtrees_associated}
  \end{align}
  where $\tree[s] = \rho \left( \tree, [n] \setminus A \right)$.
  We denote the space of such trees by
  \begin{align}
    \T_A^{*n} :=
  \left\{ \pi_{A}^{* n} \left( \tree \right) \bigg \vert \ \tree \in \T_{[n]} \right\}.
    \label{eq:collapsedtrees}
  \end{align}
\end{defi}
The aim of this section is similar to that of the previous one, where we studied a Markov chain on $\T_{[n]}$ by using a kernel to lift an element $\tree \in \T_{[n]}$ to $\Thatspace_{[n]}$.
Paraphrasing, we sampled the order of the subtrees within each branch point of $\tree$, then carried out both the down- and up-step, respectively, of the semi-planar $(\alpha, \gamma)$-chain on $\Thatspace_{[n]}$, and finally projected back to $\T_{[n]}$.

Constructing a Markov chain on decorated trees corresponding to the semi-planar $(\alpha, \gamma)$-chain on $\Thatspace_{[n]}$, is done by the following procedure.
Fix $\tree^\bullet = (\tree[s], \mathbf{y}) \in \T_{[k]}^{\bullet n}$.
Similar to the construction of ${\hat{\Pi}}_n$ in~\eqref{eq:MarkovKernel}, we will construct the kernel ${\hat{\Pi}}_k^{\bullet n}$, by setting
\begin{align}
  \forall \tree^\bullet \in \T_{[k]}^{\bullet n}, \ \that \in \Thatspace_{[n]} \colon \quad
  {\hat{\Pi}}_k^{\bullet n} (\tree^\bullet, \{\that\}) :=
  \P \left( \That_n = \that \ \middle \vert \ \hat{\pi}_{[k]}^{\bullet n} \left(\That_n\right) = \tree^\bullet \right),
  \label{eq:decoratedkernel}
\end{align}
where $\That_n$ is the $n$th step of the semi-planar $(\alpha, \gamma)$-growth process, and $\hat{\pi}_{[k]}^{\bullet n} := \pi_{[k]}^{\bullet n} \circ \hat{\pi}$ denotes the  projection from semi-planar $n$-trees to decorated $[k]$-trees of mass $n$, for brevity.
Define a transition kernel on $\T_{[k]}^{\bullet n}$ by 
\begin{align}
  K_k^{\bullet n} := {\hat{\Pi}}_k^{\bullet n} \hat{K}_n \hat{\pi}_{[k]}^{\bullet n},
  \label{eq:transitionkernel_decorated}
\end{align}
where $\hat{K}_n$ is the transition kernel for the semi-planar $(\alpha, \gamma)$-chain and we have used the same notation as in~\eqref{eq:nonplanartransitionmatrix} for the composition of kernels.
\begin{defi}[Decorated $(\alpha, \gamma)$-chain]\label{def:decorated_alphagamma_chain}
  Fix $n > k \in \N$ and let $K_k^{\bullet n}$ denote the transition kernel defined by~\eqref{eq:transitionkernel_decorated}.
The \textit{decorated $(\alpha, \gamma)$-chain on} $\T_{[k]}^{\bullet n}$ is a Markov chain on $\T_{[k]}^{\bullet n}$ with transition kernel $K_k^{\bullet n}$.
\end{defi}
\begin{rem}
  This is equivalent to the Definition~\ref{def:decorated_alphagamma_chain_intro}.
\end{rem}
The goal of the remainder of this exposition is two-fold.
On one hand we wish to show that if we project semi-planar $(\alpha, \gamma)$-chains of Definition~\ref{def:planaralphagammachain} to the space of decorated $[k]$-trees of mass $n$, we obtain the decorated $(\alpha, \gamma)$-chain of Definition~\ref{def:decorated_alphagamma_chain}.
On the other hand we wish to describe the transition kernel of the decorated $(\alpha, \gamma)$-chain without referencing the semi-planar transition kernel.
It turns out that the latter is much easier, and so we will take on this task first. 

So consider the behaviour of the semi-planar $(\alpha, \gamma)$-chain projected to decorated trees, i.e.\ for $k < n$ consider observing ${\left( \hat{\pi}_{[k]}^{\bullet n} \left( \That_n(m) \right) \right)}_{m \in \N_0}$, and recall that the semi-planar $(\alpha, \gamma)$-chain is characterized by composing two kernels in a down- and an up-step.
The focus in this section is to specify similar steps on the space of decorated trees, so that doing a down- or up-step, respectively, in the semi-planar $(\alpha, \gamma)$-chain and then projecting to a decorated tree, is the same as projecting to a decorated tree and then performing a down- or up-step, respectively.
It turns out that the down-step is significantly more complicated than the up-step, and so we start with doing the latter.
\begin{defi}[$(\alpha, \gamma)$-growth process for decorated trees]
\label{def:decorated_alphagamma_growth}
  Fix $k \in \N$, fix $\tree[s] \in \T_{[k]}$, and let $T_k^{\bullet k}$ be the unique element of $\T_{[k]}^{\bullet k}$ with tree shape $\tree[s]$.
  For any $k \leq n \in \N$ construct $T_k^{\bullet (n+1)}$ conditional on $T_k^{\bullet n} = (\tree[s], \mathbf{y}) \in \T_{[k]}^{\bullet n}$ by, for each $x \in \insertablef[s]$, setting
  \begin{align*}
    &\P \left( T_k^{\bullet (n+1)} = (\tree[s], \mathbf{y} + \mathbf{1}_x) \ \middle \vert \ T_k^{\bullet n} = (\tree[s], \mathbf{y}) \right) \\
    =\ 
    &\begin{cases}
      \frac{y_x - \alpha}{n - \alpha} & \text{if $x \in \edge(\tree[s])$ is external,} \\
      \frac{y_x + \gamma}{n - \alpha} & \text{if $x \in \edge(\tree[s])$ is internal,} \\
      \frac{y_x + (c - 1) \alpha - \gamma}{n - \alpha} & \text{if $x \in \branchpoints(\tree[s])$ with $c$ children,}
      \end{cases}
  \end{align*}
  for all decorated trees $(\tree[s], \mathbf{y}) \in \T_{[k]}^{\bullet n}$ where $\mathbf{1}_x = {\left( \delta_{xx'} \right)}_{x' \in \insertablef[s]}$.
We will call ${\left( T_k^{\bullet n} \right)}_{n \geq k}$ \textit{the decorated $(\alpha, \gamma)$-growth process started from} $T_k^{\bullet k} = \tree[s]$.
\end{defi}
We immediately note that the above is simply a P\'{o}lya urn scheme with initial weights labelled by the elements of $\insertablef[s]$.
Specifically, the initial weights will be $1-\alpha$ for an external edge, $\gamma$ for an internal edge, and $(c-1)\alpha - \gamma$ for a branch point with $c$ children in $\tree[s]$.
This ensures that if ${ \left( T_n \right) }_\nin$ is the $(\alpha, \gamma)$-growth process then
\begin{align*}
  \P \left( \pi_{[k]}^{\bullet n} (T_n) = \cdot \ \middle \vert \ T_k = \tree[s] \right)
  = \P \left( T_k^{\bullet n} = \cdot \right)
\end{align*}
for all $n \geq k$.
Additionally it holds that
\begin{align}
  \P \left( \hat{\pi}_{[k]}^{\bullet (n+1)} \left( \That_{n+1} \right) \in \cdot \ \middle \vert \ \That_n = \that \right)
  = \P \left( T_k^{\bullet (n+1)} \in \cdot \ \middle \vert \ T_k^{\bullet n} = \hat{\pi}_{[k]}^{\bullet n} (\that) \right)
  \label{eq:decoupling_upstep}
\end{align}
for all $\that \in \Thatspace_{[n]}$.
Consider again observing ${\left( \hat{\pi}_{[k]}^{\bullet n} \left( \That_n(m) \right) \right)}_{m \in \N_0}$ and let $\That_n^\downarrow(m)$ denote the semi-planar chain after the down-step.
Then \eqref{eq:decoupling_upstep} ensures that we end up with the same distribution on the space of decorated $[k]$-trees of mass $n$ by either
\begin{enumerate}
  \item projecting $\That_n^\downarrow(m)$ to the space of decorated $[k]$-trees of mass $n-1$, and then performing the up-step from the $(\alpha, \gamma)$-growth process for decorated trees, or
  \item performing the up-step from the semi-planar $(\alpha, \gamma)$-growth process, and then projecting the resulting tree to the space of decorated $[k]$-trees of mass $n$.
\end{enumerate}
What is more technical is how to construct a down-step on the space of decorated trees that is consistent with the down-step defined in the semi-planar $(\alpha, \gamma)$-chain.
For intuition, let us consider a semi-planar $(\alpha, \gamma)$-chain, ${\left( \That_n(m) \right)}_{m \in \N_0}$, and assume that, before carrying out the down-step, $\That_n(m) = \that$ and $\hat{\pi}_{[k]}^{\bullet n} ( \that )= (\tree[s], \mathbf{y}) \in \T_{[k]}^{\bullet n}$.
Now observe how there are three separate scenarios for the deletion of a leaf in the down-step:
\begin{itemize}
  \item We select a leaf with label $i > k$ in $\that$ that contributes to exactly one of the decorated masses $y_x$ for some $x \in \insertablef[s]$, in which case we reduce that mass by one, but the tree shape $\tree[s]$ and all other decorating masses stay unchanged.
  \item We initially select a leaf with label $i \in [k]$ in $\that$, search for $a$ and $b$ as defined in Definition~\ref{def:planaralphagammachain}, and define $\tildei = \max\{i, a, b\}$:
    \begin{itemize}
      \item If $\tildei \leq k$ we will swap $i$ and $\tildei$, delete $\tildei$, and relabel $(\tildei+1, \ldots, k+1)$ by $(\tildei, \ldots, k)$ using the increasing bijection.
        Depending on which $y_x$ the leaf labelled by $k+1$ contributed to in $(\tree[s], \mathbf{y})$, this can result in the tree shape changing.
        The decorating masses need to change accordingly.
      \item If $\tildei > k$ we will retain the same tree shape, irrespective of the location of $i$ in $\that$ prior to the swap.
        Indeed, if $i$ was located in a  multifurcating branch point of $\that$, our search rule guarantees that $\tildei$ is found in a subtree of the same branch point.
        If $i$ was located in a binary branch point of $\that$, $\tildei$ will either be the smallest label in the other subtree of that branch point, or be the smallest leaf in the second spinal bush on the ancestral line from $i$, but due to the deletion of the binary branch point after the swap, this will not change the tree shape of the projected tree.
        See Figure~\ref{fig:decoratedtrees_downstep} for an illustration. 
    \end{itemize}
\end{itemize}
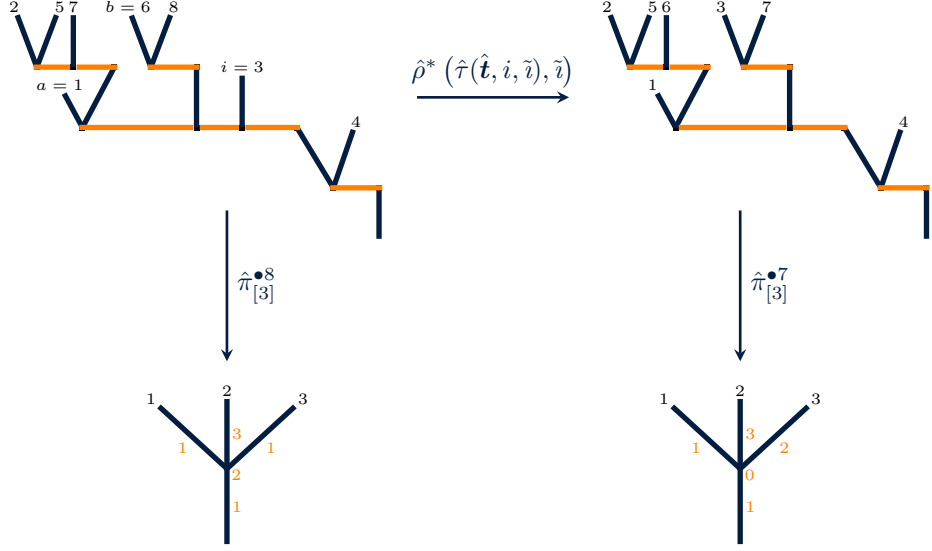
\begin{figure}[t]
  \centering
  \begin{tikzpicture}[
    bp/.style={rectangle, fill, inner sep=0pt, minimum width=2pt, minimum height=2pt},
    comment/.style={circle, draw, color=orange},
    leaf/.style={inner sep=1pt},
    weights/.style={auto, swap, orange, inner sep = 1pt},
    label/.style={text=oxfordblue},
    line/.style={-, line width = 2pt, color = oxfordblue},
    arrowline/.style={->, line width = 1pt, color = oxfordblue, >=stealth},
    scale=1
  ]
  \tikzstyle{every node}=[font=\tiny]

  %
  % Semi-planar version
  %

  \begin{scope}[xscale = -0.6, yscale = 0.8]

    \node (0) at (0,0) {};

    \node[bp] (bp_1) at (0,1) {};
    \node[bp] (bp_2) at (1,1) {};
    \node[leaf] (4) at (0.5,2.1) {$4$};

    \node[bp] (v_1) at (1.8,2) {};
    \node[bp] (v_2) at (3,2) {};
    \node[bp] (v_3) at (4,2) {};
    \node[bp] (v_4) at (6.5,2) {};

  % 1 subtree (from left) (v_1)
    \node[bp] (bp2_1) at (4,3) {};
    \node[bp] (bp2_2) at (5,3) {};
    \node[leaf] (8) at (4.5, 4) {$8$};
    \node[leaf] (6) at (5.5, 4) {$b = 6$};

  % 2 subtree (from left) (v_2)
    \node[leaf] (3) at (3,3) {$i = 3$};

  % 3 subtree (from left) (v_3)
    \node[bp] (bp3_1) at (5.8,3) {};
    \node[bp] (bp3_2) at (6.7,3) {};
    \node[bp] (bp3_3) at (7.5,3) {};
    \node[leaf] (7) at (6.7, 4) {$7$};
    \node[leaf] (5) at (7, 4) {$5$};
    \node[leaf] (2) at (8, 4) {$2$};

  % 4 subtree (from left)
    \node[leaf] (1) at (7,2.7) {$a = 1$};

    \draw[line] (0) to (bp_1);
    \draw[line] (bp_2.center) to (4);

    \draw[line] (bp_2.center) to (v_1.center);
    \draw[line] (v_2.center) to (3);
    \draw[line] (bp2_2.center) to (8);
    \draw[line] (bp2_2.center) to (6);

    \draw[line] (v_3.center) to (bp2_1.center);

    \draw[line] (v_4.center) to (bp3_1.center);
    \draw[line] (bp3_2.center) to (7);
    \draw[line] (bp3_3.center) to (5);
    \draw[line] (bp3_3.center) to (2);

    \draw[line] (v_4.center) to (1);

    \draw[line, orange] (bp_1.east) to (bp_2.west);
    \draw[line, orange] (v_1.east) to (v_2) to (v_3) to (v_4.west);
    \draw[line, orange] (bp2_1.east) to (bp2_2.west);
    \draw[line, orange] (bp3_1.east) to (bp3_2) to (bp3_3.west);
  \end{scope}
  
  %
  % Semi-planar version after down-step
  %

  \begin{scope}[shift = {(7.2,0)}, xscale = -0.6, yscale = 0.8]

    \node (0) at (0,0) {};

    \node[bp] (bp_1) at (0,1) {};
    \node[bp] (bp_2) at (1,1) {};
    \node[leaf] (3) at (0.5,2.1) {$4$};

    \node[bp] (v_1) at (1.8,2) {};
    \node[bp] (v_2) at (3,2) {};
    \node[bp] (v_3) at (5.5,2) {};

  % 1 subtree (from left) (v_1)
    \node[bp] (bp2_1) at (3,3) {};
    \node[bp] (bp2_2) at (4,3) {};
    \node[leaf] (8) at (3.5, 4) {$7$};
    \node[leaf] (6) at (4.5, 4) {$3$};

  % 2 subtree (from left) (v_3)
    \node[bp] (bp3_1) at (4.8,3) {};
    \node[bp] (bp3_2) at (5.7,3) {};
    \node[bp] (bp3_3) at (6.5,3) {};
    \node[leaf] (7) at (5.7, 4) {$6$};
    \node[leaf] (5) at (6, 4) {$5$};
    \node[leaf] (2) at (7, 4) {$2$};

  % 4 subtree (from left)
    \node[leaf] (1) at (6,2.7) {$1$};

    \draw[line] (0) to (bp_1);
    \draw[line] (bp_2.center) to (3);

    \draw[line] (bp_2.center) to (v_1.center);
    \draw[line] (bp2_2.center) to (8);
    \draw[line] (bp2_2.center) to (6);

    \draw[line] (v_2.center) to (bp2_1.center);

    \draw[line] (v_3.center) to (bp3_1.center);
    \draw[line] (bp3_2.center) to (7);
    \draw[line] (bp3_3.center) to (5);
    \draw[line] (bp3_3.center) to (2);

    \draw[line] (v_3.center) to (1);

    \draw[line, orange] (bp_1.east) to (bp_2.west);
    \draw[line, orange] (v_1.east) to (v_2) to (v_3);
    \draw[line, orange] (bp2_1.east) to (bp2_2.west);
    \draw[line, orange] (bp3_1.east) to (bp3_2) to (bp3_3.west);
  \end{scope}

  %
  % First decorated tree
  %

  \begin{scope}[shift = {(-2,-4)}]
    \node (0) at (0,0) {};
    \node (v) at (0,1) {};
    \node[leaf] (1) at (-1,2) {$1$};
    \node[leaf, anchor = south] (2) at (0, 2) {$2$};
    \node[leaf] (3) at (1, 2) {$3$};

    \draw[line] (0.mid) to node[weights] {$1$} (v.mid) to node[weights, swap] {$1$} (1);
    \draw[line] (v.mid) to node[weights] {$3$} (2);
    \draw[line] (v.mid) to node[weights] {$1$} (3);
    \node[orange] at (v.east) {$2$};
  \end{scope}
  
  %
  % Second decorated tree
  %

  \begin{scope}[shift = {(4.75,-4)}]
    \node (0) at (0,0) {};
    \node (v) at (0,1) {};
    \node[leaf] (1) at (-1,2) {$1$};
    \node[leaf, anchor = south] (2) at (0, 2) {$2$};
    \node[leaf] (3) at (1, 2) {$3$};

    \draw[line] (0.mid) to node[weights] {$1$} (v.mid) to node[weights, swap] {$1$} (1);
    \draw[line] (v.mid) to node[weights] {$3$} (2);
    \draw[line] (v.mid) to node[weights] {$2$} (3);
    \node[orange] at (v.east) {$0$};
  \end{scope}

  \draw[arrowline] (0.5,2) to node[above] {\normalsize $\hat{\rho}^* \left( \hat{\tau}(\hat{\tree}, i, \tildei), \tildei \right)$} (2.5,2);
  \draw[arrowline] (-2,0.5) to node[auto] {\normalsize $\hat{\pi}_{[3]}^{\bullet 8}$} (-2, -1.5);
  \draw[arrowline] (4.75,0.5) to node[auto] {\normalsize $\hat{\pi}_{[3]}^{\bullet 7}$} (4.75, -1.5);
\end{tikzpicture}
  \caption{An illustration of the down-step in the semi-planar $(\alpha, \gamma)$-chain on $\Thatspace_{[8]}$.
    Starting from $\that \in \Thatspace_{[8]}$ (top left), leaf $i = 3$ is selected for deletion, whereby $a = 1$ and $b = 6$ so that $\tildei = \max \{i, a, b \} = 6$.
Hence the leaves labelled by $3$ and $6$, respectively, are swapped, and subsequently the leaf labelled by $6$ is deleted, and $(7, 8)$ is relabelled by $(6, 7)$, yielding the top right semi-planar tree.
In the bottom row the projection of the two trees onto $\T_{[3]}^{\bullet 8}$ and $\T_{[3]}^{\bullet 7}$, respectively, are depicted with the decorated masses marked in orange.
}
  \label{fig:decoratedtrees_downstep}
\end{figure}
From the above description it is clear, that the most complicated scenario of the down-step is when $\tildei \leq k$, whereas the other scenarios are easier to deal with.
For an example of one of down-steps, see Figure~\ref{fig:decoratedtrees_downstep}.

In the following definition, we will for fixed $\nin$ and $\tree \in \T_{[n]}$ with $\pi_{[k]}^n ( \tree ) = (\tree[s], \mathbf{y}) \in \T_{[k]}^{\bullet n}$ let $w_x$ denote the insertion weight associated with $x \in \insertablef[s]$ from Definition~\ref{def:growthprocess_alphagamma}.
To ease notation, we will for $(\tree[s], \mathbf{y}) \in \T_{[k]}^{\bullet n}$ set
\begin{align*}
  \tilde{y}_x =
  \begin{cases}
    y_x - 1 & \text{if $x \in \edge(\tree[s])$ is external,} \\
    y_x & \text{otherwise,}
  \end{cases}
\end{align*}
for all $x \in \insertablef[s]$.
\begin{defi}[Resampling a label]\label{def:resample}
Fix $\tree^\bullet = (\tree[s], \mathbf{y}) \in \T_{[k-1]}^{\bullet (n-1)}$.
  By \textit{resampling leaf $k$ in} $\tree^\bullet$ we mean obtaining an element $\tree_\text{re}^\bullet \in \T_{[k]}^{\bullet (n-1)}$ from $\tree^\bullet$ in the following way:
  \begin{enumerate}[label=\arabic*.]
    \item Select $x \in \insertable \left( \tree[s] \right)$ with probability $\frac{\tilde{y}_x}{n - k}$, and define $\tree[s]^\text{re} = \varphi(\tree[s], x, k)$.
    \item Define $\tree_\text{re}^\bullet = (\tree[s]^\text{re}, \mathbf{y}^\text{re})$, where $\mathbf{y}^\text{re} = {\left( y_{x^\prime}^\text{re} \right)}_{x^\prime \in \insertable ( {\tree[s]}^\text{re})}$ is constructed by setting $y_{x^\prime}^\text{re} = y_{x^\prime}$ for all $x^\prime \in \insertablef[s] \setminus \{x\}$ and splitting up $y_x$ in the following way:
      \begin{itemize}
        \item If $x \in \edge(\tree[s])$, let $v$ denote the branch point created in the previous step.
        Let $\left(Y_{k}, Y_v, Y_{v\downarrow}, Y_{v\uparrow} \right) \sim \dirmult^{\tilde{y}_x - 1}\left(1-\alpha, \alpha-\gamma, \gamma, w_x \right)$, set $y_{v\downarrow}^\text{re} = Y_{v\downarrow}, y_v^\text{re} = Y_v$, $y_k^\text{re} = Y_k + 1$, and implicitly define $y_{v\uparrow}$ by $\tilde{y}_{v\uparrow} = Y_{v\uparrow}$.    
      \item If $x \in \branchpoints(\tree[s])$, let $(Y_{k}, Y_x) \sim \dirmult^{y_x - 1} \left(1-\alpha, w_x \right)$, set $y_x^{\text{re}} = y_x - Y_x - 1$ and $y_k^{\text{re}} = Y_k + 1$.
      \end{itemize}
  \end{enumerate}
\end{defi}
We will use the above resampling mechanism to give an autonomous description of the transition kernel of the decorated $(\alpha, \gamma)$-chain from Definition~\ref{def:decorated_alphagamma_chain}.
The arguments used in the proof of the following proposition are similar to the ones used to argue the spinal decomposition for the $(\alpha, \gamma)$-growth model in~\cite{RefWorks:doc:5b4cbb5fe4b02dc0c79270af}.
In the following, recall that $q_{\alpha, \theta}$ denotes the decrement matrix of the $(\alpha, \theta)$-$\ocrp$ defined by \eqref{eq:decrementmatrix}.
\begin{prop}\label{prop:decorated_transition_kernel}
  Fix $1 \leq k < n \in \N$.
  The transition kernel of the decorated $(\alpha, \gamma)$-chain on $\T_{[k]}^{\bullet n}$, ${\left( T_k^{\bullet n}(m) \right)}_{m \in \N_0}$, can be characterized as follows:
  \begin{enumerate}[label=2.\alph*]
    \item[1.] Conditional on $T_{[k]}^{\bullet n}(m) = (\tree[s], \mathbf{y}) \in \T_{[k]}^{\bullet n}$, select $x \in \insertablef[s]$ with probability $\frac{y_x}{n}$.
    \item Unless $x \in \edge ( \tree[s] )$ is external with $y_x = 1$, reduce $y_x$ by $1$.
    \item If $x = i \in \edge (\tree[s])$ is external with $y_x = 1$, let $v \in \branchpoints(\tree[s])$ denote the parent of $x$ in $\tree[s]$, and let $c_v$ denote the number of children of $v$ in $\tree[s]$.
      \begin{itemize}
        \item If $c_v > 2$, let $N \sim \betabin^{y_v} (\alpha, (c_v-2)\alpha - \gamma)$. 
          \begin{itemize}
              \item Given $N = n_g > 0$, sample $Y_{\text{new}}$ from $q_{\alpha, \alpha}( n_g, \cdot )$, and replace $y_x$ with $Y_{\text{new}}$ and $y_v$ with $y_v - Y_{\text{new}}$.
              \item If $N = 0$, perform the down-step of the $(\alpha, \gamma)$-chain on $\T_{[k]}$ from $\tree[s]$ (step (ii) of Definition~\ref{def:alphagamma_chain}) to form  $\tree_\downarrow^\bullet \in \T_{[k-1]}^{\bullet (n-1)}$, and resample leaf $k$ in $\tree_\downarrow^\bullet$.
          \end{itemize}
        \item If $c_v=2$ do the following:
          \begin{itemize}
            \item if $y_v > 0$, let $Y_{\text{new}} \sim q_{\alpha, \alpha - \gamma} (y_v, \cdot)$, and replace $y_x$ with $Y_{\text{new}}$ and $y_v$ with $y_v - Y_{\text{new}}$.
            \item if $y_v = 0$ and $y_{e_{v\downarrow}} > 0$, define ${\tree[s]' = \varphi \left( \rho(\tree[s], i), e_{v\downarrow}, i \right)}$, and let {$N_b \sim q_{\gamma, \gamma} (y_{e_{v\downarrow}}, \cdot)$}.
              Conditional on the event $\{N_b = n_b\}$, sample {$N_t - 1 \sim \betabin^{n_b - 1}(1-\alpha, \alpha - \gamma)$}, replace $y_i$ with $N_t$, $y_v$ with $N_b - N_t$, and $y_{e_{v\downarrow}}$ with $y_{e_{v\downarrow}} - N_b$.
            \item if $y_v = y_{e_{v\downarrow}} = 0$, perform the down-step of the $(\alpha, \gamma)$-chain on $\T_{[k]}$ from $\tree[s]$ (step (ii) of Definition~\ref{def:alphagamma_chain}) to form  $\tree_\downarrow^\bullet \in \T_{[k-1]}^{\bullet (n-1)}$, and resample leaf $k$ in $\tree_\downarrow^\bullet$.
            \end{itemize}
        \end{itemize}
      \item[3.] Perform an up-step using the Markov kernel of the decorated $(\alpha, \gamma)$-growth process from Definition~\ref{def:decorated_alphagamma_growth}.
    \end{enumerate}
\end{prop}
\begin{proof}
  Before we consider the down-step itself, note that it does not make any difference if we perform the up-step for the semi-planar $(\alpha, \gamma)$-chain and then project to the decorated trees, or if we project to the decorated trees immediately after the down-step, and then perform the up-step by using the decorated $(\alpha, \gamma)$-growth process, see the brief discussion following Definition~\ref{def:decorated_alphagamma_growth}.
  This ensures, that we in the following only have to argue that the down-step outlined in Step 1 and 2 is the correct characterization. 

  Consider the decorated $(\alpha, \gamma)$-chain started at $\tree^\bullet = (\tree[s], \mathbf{y}) \in \T_{[k]}^{\bullet n}$.
  Now
  \begin{enumerate}
    \item sample $\That$ from $\P \left( \That_n \in \cdot \ \middle \vert \ \hat{\pi}_{[k]}^{\bullet n} \left( \That_n \right) = \tree^\bullet \right)$,
    \item in the down-step of the semi-planar $(\alpha, \gamma)$-chain from $\That$, condition on $I = i$ and $\tilde{I} = \tildei$, delete $\tildei$, and relabel $\tildei + 1, \ldots, n$ by $\tildei, \ldots, n-1$ using the increasing bijection, obtaining $\That^\downarrow$.
  \end{enumerate}
  Furthermore let $T_{[k]}^{\bullet n}{(m)}^\downarrow$ denote $T_{[k]}^{\bullet n} (m)$ after the down-step.
  Our aim is to prove that
  \begin{align}
    \P \left( T_{[k]}^{\bullet n}{(m)}^\downarrow \in \cdot \ \middle \vert \ T_{[k]}^{\bullet n}{(m)} = \tree^\bullet \right)
    = \P \left( \hat{\pi}_{[k]}^{\bullet (n-1)} \left( \That^\downarrow \right) \in \cdot \right) 
    \label{eq:decorated_downstep_distribution}
  \end{align}
  Fundamentally, we now split up into cases based on $\tree^\bullet$ and $i$ where we recall the notation $B_x = g^{-1}(\{x\})$ for $x \in \insertablef[s]$ where $g$ is the decoration function from $\hat{T}$ to $\tree[s]$.
  Since $\hat{T}$ is random, the sets ${\left( B_x \right)}_{x \in \insertablef[s]}$ form a random partition of $[n]$.
  However, we note that in all cases Step 1 is consistent with how we select leaf $I$ in the semi-planar $(\alpha, \gamma)$-chain, since $I \sim \Unif ([n])$ immediately implies
  \begin{align*}
    \P \left( I \in B_x = g^{-1}(\{x\}) \right) = \frac{y_x}{n}
  \end{align*}
  for all $x \in \insertablef[s]$.
  
  In the following we will let $a$, $b$ and $\tildei = \max \{i, a, b\}$ be defined as in Definition~\ref{def:planaralphagammachain}, and will let $(S^\downarrow, Y^\downarrow) := \hat{\pi}_{[k]}^{\bullet (n-1)} \left( \That_n^\downarrow \right)$ be the decorated $[k]$-tree of mass $n-1$ obtained by projecting the semi-planar $(\alpha, \gamma)$-chain after the down-step.
  
  \textbf{Case A1: $i \in B_x$ for external $x \in \edge (\tree[s])$, $y_x > 1$.}
  Note how in this case $i \in B_x$ implies that $\tildei = \max \{ i, a, b \} \in B_x$ as well, implying that $\tildei > k$.
  Hence the tree shape does not change during the down-step, and $(S^\downarrow, Y^\downarrow)$ will have the property that $S^\downarrow = \tree[s]$, $Y_x^\downarrow = y_x - 1$, and $Y_{x'}^\downarrow = y_{x'}$ for all $\insertablef[s] \ni x' \neq x$.
  This corresponds to Step 2.a where $x$ is an external edge.

  \textbf{Case A2: $i \in B_x$ for internal $x \in \edge (\tree[s])$.}
  Noting that $i > k$, the definition of $a$ and $b$ implies that $\tildei \in B_x$ as well.
  Hence this case is similar to Case A1, and corresponds to Step 2.a where $x$ is an internal edge.

  \textbf{Case A3: $i \in B_x$ for $x \in \branchpoints (\tree[s])$.}
  This is completely analogous to Case A2, and corresponds to Step 2.a where $x$ is a branch point. \\

  In all of the following cases, $v \in \branchpoints (\tree[s])$ refers to the parent branch point of $i$ in $\tree[s]$, $c_v$ will refer to the number of children of $v$ in $\tree[s]$, and we construct $\that[s] = \hat{\rho} \left( \That, [n] \setminus [k] \right)$ which implies that $\hat{\pi} (\that[s]) = \tree[s]$.

  \textbf{Case B1: $i \in B_x$ for external} $x \in \edge (\tree[s]), y_x = 1, y_v > 0$, $\tildei > k$.
  Firstly, we note that the restrictions on $i$ and $y_x$ implies that $i \leq k$.
  Now consider the semi-planar structure in the branch point $v$, and note how $y_x = 1$ and $y_v > 0$ implies that $\max \{i, a, b\} = \max \{i, b\}$.
Hence this case covers the situations where (1) $c_v > 2$ and a leaf label larger than $k$ is found to the left (or right, if $i$ is placed leftmost in $v$) of $i$ in $\That$, and (2) $c_v = 2$, in which case we are guaranteed to find a label larger than $k$ to the right of $i$ in $\That$.
In order to specify the decorated tree after the down-step, we only need to characterize the number of leaves found in the subtree of $v$ containing $\tildei$ in $\That$.

From the decoration, it is clear that $y_v$ leaves have been inserted to form subtrees of $v$, different from the ones already in $\that[s]$.
For $c_v > 2$, with $i$ not located leftmost in $v$, these subtrees will have grown in the following way.
Initially there will be weight $\alpha$ to the immediate left of $i$, and an accumulated weight of $(c_v-2)\alpha - \gamma$ in the rest of the branch point.
Thus, the numbers of insertions happening to the immediate left of $i$ versus everywhere else in $v$, respectively, will follow a P\'{o}lya urn scheme with these initial weights.
Now, say that $n_g > 0$ leaves have been inserted in the “gap” between $i$ and the left neighbour of $i$ in $\that[s]$.
These $n_g$ leaves will form subtrees of $v$, and the location and sizes of these subtrees will be governed by an $(\alpha, \alpha)$-$\ocrp$.
Specifically, $b$ will be located in the first subtree to the left of $i$ in $\That$, so denoting the size of this subtree by $Y_\text{new}$, we observe that $Y_\text{new}$ is distributed as the size of the rightmost table in an $(\alpha, \alpha)$-$\ocrp$ with $n_g$ customers.
Hence swapping $i$ and $\tildei$, deleting $\tildei$, relabelling $\{\tildei + 1, \ldots, n\}$ by $\{\tildei, \ldots, n-1\}$ using the increasing bijection, and projecting to decorated $[k]$-trees of mass $n-1$, implies that we obtain $S^\downarrow = \tree[s]$, $Y_i^\downarrow = Y_{\text{new}}$, $Y_v^\downarrow = y_v - Y_{\text{new}}$ and $Y_x^\downarrow = y_x$ for all $x \in \insertable \left({\tree[s]}_\downarrow \right) \setminus \{v\}$.

The argument above is exactly the same if $i$ is placed leftmost in $v$ in $\That$, as we then search for $b$ to the right of leaf $i$ instead of to the left, but still encounter an $(\alpha, \alpha)$-$\ocrp$. 

If $c = 2$, the argument is the same, except that we are guaranteed to find something to the right of $i$ in $v$ in $\That$.
Contrary to the above, the location and sizes of the subtrees  of $v$ in $\That$ with smallest leaf label larger than $k$ is governed by an $(\alpha,\alpha-\gamma)$-$\ocrp$, but the remainder of the argument carries over.

\textbf{Case B2: $i \in B_x$ for external} $x \in \edge (\tree[s]), y_x = 1, y_v > 0, \tildei \leq k$.
Following the same reasoning as above, this case only contains the event where $n_g = 0$ leaves with labels larger than $k$ have been inserted in the gap between $i$ and the left (or right if $i$ is placed leftmost in $v$) neighbour of $i$ in $v$, implying that we have $\tildei = \max \{i, b\} \leq k$.
  By swapping $i$ and $\tildei$ and subsequently deleting $\tildei$, we end up removing one of the leaves that defines the tree shape $\tree[s]$.
  Consider how $\That$ was grown from any semi-planar representative $\that[s]$ of $\tree[s]$ using the semi-planar $(\alpha, \gamma)$-growth process.
The sets ${\left( B_x \right)}_{x \in \insertablef[s]}$ form a random partition of $[n]$ induced by a P\'{o}lya urn scheme, and it follows from Proposition~\ref{prop:polyaurn} that the probability of seeing $k+1 \in B_x$ is exactly $\frac{\tilde{y}_x}{n-k}$ for any $x \in \insertablef[s]$.
In $\That$, swap $i$ and $\tildei$, delete $\tildei$ and relabel $\tildei + 1, \ldots, n$ by $\tildei, \ldots, n-1$ using the increasing bijection, and consider the tree shape and decoration masses of the projected tree (these operations below are well-defined as no edges or vertices of $\tree[s]$ apart from the leaf edge $\tildei$ are deleted).
Say that $k+1 \in B_{X_0}$ for a random $X_0 \in \insertablef[s]$, the tree shape after projecting to decorated trees will thus be $\varphi \left( \rho ( \tau ( \tree[s], i, \tildei ), \tildei), X_0, k \right)$, i.e.
  \begin{enumerate}
    \item if $X_0 = v \in \branchpoints (\tree[s])  $ the new tree shape will have $k$ as a child of $v$, and
    \item if $X_0 = e \in \edge (\tree[s])$, $k$ will split the edge $e$ up into a new branch point and three edges.
  \end{enumerate}
Further considerations of the correspondence between urn schemes and inserting the leaves $B_{X_0} \setminus \{k+1\}$ into $\varphi(\tree[s], X_0, k+1) = \hat{\pi} \left( \hat{\rho} (\That, [n] \setminus [k+1]) \right)$, or more precisely into the insertable parts of the projected tree specified above, shows that resampling $\tildei$ exactly corresponds to deleting $\tildei$, relabelling using the increasing bijection, and subsequently projecting to decorated $[k]$-trees of mass $n-1$.

  \textbf{Case B3: $i \in B_x$ for external} $x \in \edge (\tree[s])$, $y_x = 1$, $y_v = 0$, $\tildei > k$.
  This case is where we have a binary branch point, and there is mass on the edge $e_{v\downarrow}$ in $\tree[s]$, but is otherwise similar to Case B1, with the notable difference that the leaf insertions into $e_{v\downarrow} \in \edge (\tree[s])$ are governed by a $(\gamma, \gamma)$-$\ocrp$.
  In this case, $\tildei = b$ and $b$ will be the smallest leaf in the second spinal bush on the ancestral line from leaf $i$ in $\That$.
  Letting $N_b$ denote the random number of leaves in this spinal bush, by Proposition~\ref{prop:decrementmatrix} we have that $N_b \sim q_{\gamma,\gamma}(e_{v\downarrow}, \cdot)$.
  Further considerations of the insertions, as in Case B2, show that conditional on $N_b = n_b$, the number of leaves in each of the subtrees of of the spinal bush will be governed by a Chinese Restaurant process with parameters $(\alpha, \alpha-\gamma)$ with $n_b - 1$ customers.
  This is also an immediate consequence of the spinal decomposition of an $(\alpha, \gamma)$-tree (see Lemma 11 of~\cite{RefWorks:doc:5b4cbb5fe4b02dc0c79270af}). 
  
\textbf{Case B4: $i \in B_x$ for external} $x \in \edge (\tree[s])$, $y_x = 1$, $y_v = 0$, $\tildei \leq k$.
  This case is where we have a binary branch point, and there is no mass on $e_{v\downarrow}$ in $\tree[s]$, and so $\tildei < k$ and the argument presented in Case B2 can be used verbatim.
\end{proof}
With an autonomous description of the transition kernel for the decorated $(\alpha, \gamma)$-chain on $\T_{[k]}^{\bullet n}$ we turn to the more complicated task of proving that the projection of the semi-planar $(\alpha, \gamma)$-chain on $\Thatspace_{[n]}$ onto $\T_{[k]}^{\bullet n}$, ${\left( \hat{\pi}_{[k]}^{\bullet n} \left( \That_n(m) \right) \right)}_{m \in \N_0}$, is a Markov chain on $\T_{[k]}^{\bullet n}$ with the transition kernel characterized in Proposition~\ref{prop:decorated_transition_kernel}.

\section{Sampling semi-planar $(\alpha, \gamma)$-trees}
Similar to how we characterized $\hat{\Pi}$ in~\eqref{eq:MarkovKernel}, we now wish to characterize $\hat{\Pi}_k^{\bullet n}$.
In order to do this, let us define the concept of internal structures which will allow us to efficiently sample non-planar or semi-planar $[n]$-trees from a collapsed or decorated tree of mass $n$.

%  Specifically, for each $x \in \branchpoints(\tree[s]) \cup \edge (\tree[s])$ define
%  %
%  \begin{align}
%    {\tree}^{\tree[s]}_x = \left\{ x' \in \vertices(\tree) \cup \edge(\tree)
%      \ \vert \ 
%      \exists x_0 \in f^{-1}(\{x\}) \colon x \ll_{\tree} \ x' \underline{\ll}_{\tree} \ x_0 \right\},
%    \label{eq:internalstructure_proj_int}
%  \end{align}
%  %
%  and for each $x \in \edge(\tree[s])$ where the parent and child vertex of $x$ in $\tree[s]$ are $v$ and $v'$, respectively, define
%  %
%  \begin{align}
%    {\tree}^{\tree[s]}_x = \left\{ x' \in \vertices(\tree) \cup \edge(\tree) \ \bigg\vert \
%      \begin{matrix}
%      \exists v_0 \in \vertices (\tree) \colon \text{$v$ is the parent of $v_0$, no} \\
%      \text{ancestor of $x'$ is a descendant of $v'$, and} \\
%    \text{$v_0$ is on the ancestral line from $x'$ in $\tree$} 
%  \end{matrix}
%\right\},
%    \label{eq:internalstructure_proj_edge}
%  \end{align}
%  %
  So fix $\tree \in \T_{[n]}$ and let $\pi_{[k]}^{\bullet n} ( \tree ) = \tree^\bullet = \left( \tree[s], \mathbf{y} \right)$ for some $\tree[s] \in \T_{[k]}$, and let $g$ be the decorating function.
There is a partition of all edges and vertices of $\tree$ indexed by the elements of $\insertablef[s]$.
Define for each $x \in \insertablef[s]$ the set of all vertices of $\tree$ that are not vertices of $\tree[s]$, and from which $x$ is the “first” edge or vertex of $\tree[s]$, respectively, that is reached on the ancestral line.
More precisely, for each $v \in \branchpoints (\tree[s])$ define
  \begin{align*}
    \mathring{V}_{\tree[s], \tree}^v
    = \left\{ v' \in \vertices (\tree) \ \middle \vert \  
      \begin{matrix}
        v' = v \ \text{or}\ \exists j \in g^{-1}(\{v\})\ \exists v_0, \ldots, v_m \in \vertices (\tree) \colon \\
        v_0 = v, v_m = j, v' = v_i \ \text{for some}\ i \in [m], \\
        \text{and}\ v_{i-1} \prec_{\tree} v_i \ \text{for all}\ i \in [m]
      \end{matrix}
  \right\},
\end{align*}
  and
  \begin{align*}
    \mathring{E}_{\tree[s],\tree}^v
    = \left\{ e' \in \edge(\tree) \ \middle \vert \ \exists u, u' \in \mathring{V}_{\tree[s], \tree}^v \colon e' = e_{uu'} \right\},
  \end{align*}
  whilst for each $e \in \edge (\tree[s])$ with, say, $e = e_{uv}$, define
\fxnote{Is Figure 5 enough to illustrate these sets?}
  \begin{align*}
    \mathring{V}_{\tree[s], \tree}^e
    = \left\{ v' \in \vertices(\tree) \ \middle \vert \
      \begin{matrix}
        \exists j \in g^{-1}(\{e\})\ \exists v_0, \ldots, v_m \in \vertices (\tree) \colon \\
        v_0 = u, v_m = j, v' = v_i \ \text{for some}\ i \in [m], \\
        \text{and}\ v_{i-1} \prec_{\tree} v_i \ \text{for all}\ i \in [m]
      \end{matrix}
    \right\}.
  \end{align*}
  and
  \begin{align*}
    \mathring{E}_{\tree[s],\tree}^e
    = \left\{ e' \in \edge(\tree) \ \middle \vert \
      \begin{matrix}
        e' = e_{uv} \in \edge (\tree) \ \text{if}\ \mathring{V}_{\tree[s], \tree}^e = \emptyset, \text{or} \\
      \exists u', v' \in \mathring{V}_{\tree[s], \tree}^e \colon e' = e_{u'v'} \ \text{or}\ e' = e_{uu'} \ \text{or}\ e' = e_{v'v}
    \end{matrix}
  \right\}.
  \end{align*}
  We note that $\mathring{V}_{\tree[s], \tree}^x$ either is empty, contains only $\{x\}$ if $x \in \branchpoints (\tree[s])$,  or will contain at least one leaf for each $x \in \insertablef[s]$, and in the latter case we continue to label these leaves by their corresponding label in $\tree$.
  Additionally, the only situation where $\mathring{V}_{\tree[s], \tree}^x$ is empty, but $\mathring{E}_{\tree[s], \tree}^x$ is not, is when $x \in \edge(\tree[s])$ is internal and $g^{-1}(\{x\}) = \emptyset$, in which case $\mathring{E}_{\tree[s], \tree}^x = \{x\}$.
  Likewise, whenever $x \in \branchpoints (\tree)$ and $g^{-1}(\{ x \}) = \emptyset$, $\mathring{V}_{\tree[s], \tree}^x = \{x\}$ and $\mathring{E}_{\tree[s], \tree}^x = \emptyset$.

The above sets constitute a partition of $\tree$, i.e.\ a partition of both the vertex and edge sets of $\tree$, labelled by the elements of $\tree[s]$, but it is only almost true that $(\mathring{V}_{\tree[s], \tree}^x, \mathring{E}_{\tree[s], \tree}^x)$ constitutes a tree for each $x \in \insertablef[s]$.
  To ease stating later results, we amend this by adding some missing parts:
  \begin{defi}[Internal Structure]
    Fix $\tree \in \T_{[n]}$ and {$\pi_{[k]}^{\bullet n} (\tree) = (\tree[s], \mathbf{y})$}.
    For each $x \in \insertablef[s]$, make the following alterations to $\mathring{V}_{\tree[s], \tree}^x$ and $\mathring{E}_{\tree[s], \tree}^x$ defined above:
      \begin{itemize}
        \item If $x = e_{uv} \in \edge(\tree[s])$ is external such that $u \prec_{\tree[s]} v$ then 
          \begin{enumerate}
            \item add $\emptyset$ to $\mathring{V}_{\tree[s], \tree}^x$, and
            \item replace $e_{uu'}$ with $e_{\emptyset u'}$ in $\mathring{E}_{\tree[s], \tree}^x$, with $u' \in \mathring{V}_{\tree[s], \tree}^x$ uniquely characterised by $u \prec_{\tree} u'$.
          \end{enumerate}
        \item If $x = e_{uv} \in \edge(\tree[s])$ is internal such that $u \prec_{\tree[s]} v$ then
          \begin{enumerate}
            \item add $\emptyset$ and $1$ to $\mathring{V}_{\tree[s], \tree}^x$, and
            \item if $\mathring{V}_{\tree[s], \tree}^x = \emptyset$ replace the edge $e_{uv}$ with $e_{\emptyset 1}$, and otherwise replace the two edges $e_{uu'}$ and $e_{v'v}$ with $e_{\emptyset u'}$ and $e_{v' 1}$ in $\mathring{E}_{\tree[s], \tree}^x$, where $u', v' \in \mathring{V}_{\tree[s], \tree}^x$ uniquely characterized by $u \prec_{\tree} u'$ and $v' \prec_{\tree} v$.
          \end{enumerate}
        \item If $x = v \in \branchpoints (\tree[s])$ with $c$ children in $\tree[s]$ then
          \begin{enumerate}
            \item add $\emptyset$ and $1, \ldots, c$ to $\mathring{V}_{\tree[s], \tree}^x$, and
            \item add the edges $e_{\emptyset v}, e_{v1}, \ldots, e_{vc}$ to $\mathring{E}_{\tree[s], \tree}^x$.
          \end{enumerate}
      \end{itemize}
    In all three cases, replace each of the leaf labels in the resulting tree with their ranks to obtain $V_{\tree[s], \tree}^x$ and $E_{\tree[s], \tree}^x$, and define the \textit{internal structure of $x$ in} $\tree$ to be the tree $\intstruct(x) := \left( V_{\tree[s], \tree}^x, E_{\tree[s], \tree}^x \right)$.
      \label{def:internalstructure}  
\end{defi}
We note that if $x \in \edge (\tree[s])$ is external then $\intstruct (x) \in \T_{[y_x]}$, if $x \in \edge (\tree[s])$ is internal then $\intstruct (x) \in \T_{[y_x + 1]}$, whilst if $x \in \branchpoints(\tree[s])$ with $c$ children then $\intstruct (x) \in \T_{[c + {y_x}]}$.
Additionally, note how all internal structures above could have been defined equivalently from any semi-planar tree, $\that$, such that $\hat{\pi} (\that) = \tree$, and even turned into semi-planar trees, simply by letting each vertex in the internal structure inherit the permutation from $\that$ as all ranks of subtrees are preserved (note how this is well-defined by the addition of the leaves $\{1, \ldots, c\}$ to $V_{\tree[s], \tree}^v$ before relabelling).
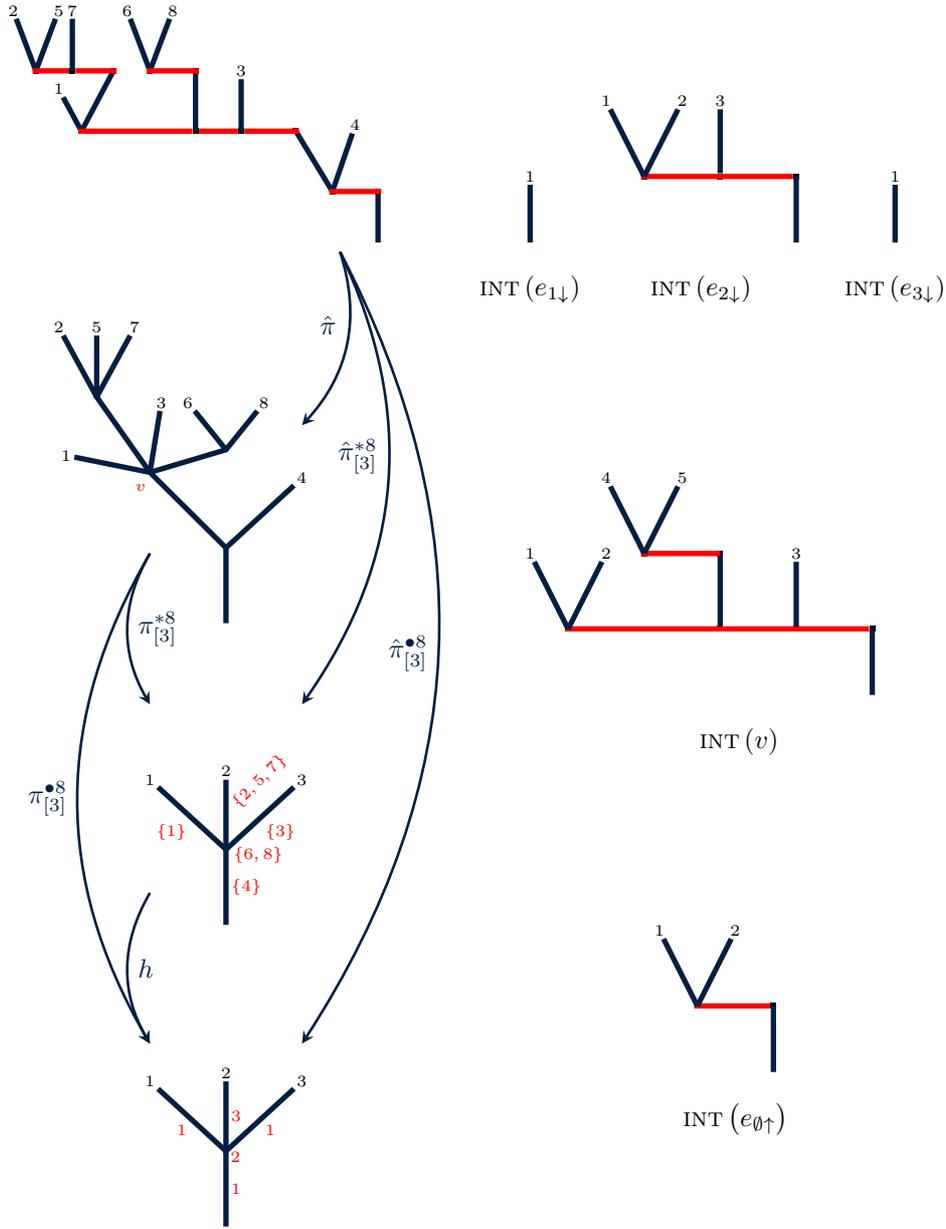
\begin{figure}[t!]
  \centering
  \begin{tikzpicture}[
    bp/.style={rectangle, fill, inner sep=0pt, minimum width=2pt, minimum height=2pt},
    comment/.style={circle, draw, color=red},
    leaf/.style={inner sep=1pt},
    weights/.style={auto, swap, red, inner sep = 1pt},
    label/.style={text=oxfordblue},
    line/.style={-, line width = 2pt, color = oxfordblue},
    arrowline/.style={->, line width = 1pt, color = oxfordblue, >=stealth},
    scale=1
  ]
  \tikzstyle{every node}=[font=\tiny]

  %
  % Semi-planar version
  %

  \begin{scope}[xscale = -0.6, yscale = 0.8]

    \node (0) at (0,0) {};

    \node[bp] (bp_1) at (0,1) {};
    \node[bp] (bp_2) at (1,1) {};
    \node[leaf] (4) at (0.5,2.1) {$4$};

    \node[bp] (v_1) at (1.8,2) {};
    \node[bp] (v_2) at (3,2) {};
    \node[bp] (v_3) at (4,2) {};
    \node[bp] (v_4) at (6.5,2) {};

  % 1 subtree (from left) (v_1)
    \node[bp] (bp2_1) at (4,3) {};
    \node[bp] (bp2_2) at (5,3) {};
    \node[leaf] (8) at (4.5, 4) {$8$};
    \node[leaf] (6) at (5.5, 4) {$6$};

  % 2 subtree (from left) (v_2)
    \node[leaf] (3) at (3,3) {$3$};

  % 3 subtree (from left) (v_3)
    \node[bp] (bp3_1) at (5.8,3) {};
    \node[bp] (bp3_2) at (6.7,3) {};
    \node[bp] (bp3_3) at (7.5,3) {};
    \node[leaf] (7) at (6.7, 4) {$7$};
    \node[leaf] (5) at (7, 4) {$5$};
    \node[leaf] (2) at (8, 4) {$2$};

  % 4 subtree (from left)
    \node[leaf] (1) at (7,2.7) {$1$};

    \draw[line] (0) to (bp_1);
    \draw[line] (bp_2.center) to (4);

    \draw[line] (bp_2.center) to (v_1.center);
    \draw[line] (v_2.center) to (3);
    \draw[line] (bp2_2.center) to (8);
    \draw[line] (bp2_2.center) to (6);

    \draw[line] (v_3.center) to (bp2_1.center);

    \draw[line] (v_4.center) to (bp3_1.center);
    \draw[line] (bp3_2.center) to (7);
    \draw[line] (bp3_3.center) to (5);
    \draw[line] (bp3_3.center) to (2);

    \draw[line] (v_4.center) to (1);

    \draw[line, red] (bp_1.east) to (bp_2.west);
    \draw[line, red] (v_1.east) to (v_2) to (v_3) to (v_4.west);
    \draw[line, red] (bp2_1.east) to (bp2_2.west);
    \draw[line, red] (bp3_1.east) to (bp3_2) to (bp3_3.west);
  \end{scope}
  
  %
  % Semi-planar version after down-step
  %

  \begin{scope}[shift = {(-2,-5)}, xscale = -1]
    \node (0) at (0,0) {};
    \node (bp) at (0,1) {};
    \node[leaf] (4) at (-1,2) {$4$};
    \node (v) at (1,2) {};

  % 1 subtree (from left)
    \node (bp2) at (0,2.3) {};
    \node[leaf] (8) at (-0.5, 3) {$8$};
    \node[leaf] (6) at (0.5, 3) {$6$};

  % 2 subtree (from left)
    \node[leaf] (3) at (0.85,3) {$3$};

  % 3 subtree (from left)
    \node (bp3) at (1.7,3) {};
    \node[leaf] (7) at (1.2, 4) {$7$};
    \node[leaf] (5) at (1.7, 4) {$5$};
    \node[leaf] (2) at (2.2, 4) {$2$};

  % 4 subtree (from left)
    \node[leaf] (1) at (2.1,2.3) {$1$};

    \draw[line] (0.mid) to (bp.mid) to (4);
    \draw[line] (bp.mid) to (v.mid) to (bp2.mid) to (8);
    \draw[line] (bp2.mid) to (6);
    \draw[line] (v.mid) to (3);
    \draw[line] (v.mid) to (bp3.mid) to (7);
    \draw[line] (bp3.mid) to (5);
    \draw[line] (bp3.mid) to (2);
    \draw[line] (v.mid) to (1);

    \node[red] at (v.south west) {$v$};
  \end{scope}
  %
  % First decorated tree
  %
  \begin{scope}[shift = {(-2,-9)}]
    \node (0) at (0,0) {};
    \node (v) at (0,1) {};
    \node[leaf] (1) at (-1,2) {$1$};
    \node[leaf, anchor = south] (2) at (0, 2) {$2$};
    \node[leaf] (3) at (1, 2) {$3$};

    \draw[line] (0.mid) to node[weights] {$\{4\}$} (v.mid) to node[weights, swap] {$\{1\}$} (1);
    \draw[line] (v.mid) to node[weights, rotate=45, shift={(0.1,0)}] {$\{2,5,7\}$} (2);
    \draw[line] (v.mid) to node[weights] {$\{3\}$} (3);
    \node[red, shift={(0.3,0)}] at (v.east) {$\{6,8\}$};
  \end{scope}
  \begin{scope}[shift = {(-2,-13)}]
    \node (0) at (0,0) {};
    \node (v) at (0,1) {};
    \node[leaf] (1) at (-1,2) {$1$};
    \node[leaf, anchor = south] (2) at (0, 2) {$2$};
    \node[leaf] (3) at (1, 2) {$3$};

    \draw[line] (0.mid) to node[weights] {$1$} (v.mid) to node[weights, swap] {$1$} (1);
    \draw[line] (v.mid) to node[weights] {$3$} (2);
    \draw[line] (v.mid) to node[weights] {$1$} (3);
    \node[red] at (v.east) {$2$};
  \end{scope}
  %
  %%% 2nd column, 1st row
  %
  \begin{scope}[shift ={(2,0)}]

    \node (0) at (0,0) {};

    \node[leaf] (1) at (0,1) {$1$};

    \draw[line] (0) to (1.south);
    
    \node at (0,-0.5) {\normalsize $\intstruct \left( e_{1\downarrow} \right)$};
  \end{scope}
  \begin{scope}[shift={(5.5,0)}]
    \node (0) at (0,0) {};

    \node[bp] (bp_l3) at (0,1) {};
    \node[bp] (bp_l2) at (-1,1) {};
    \node[bp] (bp_l1) at (-2,1) {};
    
    \node[leaf] (1) at (-2.5,2) {$1$};
    \node[leaf] (2) at (-1.5,2) {$2$};
    \node[leaf] (3) at (-1,2) {$3$};
    
    \draw[line] (0) to (bp_l3.center);
    \draw[line, red] (bp_l3.west) to (bp_l2.center) to (bp_l1.west);
    \draw[line] (bp_l1.center) to (1);
    \draw[line] (bp_l1.center) to (2);
    \draw[line] (bp_l2.north) to (3);

    \node at (-1.25,-0.5) {\normalsize $\intstruct \left( e_{2\downarrow} \right)$};
  \end{scope}
  \begin{scope}[shift ={(6.8,0)}]

    \node (0) at (0,0) {};

    \node[leaf] (1) at (0,1) {$1$};

    \draw[line] (0) to (1);
    
    \node at (0,-0.5) {\normalsize $\intstruct \left( e_{3\downarrow} \right)$};
  \end{scope}
  \begin{scope}[shift={(6.5,-6)}]
    \node (0) at (0,0) {};

    \node[bp] (bp_l4) at (0,1) {};
    \node[bp] (bp_l3) at (-1,1) {};
    \node[bp] (bp_l2) at (-2,1) {};
    \node[bp] (bp_l1) at (-4,1) {};
    \node[leaf] (1) at (-4.5,2) {$1$};
    \node[leaf] (2) at (-3.5,2) {$2$};
    \node[leaf] (3) at (-1,2) {$3$};

    \node[bp] (bp2_l1) at (-3,2) {};
    \node[bp] (bp2_l2) at (-2,2) {};
    \node[leaf] (4) at (-3.5,3) {$4$};
    \node[leaf] (5) at (-2.5,3) {$5$};
    
    \draw[line] (0) to (bp_l4.center);
    \draw[line, red] (bp_l4.west) to (bp_l3.center) to (bp_l2.center) to (bp_l1.west);
    \draw[line] (bp_l1.center) to (1);
    \draw[line] (bp_l1.center) to (2);
    \draw[line] (bp_l2.north) to (bp2_l2.center);
    \draw[line] (bp_l3.north) to (3);
    \draw[line, red] (bp2_l2.center) to (bp2_l1.west);
    \draw[line] (bp2_l1.center) to (4);
    \draw[line] (bp2_l1.center) to (5);

    \node at (-1.75,-0.5) {\normalsize $\intstruct \left( v \right)$};
  \end{scope}
  \begin{scope}[shift={(5.2,-11)}]
    \node (0) at (0,0) {};

    \node[bp] (bp_l2) at (0,1) {};
    \node[bp] (bp_l1) at (-1,1) {};
    
    \node[leaf] (1) at (-1.5,2) {$1$};
    \node[leaf] (2) at (-0.5,2) {$2$};
    
    \draw[line] (0) to (bp_l2.center);
    \draw[line, red] (bp_l2.west) to (bp_l1.west);
    \draw[line] (bp_l1.center) to (1);
    \draw[line] (bp_l1.center) to (2);

    \node at (-0.5,-0.5) {\normalsize $\intstruct \left( e_{\emptyset \uparrow} \right)$};
  \end{scope}
  %
  %
  % Semi-planar version after down-step
  %

%  \node[inner sep=0] (split) at (-0.2,-2) {};
%  \draw[arrowline] (-2,0.5) to node[auto, swap] {\normalsize $\pi$} (-2, -1.5);
  \draw[arrowline] (-3,-4) to [bend right] node[auto] {\normalsize $\pi_{[3]}^{* 8}$} (-3, -6);
  \draw[arrowline] (-3,-8.5) to [bend right] node[auto] {\normalsize $h$} (-3, -10.5);
  \draw[arrowline] (-3,-4) to [bend right] node[auto, swap] {\normalsize $\pi_{[3]}^{\bullet 8}$} (-3, -10.5);
%  \draw[arrowline] (-2,0.5) to (split.north west) to [bend left] node[auto, swap] {\normalsize $\hat{\pi}_{[3]}^{* 8}$} (-0.5, -6);
%  \draw[arrowline] (split.north west) to [bend left] node[auto, swap] {\normalsize $\hat{\pi}_{[3]}^{\bullet 8}$} (-0.5, -11);

  \draw[arrowline] (-0.5,0) to [bend left] node[auto, swap] {\normalsize $\hat{\pi}$} (-1, -2.3);
  \draw[arrowline] (-0.5,0) to [bend left] node[auto, swap] {\normalsize $\hat{\pi}_{[3]}^{* 8}$} (-1, -6);
  \draw[arrowline] (-0.5,0) to [bend left] node[auto, swap] {\normalsize $\hat{\pi}_{[3]}^{\bullet 8}$} (-1, -10.5);
\end{tikzpicture}
  \caption{An illustration of the internal structures of an element of $\Thatspace_{[8]}$.}
  \label{fig:decoratedtrees_internalstructure}
\end{figure}
In most cases, it will be completely clear what $\tree$ and $\tree[s]$ are in the above setup, but whenever that is not the case, we will write $\intstruct_k^{\tree}(x)$ to refer to the internal structure of $x \in \insertablef[s]$ in $\tree \in \T_{[n]}$, where $\tree[s] = \rho \left( \tree, [n] \setminus [k] \right)$.

Before moving on to characterizing the (random) internal structure of a (random) decorated tree, we will need the following growth procedures, which are only slightly different from the $(\alpha, \gamma)$-growth process, and arise from the internal structure of internal edges and branch points, respectively.
\begin{defi}[Internal $(\alpha, \gamma)$-growth process]\label{def:internalalphagammagrowth}
  Fix $0 < \gamma \leq \alpha \leq 1$.
  Let $T_0^\gamma$ be the unique element of $\T_{[1]}$.
  For any $n \in \N_0$ construct $T_{n+1}^\gamma$ conditional on $T_n^\gamma = \tree \in \T_{[n + 1]}$, by
  \begin{align*}
    w_x
    = 
    \begin{cases}
      1 - \alpha & \text{if $x \in \edge \left( \tree[s] \right)$ is external and $x \neq e_{1\downarrow}$} \\
      \gamma & \text{if $x \in \edge \left( \tree[s] \right)$ is internal or $x = e_{1\downarrow}$} \\
      \left( c_x - 1 \right) \alpha - \gamma & \text{if $x \in \branchpoints (\tree[s])$ with $c_x$ children} \\
    \end{cases}
  \end{align*}
  and then defining the conditional distribution of $T_{n+1}^\gamma$ by
  \begin{align*}
    \P \left( T_{n+1}^\gamma = \varphi \left( \tree[s], x, n+1 \right) \ \big \vert \ T_n^\gamma = \tree[s] \right)
    = \frac{w_x}{n + \gamma},
  \end{align*}
  for each $x \in \insertablef[t]$.
  The sequence ${\left( T_n^\gamma \right)}_{n \in \N_0}$ is referred to as the \textit{internal $(\alpha, \gamma)$-growth process}. 
\end{defi}
\begin{defi}[$c$-order branch point $(\alpha, \gamma)$-growth process]\label{def:branchpointalphagammagrowth}
  Fix $0 \leq \gamma < \alpha \leq 1$ and $c \geq 2$.
  Let $T_0^{c-\text{bp}}$ be the unique element of $\T_{[c]}$ with only one branch point.
  For any $n \in \N_0$ construct $T_{n+1}^{c-\text{bp}}$ conditional on $T_n^{c-\text{bp}} = \tree \in \Thatspace_{[c + n]}$, by setting
  \begin{align*}
    w_x
    = 
    \begin{cases}
      1 - \alpha & \text{if $x \in \edge \left( \tree[s] \right) \setminus \{e_{1\downarrow}, \ldots, e_{c\downarrow}\}$ is external} \\
      \gamma & \text{if $x \in \edge \left( \tree[s] \right)$ is internal and $x \neq e_{\emptyset \uparrow}$} \\
      \left( c_x - 1 \right) \alpha - \gamma & \text{if $x \in \branchpoints (\tree[s])$ with $c_x$ children} \\
      0 & \text{if}\ x \in \{e_{\emptyset \uparrow}, e_{1\downarrow}, \ldots, e_{c\downarrow}\}
    \end{cases}
  \end{align*}
  and then defining the conditional distribution of $T_{n+1}^{c-\text{bp}}$ by
  \begin{align*}
    \P \left( T_{n+1}^{c-\text{bp}} = \varphi \left( \tree[s], x, n+1 \right) \ \big \vert \ T_n^{c-\text{bp}} = \tree[s] \right)
    = \frac{w_x}{n + (c-1)\alpha - \gamma},
  \end{align*}
  for each $x \in \insertablef[t]$.
  The sequence ${\left( T_n^{c-\text{bp}} \right)}_{n \in \N_0}$ is referred to as the \textit{$c$-order branch point $(\alpha, \gamma)$-growth process}. 
\end{defi}
As the growth processes defined above are very similar to the $(\alpha, \gamma)$-growth process described in Section~\ref{sec:semiplanar_alphagamma}, a wide range of results for the $(\alpha, \gamma)$-growth process can can be modified to fit the internal growth processes of Definitions~\ref{def:internalalphagammagrowth} and~\ref{def:branchpointalphagammagrowth}.
It should also be clear that we can define a semi-planar version of the above growth procedures, just as we did for the $(\alpha, \gamma)$-growth process.
For the $c$-order branch point process, this entails sampling an initial distribution to provide the order of the initial $c$ leaves in the first branch point, as described by~\eqref{eq:internalstructure}.
We will refer to these processes as the \textit{semi-planar internal $(\alpha, \gamma)$-growth process} and the \textit{semi-planar $c$-order branch point $(\alpha, \gamma)$-growth process}.

From the spinal decomposition of the $(\alpha, \gamma)$-growth process~\cite{RefWorks:doc:5b4cbb5fe4b02dc0c79270af}, it is clear that there is an equally intimate link between (un)ordered Chinese Restaurant Processes and the internal $(\alpha, \gamma)$-process.
The only difference is that the ordered Chinese Restaurant involved in the spinal decomposition outlined in Lemma 11 of~\cite{RefWorks:doc:5b4cbb5fe4b02dc0c79270af} is a $(\gamma, \gamma)$-$\ocrp$ rather than a $(\gamma, 1-\alpha)$-$\ocrp$.
This $(\gamma, \gamma)$-$\ocrp$ is the same as the one appearing in Proposition~\ref{prop:decorated_transition_kernel}.
From similar considerations we get the following characterization of the first split in the $c$-order branch point $(\alpha, \gamma)$-growth process:
\begin{lemma}%
\label{lemma:k_bp_alphagammaprocess}
Fix $0 < \gamma < \alpha \leq 1$ and $c \geq 2$.
Let ${\left( \That_{n}^{c-\text{bp}} \right)}_{n \in \N_0}$ denote the semi-planar $c$-order branch point $(\alpha, \gamma)$-process.
For each $\nin$ and $l \in [c-1]$ let $Y_l^n$ denote the number of leaves with labels larger than $c$ that have been inserted in the $l$'th gap (enumerated from left to right) in the branch point closest to the root in $\That_n^{c-\text{bp}}$.
In addition, let ${\That[S]}_l^n = \left( {\That[S]}_{l,1}^n, \ldots, {\That[S]}_{l,M_l}^n \right)$, enumerated from left to right, denote the subtrees in the $l$'th gap of that same branch point, let $X_l^n$ denote the associated ordered partition relabelled using the increasing bijection $\{c + 1, \ldots, c + n\} \mapsto \{1, \ldots, n\}$, and let ${\That[S]}_{l,j}^{n, \text{re}}$ denote ${\That[S]}_{l,j}^{n}$ relabelled by $\{1, \ldots, \# {\That[S]}_{l,j}^n \}$ using the increasing bijection for each $j \in [M_l], l \in [c - 1]$. 
  Then, $(Y_1^n,\ldots, Y_{c-1}^n) \sim \dirmult^n(\alpha, \ldots, \alpha, \alpha - \gamma)$, and conditionally given $\left( Y_1^n, \ldots, Y_{c-1}^n \right) = \left( y_1, \ldots, y_{c-1} \right)$, ${\left( X_l^n \right)}_{l \in [c-1]}$ are independent and
  \begin{align*}
    \forall 1 \leq l < c-1 \colon \ X_l^n \sim \ocrp^{y_l} \left( \alpha, \alpha \right)
    \quad \text{and} \quad
    X_{c-1}^n \sim \ocrp^{y_{c-1}} \left( \alpha, \alpha - \gamma \right).
  \end{align*}
  Further conditioning on $\left( \# {\That[S]}_{l,1}^n, \ldots, \# {\That[S]}_{l,M_l}^n\right) = \left( n_{l,1}, \ldots, n_{l,M_l} \right)$ yields that \\
  ${\That[S]}_{l,1}^{n, \text{re}}, \ldots, {\That[S]}_{l,M_l}^{n, \text{re}}$ are independent with ${\That[S]}_{l,j}^{n, \text{re}} \deq \That_{n_{l,j}}$ for each $j \in [M_l], l \in [c-1]$.
\end{lemma}
\begin{proof}
  All of the above assertions follows from (by now) standard growth process arguments, as one notes that the only possible insertions in these processes are either into the branch point described above, or into a subtree of that branch point with a least label larger than $k$.
  In the semi-planar tree, consider the insertion of these leaves into the “gaps” of the branch point between the subtrees with least leaf label smaller than $k$.
  There are $c-1$ such gaps, of which $c-2$ has weight $\alpha$ and the rightmost one has weight $\alpha - \gamma$.
  Hence the numbers of leaves being inserted into the gaps follow a P\'{o}lya urn scheme with initial weights $\mathbf{w} = (\alpha, \ldots, \alpha, \alpha - \gamma) \in {[0,1]}^{c-1}$, and Proposition~\ref{prop:polyaurn} directly yields that $\left( Y_1^n, \ldots, Y_{c-1}^n \right) \sim \dirmult^n(\mathbf{w})$.
  When conditioning on which gap a new leaf is inserted into, the location within that gap is clearly governed by an $\ocrp$, the distribution of which is detailed by the same argument as in Lemma~\ref{lemma:internalstructure}.
  And lastly, when conditioning on which leaves end up in the same subtree of $v$, each of these subtrees will after relabelling be a semi-planar $(\alpha, \gamma)$-tree with the corresponding number of leaves. 
\end{proof}
\begin{cor}
  Fix $0 < \gamma < \alpha \leq 1$ and $c \geq 2$.
  Let ${\left( T_{n}^{c-\text{bp}} \right)}_{n \in \N_0}$ denote the (non-planar) $c$-order branch point $(\alpha, \gamma)$-process.
  If $S_1^n, \ldots, S_L^n$ denotes the subtrees with least label larger than $c$ in the branch point closest to the root of $T_n^{c-\text{bp}}$, enumerated in the order of least leaf label, then
  \begin{align*}
    \P \left( L = l, S_1^n = {\tree[s]}_1, \ldots, S_m^n = {\tree[s]}_l \right)
    = \
    &\P \left( T_1' = {\tree[s]}_1^\text{re}, \ldots, T_k' = {\tree[s]}_l^\text{re} \right) \\ 
    &\ \times \P \left( \left(N_1, \ldots, N_L \right) = \left( n_1, \ldots, n_l \right) \right)
  \end{align*}
  where ${\tree[s]}_{l^\prime}^\text{re}$ is ${\tree[s]}_{l^\prime}$ relabelled by $[n_{l^\prime}]$ using the increasing bijection and $T_{l^\prime}^\prime \deq T_{n_{l^\prime}}$ for each $l^\prime \in [l]$, ${\left( T_{l^\prime}^\prime \right)}_{l^\prime \in [l]}$ are independent, and $\left( N_1, \ldots, N_L \right)$ are the numbers of customers sitting at the $L$ tables of an $(\alpha, (c - 2)\alpha - \gamma)$-$\crp$ with $n$ customers.
\end{cor}
We will now prove some useful versions of Lemma~\ref{lemma:condind} and Theorem~\ref{thm:planarstationarity} for the semi-planar $(\alpha, \gamma)$-growth process with a special tree as a starting point $\that_0 \in \T_{[n_0]}$ with associated weights $\mathbf{w}_0 = {\left( w_{x, l}^0 \right)}_{x \in \insertable (\that_0), 0 \leq l \leq c_x - 1}$, where we use the convention that $c_x = 1$ for $x \in \edge(\that_0)$ and $1 \leq l \leq c_x - 1$ if $x \in \branchpoints (\that)$ with $c_x$ children.
Crucially, we allow these weights to be different from the weight specification of the growth process.
The result that we will now outline have wider implications for label swapping in more general structures that induce partitions and come about from a growth procedure, such as ordered Chinese Restaurant Processes.
However, we will not investigate this further in this paper.

In the following we need to be careful about `where' the modified weights appear after the insertion of a leaf according to the growth rule.
Hence for any $\that \in \Thatspace_{[n]}$ consider how $\that^\prime = \hat{\varphi}\left( \that, n+1, x, l \right)$ is formed by carrying out one step of the semi-planar $(\alpha, \gamma)$-growth rule, where $x \in \insertable (\that)$ and $0 \leq l \leq c_x - 1$.
We will organize the weights of $\that^\prime$ such that
\begin{align}
  \label{eq:organized_weight_schemes}
  w_{x^\prime,l^\prime}^\prime
  =
  \begin{cases}
    \alpha & x^\prime = x \in \branchpoints (\that), l^\prime = l \\
    w_{x, l - 1} & x^\prime = x \in \branchpoints (\that), l \leq l^\prime \leq c_{x^\prime} - 1 \\
    \alpha - \gamma & x^\prime = (n+1)\downarrow \in \branchpoints (\that^\prime), x \in \edge (\that), l^\prime = 1 \\
    \gamma & x = e_{uv} \in \edge (\that), x^\prime = e_{u \left( n+1 \right) \downarrow}, l^\prime = 0 \\
    w_{x, 0} & x = e_{uv} \in \edge (\that), x^\prime = e_{\left( n+1 \right) \downarrow v}, l^\prime = 0 \\
    1-\alpha & x^\prime = n+1 \in \edge (\that^\prime), l^\prime = 0 \\
    w_{x, l} & \text{otherwise.}
  \end{cases}
\end{align}
for every $x^\prime \in \insertable (\that^\prime)$ and $0 \leq l^\prime \leq c_{x^\prime} - 1$.
This is illustrated in Figure~\ref{fig:weight_ordering}.
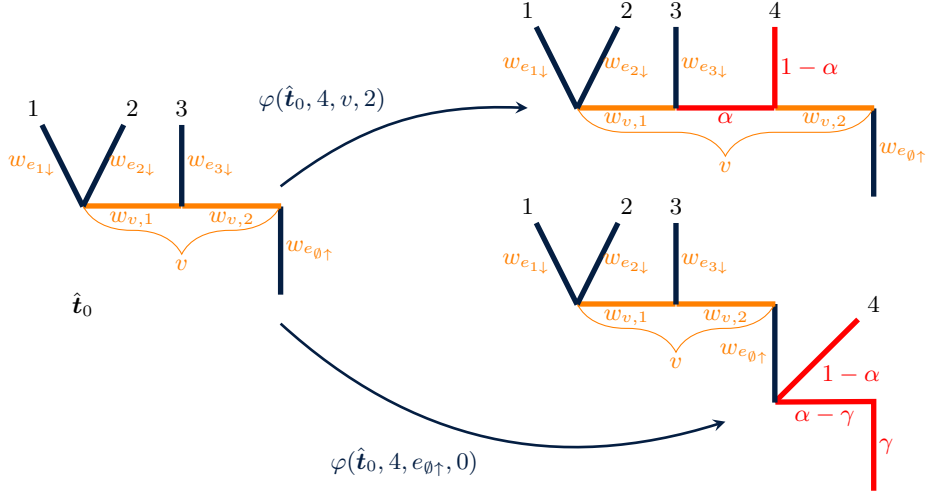
\begin{figure}[t]\label{fig:weight_ordering}
  \centering
  \begin{tikzpicture}[%
    bp/.style={rectangle, draw, inner sep=0pt, minimum width=0pt, minimum height=0pt},
    comment/.style={circle, draw, color=red},
    leaf/.style={inner sep=1pt},
    weights/.style={auto, swap, orange, inner sep = 1pt},
    label/.style={text=oxfordblue},
    line/.style={-, line width = 2pt, color = oxfordblue},
    arrowline/.style={->, line width = 1pt, color = oxfordblue, >=stealth},
  scale=1.3]
  \tikzstyle{every node}=[font=\small]

  \begin{scope}[xscale = -1]
    \node (0) at (0,0) {};
    \node[bp] (bp_start) at (0,1) {};
    \node[bp] (bp_mid) at (1, 1) {};
    \node[bp] (bp_end) at (2, 1) {};
    \node (1) at (2.5, 2) {$1$};
    \node (2) at (1.5, 2) {$2$};
    \node (3) at (1, 2) {$3$};

    \draw[line] (0) to node[weights] {$w_{e_{\emptyset \uparrow}}$} (bp_start);
    \draw[line, color = orange] (bp_start.east) to node[weights, swap] {$w_{v, 2}$} (bp_mid) to node[weights, swap] {$w_{v, 1}$} (bp_end.east);
    \draw[line] (bp_mid.south) to node[weights] {$w_{e_{3 \downarrow}}$} (3);
    \draw[line] (bp_end.south) to node[weights, swap, anchor = east] {$w_{e_{1 \downarrow}}$} (1);
    \draw[line] (bp_end.south) to node[weights, anchor = west] {$w_{e_{2 \downarrow}}$} (2);

    \draw [decorate,decoration={brace,amplitude=18pt}, orange] (bp_start.south west) to (bp_end.south east);
    \node [orange] at (1, 0.4) {$v$};
    \node at (2, 0) {$\that_0$};
  \end{scope}
  \begin{scope}[shift = {(6,1)}, xscale = -1]
    \node (0) at (0,0) {};
    \node[bp] (bp_start) at (0,1) {};
    \node[bp] (bp_mid1) at (1, 1) {};
    \node[bp] (bp_mid2) at (2, 1) {};
    \node[bp] (bp_end) at (3, 1) {};
    \node (1) at (3.5, 2) {$1$};
    \node (2) at (2.5, 2) {$2$};
    \node (3) at (2, 2) {$3$};
    \node (4) at (1, 2) {$4$};

    \draw[line] (0) to node[weights] {$w_{e_{\emptyset \uparrow}}$} (bp_start);
    \draw[line, color = orange] (bp_start.east) to node[weights, swap] {$w_{v, 2}$} (bp_mid1);
    \draw[line, color = red] (bp_mid1) to node[weights, swap, red] {$\alpha$} (bp_mid2);
    \draw[line, color = orange] (bp_mid2) to node[weights, swap] {$w_{v, 1}$} (bp_end.east);
    \draw[line] (bp_mid2.south) to node[weights] {$w_{e_{3 \downarrow}}$} (3);
    \draw[line, color = red] (bp_mid1.south) to node[weights, color = red] {$1-\alpha$} (4);
    \draw[line] (bp_end.south) to node[weights, swap, anchor = east] {$w_{e_{1 \downarrow}}$} (1);
    \draw[line] (bp_end.south) to node[weights, anchor = west] {$w_{e_{2 \downarrow}}$} (2);

    \draw [decorate,decoration={brace,amplitude=18pt}, orange] (bp_start.south west) to (bp_end.south east);
    \node [orange] at (1.5, 0.4) {$v$};
  \end{scope}
  \begin{scope}[shift = {(5,-1)}, xscale = -1]
    \node[bp] (0) at (0,0) {};
    \node[bp] (bp_start) at (0,1) {};
    \node[bp] (bp_mid) at (1, 1) {};
    \node[bp] (bp_end) at (2, 1) {};
    \node (1) at (2.5, 2) {$1$};
    \node (2) at (1.5, 2) {$2$};
    \node (3) at (1, 2) {$3$};
    
    \node (4) at (-1, 1) {$4$};
    \node[bp] (bp_extra) at (-1, 0) {};
    \node (new_root) at (-1, -1) {};
    \draw[line, color = red] (new_root) to node[weights, color = red] {$\gamma$} (bp_extra.north) to node[weights, color = red, swap] {$\alpha - \gamma$} (0.east);
    \draw[line, color = red] (0.center) to node[weights, color = red] {$1 - \alpha$} (4);

    \draw[line] (0.south) to node[weights, swap] {$w_{e_{\emptyset \uparrow}}$} (bp_start.center);
    \draw[line, color = orange] (bp_start.east) to node[weights, swap] {$w_{v, 2}$} (bp_mid) to node[weights, swap] {$w_{v, 1}$} (bp_end.east);
    \draw[line] (bp_mid.south) to node[weights] {$w_{e_{3 \downarrow}}$} (3);
    \draw[line] (bp_end.south) to node[weights, swap, anchor = east] {$w_{e_{1 \downarrow}}$} (1);
    \draw[line] (bp_end.south) to node[weights, anchor = west] {$w_{e_{2 \downarrow}}$} (2);

    \draw [decorate,decoration={brace,amplitude=18pt}, orange] (bp_start.south west) to (bp_end.south east);
    \node [orange] at (1, 0.4) {$v$};
  \end{scope}
  \draw[arrowline] (0, -0.2) to [bend right] node[anchor = north east] {$\varphi(\that_0, 4, e_{\emptyset \uparrow}, 0)$} (4.5, -1.2);
  \draw[arrowline] (0, 1.2) to [bend left = 20] node[anchor = south east] {$\varphi(\that_0, 4, v, 2)$} (2.5, 2);
\end{tikzpicture}    
  \caption{Illustration of the structuring of weights following a leaf insertion into a tree with pre-specified weights that might not coincide with the usual weights of the $(\alpha, \gamma)$-growth process, as described by~\eqref{eq:organized_weight_schemes}.
  The weights appearing as a result of the leaf insertion are marked in red.}
\end{figure}
The \textit{semi-planar $(\alpha, \gamma)$-growth process started from $\that_0$} is simply the same growth process as in Definition~\ref{def:planaralphagamma}, started from $\that_0 \in \Thatspace_{[n_0]}$, i.e.\ where
\begin{align}
  \P \left( \That_1 = \varphi \left( \that_0, x, l, n_0 + 1 \right) \right)
  = \frac{w_{x, l}^0}{\sum_{x, l} w_{x, l}^0}
  \label{eq:first_step_generalstart}
\end{align}
and we iteratively use the weight scheme defined in~\eqref{eq:organized_weight_schemes} to form $T_{n+1}$ from $T_n$ for every $\nin$.

We note that the above specification of the weights is consistent with the original definition of the semi-planar $(\alpha, \gamma)$-growth process, if the weights $\mathbf{w}_0$ above are the same as those specified in Definition~\ref{def:planaralphagamma}.
The importance of organizing the weights as in~\eqref{eq:organized_weight_schemes} is that we can handle explicitly starting a semi-planar $(\alpha, \gamma)$-growth process from a tree with associated weights different from the ones coming out of the semi-planar $(\alpha, \gamma)$-growth process.
This is the case for the semi-planar internal $(\alpha, \gamma)$-growth process as well as the semi-planar $c$-order branch points $(\alpha, \gamma)$-growth process.
We have the following result, where we recall that $f$ is the transformation function from Definition~\ref{def:planaralphagammachain}.
\begin{prop}
  Fix $n_0 \in \N$ and $\that_0 \in \T_{[n_0]}$ with associated non-negative weights $\mathbf{w}_0 = {\left( w_{x, l}^0 \right)}_{x \in \insertable (\that_0), 0 \leq l \leq c_x - 1}$ such that $\sum_{x, l} w_{x,l}^0 > 0$.
  Let ${\left( \That_m \right)}_{1 \leq m \leq n}$ be the first $n$ steps of the semi-planar $(\alpha, \gamma)$-growth process started from $\that_0$.
  Fix $n_0 + 1 \leq i \leq \tildei \leq n_0 + n$ and define $E_{i, \tildei} = \left\{ \tildei = \max \{i, a, b \} \ \text{in} \ \That_n \right\}$, with $a$ and $b$ from Definition~\ref{def:planaralphagammachain}.
  Then $E_{i, \tildei} \independent \That_{\tildei - 1}$ and, conditional on the event $E_{i, \tildei}$, $\hat{\rho}^* \left( \hat{\tau}^* \left( \That_n, i, \tildei \right), \tildei \right)$ and $T_{n-1}$ have the same distribution.
Furthermore, if $I \sim \Unif \left( [n_0 + n] \setminus [n_0] \right)$, then
  \begin{align}
    \P \left( \hat{\rho}^* \left( f\left( \That_n, I \right) \right) \in \cdot \right) = \P \left( \That_{n-1} \in \cdot \right)
    \label{eq:generalresult_downstep_downstepdistribution}
  \end{align}
  for any $n \in \N$.
  \label{prop:generalresult_downstep_stationary_lemma_theorem}
\end{prop}
\begin{proof}
  The proof of Lemma~\ref{lemma:condind} can used almost verbatim here, as we note that the restrictions on $i$ and $\tildei$ ensure none of the insertable parts of the trees $\That_{\tildei - 1}, \ldots, \That_n$ that contribute to the characterization of $E_{i, \tildei}$ have $\mathbf{w}_0$ as associated weights.
  This is a consequence of the way we have organized the weights in the growth process, combined with the definition of $a$ and $b$ from Definition~\ref{def:planaralphagammachain}.
\end{proof}
Note how Proposition~\ref{prop:generalresult_downstep_stationary_lemma_theorem} generalises Lemma~\ref{lemma:condind}.
This generalisation is useful, as have have the following corollaries, which are easily checked to satisfy the assumptions of Proposition~\ref{prop:generalresult_downstep_stationary_lemma_theorem}.
\begin{cor}%
\label{cor:internal_downupchain_stationary_lemma_theorem}
  Let ${\left( \That_m^\gamma \right)}_{m \in [n]}$ be the first $n$ steps of the semi-planar internal $(\alpha, \gamma)$-growth process.
Fix $2 \leq i \leq \tildei \leq n$ and, with $a$ and $b$ as in Definition~\ref{def:planaralphagammachain}, define $E_{i, \tildei}^\gamma = \left\{ \tildei = \max \{i, a, b\}\ \text{in}\ \That_n^\gamma \right\}$.
  Then $E_{i, \tildei}^\gamma \independent \That_{i-1}^\gamma$ and, conditional on the event $E_{i, \tildei}^\gamma$, $\hat{\rho} \left( \hat{\tau} \left( \That_n^\gamma, i, \tildei \right), \tildei \right)$ and $\That_{n-1}^\gamma$ have the same distribution.
  Furthermore, if $I \sim \Unif \left( \{2, \ldots, n + 1\} \right)$, then
  \begin{align}
    \P \left( \hat{\rho}^* \left( f\left( \That_n^\gamma, I \right) \right) \in \cdot \right) = \P \left( \That_{n-1}^\gamma \in \cdot \right)
    \label{eq:internal_downupchain_downstepdistribution}
  \end{align}
  for any $n \in \N$.
\end{cor}
\begin{cor}%
\label{cor:korder_downupchain_stationary_lemma_theorem} 
  Fix $c \geq2$.
  Let ${\left( \That_m^{c-\text{bp}} \right)}_{m \in [n]}$ be the first $n$ steps of the semi-planar $c$-order branch point $(\alpha, \gamma)$-growth process started from $\that[s] \in \Thatspace_{[c]}$ with $\# \branchpoints (\that[s]) = 1$.
  Fix $c < i \leq \tildei \leq c + n$ and, with $a$ and $b$ as in Definition~\ref{def:planaralphagammachain}, define $E_{i, \tildei}^{c-\text{bp}} = \left\{ \tildei = \max \{i, a, b\}\ \text{in}\ \That_n^{c-\text{bp}} \right\}$.
  Then $E_{i, \tildei}^{c-\text{bp}} \independent \That_{i-1}^{c-\text{bp}}$ and, conditional on the event $E_{i, \tildei}^{c-\text{bp}}$, $\hat{\rho} \left( \hat{\tau} \left( \That_n^{c-\text{bp}}, i, \tildei \right), \tildei \right)$ and $\That_{n-1}^{c-\text{bp}}$ have the same distribution.
  Furthermore, if $I \sim \Unif \left( \{ c + 1, \ldots, c + n\} \right)$, then
  \begin{align}
    \P \left( \hat{\rho}^* \left( f\left( \That_n^{c-\text{bp}}, I \right) \right) \in \cdot \right) = \P \left( \That_{n-1}^{c-\text{bp}} \in \cdot \right)
    \label{eq:korder_downupchain_downstepdistribution}
  \end{align}
  for any $n \in \N$.
\end{cor}
%
%\begin{cor}
%  \fxfatal{I don't really need this result anymore. Useful to have????}
%  Let $X = {\left( X_n \right)}_\nin \sim \ocrp(\alpha, \theta)$.
%  Fix $1 \leq \tildei \leq n$ and define
%  %
%  \begin{align*}
%    E_{i, \tildei}
%    = \begin{cases}
%      \max \{i, b\} & \text{if $i$ is the only customer at a table, and} \\
%      \max \{i, A\} & \text{otherwise,}
%    \end{cases}
%  \end{align*}
%  where $A$ is a random variable only taking values on the lable set $\{a \in [n] \ \middle \vert \a \sim i, a \geq i\}$ the minimum label not equal to $i$ placed at the same table as $i$, and $b$ is the minimum label found in the first table to the left of the table containing $i$.
%  Then $E_{i, \tildei} \independent \That_{i-1}$ and, conditional on the event $E_{i, \tildei}$, $\rho \left( \tau \left( X_n, i, \tildei \right), \tildei \right)$ and $X_{n-1}$ has the same distribution.
%  Furthermore, if $I \sim \Unif \left( \{1, \ldots, n\} \right)$, then
%  %
%  \begin{align}
%    \P \left( \rho \left( f\left( X_n, I \right) \right) \in \cdot \right) = \P \left( X_{n-1} \in \cdot \right)
%    \label{eq:ocrp_downupchain_downstepdistribution}
%  \end{align}
%  %
%  for any $n \in \N$.
%  \label{cor:ocrp_downupchain_stationary_lemma_theorem}
%\end{cor}
%
%
% End of new material
%
%
With the above machinery in place we can now characterize the decorated trees in the following way:
\begin{lemma}[Spatial Markov property]\label{lemma:spatialMarkovProperty}
  Fix $n \in \N$.
  Let ${\left( T_m \right)}_{m \in [n]}$, ${\left( T_m^\gamma \right)}_{m \in [n]}$, and ${\left( T_m^{c-\text{bp}} \right)}_{m \in [n]}$ denote the first $n$ steps of the $(\alpha, \gamma)$-, internal $(\alpha, \gamma)$-, and $c$-order branch point $(\alpha, \gamma)$-growth processes, respectively.
  Fix $\tree \in \T_{[n]}$, and define $\tree^* = \pi_{[k]}^{*n} (\tree) = (\tree[s], {(g^{-1}(\{x\})}_{x \in \insertablef[s]}) \in \T_{[k]}^{*n}$ and $\tree^\bullet = \pi_{[k]}^{\bullet n} (\tree) = (\tree[s], \mathbf{y}) \in \T_{[k]}^{\bullet n}$.
Then
\begin{enumerate}
  \item\label{lemma:spatial1} ${\left( \intstruct (x) \right)}_{x \in \insertablef[s]}$ are conditionally independent given $\pi_{[k]}^{\bullet n} \left( T_n \right) = \tree^\bullet$,
  \item\label{lemma:spatial1a} ${\left( \intstruct (x) \right)}_{x \in \insertablef[s]}$ are conditionally independent given $\pi_{[k]}^{* n} \left( T_n \right) = \tree^*$,
  \item\label{lemma:spatial2} ${\left( \intstruct (x) \right)}_{x \in \insertablef[s]}$ and ${\left( g^{-1}(\{x\}) \right)}_{x \in \insertablef[s]}$ are conditionally independent given $\pi_k^{\bullet n} \left( T_n \right) = \tree^\bullet$,
    \item\label{lemma:spatial3} $\P \left( \intstruct (x) = \cdot \ \middle \vert \ \pi_k^{\bullet n} \left( T_n \right) = \tree^\bullet \right) = \P(T_{y_x}^\gamma = \cdot )$ for each internal $x \in \edge (\tree[s])$,
    \item\label{lemma:spatial4} $\P \left( \intstruct (x) = \cdot \ \middle \vert \ \pi_k^{\bullet n} \left( T_n \right) = \tree^\bullet \right) = \P(T_{y_x} = \cdot )$ for each external $x \in \edge (\tree[s])$, and
    \item\label{lemma:spatial5} $\P \left( \intstruct (x) = \cdot \ \middle \vert \ \pi_k^{\bullet n} \left( T_n \right) = \tree^\bullet \right) = \P(T_{y_x}^{c-\text{bp}} = \cdot )$ for each $x \in \branchpoints (\tree[s])$ with $c$ children.
  \end{enumerate}
  \label{lemma:InternalStructure}
\end{lemma}
\begin{proof}
  Firstly note that the conditional independence in~\ref{lemma:spatial1a} follows directly from the independence of the insertions of new leaves in the $(\alpha, \gamma)$-growth process.
  This observation directly implies~\ref{lemma:spatial1}.
  For~\ref{lemma:spatial2}, fix $x \in \insertablef[s]$, and note that the probability of $\intstruct (x)$ being equal to a specific tree is independent of the labels, by virtue of the growth process.
  Next,~\ref{lemma:spatial3} and~\ref{lemma:spatial4} follows by noting how $T_n$ is produced from $T_k = \tree[s]$, since the latter is a part of what we condition on.
  For any $x \in \edge(\tree[s])$, note how $\intstruct_k^{T_k} (x)$ will be the unique element of $\T_{[1]}$, where the edge has associated weight $\gamma$ if $x$ is an internal edge and $1-\alpha$ if $x$ is a leaf edge, respectively.
  Now observe how each $\intstruct (x)$ grows according to either the $(\alpha, \gamma)$- or the internal $(\alpha, \gamma)$-growth process, according to the initial weight of the leaf. 
  For (vi), note how for fixed $x \in \branchpoints(\tree[s])$ with $y_x > 0$ we can construct $\intstruct (x)$ using the $c$-order branch point $(\alpha, \gamma)$-growth process.
  The insertion of a leaf into the branch point, $x$, corresponds to opening up a new table in a $\crp$, and for every such insertion the weight at $x$ increases by $\alpha$.
  By construction of $\intstruct (x)$ all of these subtrees will be located at the branch point closest to the root in $\intstruct (x)$, and each of the subtrees evolves as a $(\alpha, \gamma)$-growth process.
  Hence we obtain the above characterization.
\end{proof}
In addition to the above characterization we have the result below, which will play a role in the relabelling of a decorated $[k]$-tree later on, where we recall that $\tilde{y}_x$, as in Definition~\ref{def:resample}, is equal to $y_x - 1$ for external $x \in \edge(\tree[s])$ and equal to $y_x$ otherwise.
To ease notation, we will for $(\tree[s], \mathbf{y}) \in \T_{[k]}^{\bullet n}$ set
\begin{align*}
  \tilde{y}_x =
  \begin{cases}
    y_x - 1 & \text{if $x \in \edge(\tree[s])$ is external,} \\
    y_x & \text{otherwise,}
  \end{cases}
\end{align*}
for all $x \in \insertablef[s]$.
\begin{lemma}
  Let ${\left( T_m \right)}_{1 \leq m \leq n}$ denote the first $n$ elements of the $(\alpha, \gamma)$-growth process.
  Fix $\tree \in \T_{[n]}$ and define $\pi_{[k]}^{\bullet n} ( \tree ) = (\tree[s], \mathbf{y}) \in \T_{[k]}^{\bullet n}$.
  Then
  \begin{align*}
    \P \left( T_{k+1} = \varphi (\tree[s], x, k+1) \ \middle \vert \ \pi_{[k]}^{\bullet n} \left( T_n \right) = (\tree[s], y) \right)
    = \frac{\tilde{y}_x}{n-k}
  \end{align*}
  for each $x \in \insertablef[s]$.
  \label{lemma:locationof_k_plus1}
\end{lemma}
\begin{proof}
  Note how $T_n$ is constructed from ${T_k = \tree[s]}$ in the $(\alpha, \gamma)$-growth process, if we are solely interested in the number of leaves being inserted into each $x \in \insertablef[s]$.
  Then this is simply a generalized P\'{o}lya urn scheme with weights corresponding to the weight associated with $x$ in $\tree[s]$.
  In this analogy we are interested in the colour of the first ball drawn from this urn scheme where we have a total of $n-k$ draws.
  According to Proposition~\ref{prop:polyaurn}~\ref{prop:polyaurn_conditioning} this is exactly $\frac{\tilde{y}_x}{n-k}$ for each $x \in \insertablef[s]$.
\end{proof}
The above proposition states that, if we follow an $(\alpha, \gamma)$-growth process and condition on a specific decorated tree, the probability that we originally inserted leaf $k+1$ in a specific location of the tree shape associated with the decorated tree, only depends on the total number of insertions we made into that part of the tree shape.
\section{Consistency in stationarity}
Having obtained an autonomous description of the decorated $(\alpha, \gamma)$-chain, we now turn our focus to proving that the projection of the $(\alpha, \gamma)$-chain on $\Thatspace_{[n]}$ onto the space of decorated $[k]$-trees of mass $n$, ${\left( \hat{\pi}_{[k]}^{\bullet n} \left( \That_n (m) \right) \right)}_{m \in \N_0}$, is indeed the decorated $(\alpha, \gamma)$-chain on $\T_{[k]}^{\bullet n}$.
For brevity we have used the notation that $\hat{\pi}_{[k]}^{\bullet n} := \pi_{[k]}^{\bullet n} \circ \hat{\pi}$.
As noted in~\cite{RefWorks:doc:5b4cbc93e4b07f5746e47014}, this is an instance of a more general question:
under what conditions is a surjective function of a Markov chain Markovian in its on right?
To be more precise, let ${\left( X_m \right)}_{m \in \N_0}$ be a Markov chain on some finite state space $\XX$, and let $\lambda \colon \XX \to \YY$ be a function.
Then the question is, if $Y_m = \lambda \left( X_m \right)$, $m \in \N_0$, is a Markov chain.
We will use the following two well-known criteria.
\begin{prop}%
  \label{prop:intertwining_kemenysnell}
  Let ${\left( X_m \right)}_{m \in \N_0}$ be a Markov chain on a finite state space $\XX$ with transition kernel $K$.
  Assume that $\lambda \colon \XX \to \YY$ is a surjective function.
  \begin{enumerate}
    \item\label{prop:kemenysnell_criterion} If for each $x_1, x_2 \in \XX$ such that $\lambda(x_1) = \lambda(x_2)$ it holds that
      \begin{align}
        K \left( x_1, \lambda^{-1}(\{y\}) \right)
        =
        K \left( x_2, \lambda^{-1}(\{y\}) \right)
        \qquad \text{for any $y \in \YY$},
        \label{eq:kemeneysnell}
      \end{align}
      define the transition kernel, $Q$, by $Q(y_1, \{ y_2 \}) = K(x, \lambda^{-1}(\{y_2\}))$ for any $x \in \lambda^{-1}(\{y_1\})$.
    \item\label{prop:intertwining_criterion} Let $q$ be a stationary distribution of ${\left( X_m \right)}_{m \in \N_0}$, and further assume that $X_0 \sim \Lambda(Y_0, \cdot)$ for some $\YY$-valued random variable $Y_0$, where $\Lambda(y, \cdot) = q \left( \cdot \ \middle \vert \ \lambda = y \right)$.
      If for each $y \in \YY$ it holds that
      \begin{align}
        \P \left( X_1 = x \ \middle \vert \ \lambda(X_0) = y_0, \lambda(X_1) = y \right)
        =
        \Lambda \left( y, \{x\} \right)
        \label{eq:intertwiningcriterion}
      \end{align}
      for all $x \in \XX$ with $\lambda(x) = y$, define $Q = \Lambda K \lambda$ using the same notation as in~\eqref{eq:nonplanartransitionmatrix}.
  \end{enumerate}
  In either case ${\left( \lambda(X_m) \right)}_{m \in \N_0}$ is a Markov chain with transition kernel, $Q$, and initial distribution given by $\lambda(X_0)$.
\end{prop}
In the above,~\ref{prop:kemenysnell_criterion} is referred to as the \textit{Kemeny-Snell criterion}~\cite{MR0115196}, whilst~\ref{prop:intertwining_criterion} is known as the \textit{intertwining criterion}~\cite{MR624684}.
The goal for this section is to combine the two criteria to prove that ${\left( \pi_{[k]}^{\bullet n} \left( T_n(m) \right) \right)}_{m \in \N_0}$ is a Markov chain with transition kernel described in Proposition~\ref{prop:decorated_transition_kernel}.
Illustrated in Figure~\ref{fig:commutativediagram_complete}, we seek to do this by splitting this up, so rather than projecting all the way to decorated trees we will use collapsed trees as an intermediary.
In doing this we note that $\pi_{[k]}^{\bullet n} = h \circ \pi_{[k]}^{* n}$, where $h$ denotes the natural projection from $\T_{[k]}^{* n}$ to $\T_{[k]}^{\bullet n}$, i.e.\ $\tree^* = \left( \tree[s], {\left( B_x \right)}_{x \in \insertablef[s]} \right) \stackrel{h}{\mapsto} \left( \tree[s], {\left( \# B_x \right)}_{x \in \insertablef[s]} \right) = \tree^\bullet$.

We therefore construct an auxiliary Markov chain on $\T_{[k]}^{* n}$ analogously to the way we constructed the decorated $(\alpha, \gamma)$-chain on $\T_{[k]}^{\bullet n}$ in Definition~\ref{def:decorated_alphagamma_chain}.
Define a Markov kernel from $\T_{[k]}^{* n}$ to $\Thatspace_{[n]}$ by
% %
\begin{align}
  \forall \tree^* \in \T_{[k]}^{* n}, \ \that \in \Thatspace_{[n]} \colon \quad
  {\hat{\Pi}}_k^{* n} (\tree^*, \{\that\}) :=
  \P \left( \That_n = \that \ \middle \vert \ \hat{\pi}_{[k]}^{* n} \left(\That_n\right) = \tree^* \right),
  \label{eq:collapsedkernel}
\end{align}
where $\That_n$ is the $n$th step of the semi-planar $(\alpha, \gamma)$-growth process, and $\hat{\pi}_{[k]}^{* n} := \pi_{[k]}^{* n} \circ \hat{\pi}$ denotes the  projection from semi-planar $n$-trees to collapsed $[k]$-trees of mass $n$, for brevity.
Now define the \textit{collapsed $(\alpha, \gamma)$-chain on $\T_{[k]}^{* n}$} be a Markov chain on $\T_{[k]}^{* n}$ with transition kernel 
\begin{align}
  K_k^{* n} := {\hat{\Pi}}_k^{* n} \hat{K}_n \hat{\pi}_{[k]}^{* n},
  \label{eq:transitionkernel_collapsed}
\end{align}
Lastly, we define a Markov kernel from $\T_{[k]}^{\bullet n}$ to $\T_{[k]}^{* n}$ by
%
% %
\begin{align}
  \forall \tree^\bullet \in \T_{[k]}^{* n}, \ \tree^* \in \T_{[k]}^{* n} \colon \quad
  H (\tree^*, \{\that\}) :=
  \P \left( \pi_{[k]}^{* n} \left( \That_n \right) = \tree^* \ \middle \vert \ \hat{\pi}_{[k]}^{\bullet n} \left(\That_n\right) = \tree^\bullet \right).
  \label{eq:kernel_collapsed_decorated}
\end{align}
The goal is thus to show that both the upper and lower part of the diagram in Figure~\ref{fig:commutativediagram_complete} commutes, respectively.
We will deploy the Kemeny-Snell criterion to show that the lower diagram commutes, whilst the intertwining criterion will be used for the upper part of the diagram.
\begin{figure}[t]
  \centering
  \begin{tikzcd}[row sep = large, column sep = large, shape = asymmetrical rectangle]
    \That_n(m) \arrow[r, "\hat{K}_n"] & \That_n(m+1) \arrow[d, bend right, "\hat{\pi}_{[k]}^{* n}"'] \arrow[dd, bend left = 65, "\hat{\pi}_{[k]}^{\bullet n}"]\\
    T_{k}^{* n}(m) \arrow[u, "\hat{\Pi}_k^{* n}"] \arrow[r, red, "K_k^{* n}" red] & T_{k}^{* n}(m+1) \arrow[d, bend right, "h"'] \arrow[u, bend right, "\hat{\Pi}_k^{* n}"'] \\ 
    T_{k}^{\bullet n}(m) \arrow[u, "H"] \arrow[uu, bend left = 60, "\hat{\Pi}_{k}^{\bullet n}"] \arrow[r, red, "K_k^{\bullet n}" red] & T_{k}^{\bullet n}(m+1) \arrow[u, bend right, "H"']
  \end{tikzcd}
  \caption{Commutative diagram illustrating the construction of $K_k^{\bullet n}$ described in Proposition~\ref{prop:decorated_transition_kernel} and the proof tactic to show that the projection of the semi-planar $(\alpha, \gamma)$-chain is indeed a Markov chain with that same transition kernel.}
  \label{fig:commutativediagram_complete}
\end{figure}
We will start by proving that the projection from collapsed trees to decorated trees satisfies the Kemeny-Snell criterion.
\begin{lemma}\label{lemma:kemeny-snell}
  Let ${\left( T_n(m) \right)}_{m \in \N_0}$ denote the $(\alpha, \gamma)$-chain.
  Then 
  \begin{align*}
    &\P \left( \pi_{[k]}^{* n} \left( T_n(1) \right) \in h^{-1}( \{ \tree^\bullet \}) \ \middle \vert \ \pi_{[k]}^{* n} \left( T_n(0) \right) = \tree_1^* \right) \nonumber \\
    = \quad &\P \left( \pi_{[k]}^{* n} \left( T_n(1) \right) \in h^{-1}( \{ \tree^\bullet \}) \ \middle \vert \ \pi_{[k]}^{* n} \left( T_n(0) \right) = \tree_2^* \right)
  \end{align*}
  for all $\tree^\bullet \in \T_{[k]}^{\bullet n}$, whenever $h (\tree_1^*) = h (\tree_2^*)$.
\end{lemma}
\begin{proof}
  Say that $\tree_1^* = \left( {\tree[s]_1}, {\left( B_x^1 \right)}_{x \in \insertable \left({\tree[s]}_1\right)} \right)$ and $\tree_2^* = \left( {\tree[s]_2}, {\left( B_x^2 \right)}_{x \in \insertable \left({\tree[s]}_2\right)} \right)$.
  As $h (\tree_1^*) = h (\tree_2^*)$ we easily see that ${\tree[s]}_1 = {\tree[s]}_2 =: \tree[s]$ and $\# B_x^1 = \# B_x^2$, ensuring that $\P \left( I \in B_x^1 \right) = \P \left( I \in B_x^2 \right)$, for every $x \in \insertablef[s]$.

Now observe that there is a unique bijection $\tau \colon [n] \to [n]$ which is increasing on $B_x^1$ for each $x \in \insertable \left( {\tree[s]}_1 \right)$ such that $B_x^2 = \tau \left( B_x^1 \right)$ for every $x \in \insertable \left( \tree[s] \right)$, and use this bijection to obtain $\tau(\tree)$ from $\tree \in \T_{[n]}$.
  Then it holds that
  \begin{align*}
    \left\{ \tau (\tree) \ \middle \vert \ \pi_{[k]}^{* n} \left( \tree \right) = \tree_1^* \right\}
    = {\left( \pi_{[k]}^{* n} \right)}^{-1} \left( \tree_2^* \right)
  \end{align*}
  and that $\P \left( T_n = \tree \right) = \P \left( T_n = \tau (\tree) \right)$ for each $\tree \in \T_{[n]}$ that satisfies $\pi_{[k]}^{* n} (\tree) = \tree_1^*$.
  Crucially, the rank of each element of any set $B_x$ is preserved under this relabelling.
  Note how the internal structure, $\intstruct (x)$, is completely preserved under $\tau$ by Lemma~\ref{lemma:spatialMarkovProperty}~\ref{lemma:spatial2}, for each $x \in \insertable \left( \tree[s] \right)$, and that the distribution of each $\intstruct (x)$ is described by Lemma~\ref{lemma:spatialMarkovProperty}\ref{lemma:spatial3}-\ref{lemma:spatial5}.
  This immediately implies that $\pi_{[k]}^{\bullet n} \left( T_n(1) \right) = \pi_{[k]}^{\bullet n} \left( T_n^\tau(1) \right)$ (not just in distribution!), where $T_n(\cdot)$ and $T_n^\tau(\cdot)$ are both (non-planar) $(\alpha, \gamma)$-chains but $T_n(0) = \tree$ and $T_n^\tau(0) = \tau (\tree)$.
  This proves that the Kemeny-Snell criterion is satisfied.
\end{proof}
With the bottom diagram of Figure~\ref{fig:commutativediagram_complete} commuting, we only need to prove that the upper part of the same diagram commutes as well.
This is a slightly more intricate argument, where we will aim to show that the intertwining criterion is satisfied with ${\left( X_m \right)}_{m \in \N_0} = {\left( \That_n(m) \right)}_{m \in \N_0}$ and $\lambda = \hat{\pi}_{[k]}^{*n} = \pi_{[k]}^{*n} \circ \hat{\pi} \colon \Thatspace \to \T_{[k]}^{*n}$.
\begin{lemma}\label{lemma:intertwining}
  Let ${\left( \That_n(m) \right)}_{m \in \N_0}$ be the semi-planar $(\alpha, \gamma)$-chain with $\That_n(0) \sim \hat{\Pi}_{k}^{* n} (\tree_0^*, \cdot)$ for some $\tree^{* n} \in \T_{[k]}^{* n}$.
  Then 
  \begin{align*}
    \P \left( \That_n(1) = \that \ \middle \vert \ \hat{\pi}_{[k]}^{* n} \left( \That_n(0) \right) = \tree_0^*, \hat{\pi}_{[k]}^{* n} \left( \That_n(1) \right) = \tree_1^* \right)
    = \hat{\Pi}_k^{* n} (\tree_1^*, \{ \that \})
  \end{align*}
  for all $\that \in \Thatspace_{[n]}$, whenever $\tree_0^*, \tree_1^* \in \T_{[k]}^{*n}$ are such that
\begin{align*}
  \P \left( \hat{\pi}_{[k]}^{* n} \left( \That_n(0) \right) = \tree_0^*, \hat{\pi}_{[k]}^{* n} \left( \That_n(1) \right) = \tree_1^* \right) > 0.
\end{align*}
\end{lemma}
\noindent
Essentially, the following proof will follow the structure of Proposition~\ref{prop:decorated_transition_kernel}, meaning that we will split up the proof into cases based on the location and value of $I$ in $\hat{\pi}_{[k]}^{*n} \left( \That_n \right)$:
\begin{proof}
  Consider the collapsed tree $\tree_0^* = \left( \tree[s], {\left( B_x \right)}_{x \in \insertable (\tree[s])} \right) \in \T_{[k]}^{*n}$, the event that the projection of the inital state of the semi-planar $(\alpha, \gamma)$-chain onto $\T_{[k]}^{* n}$ yields $\tree_0^*$, $\left\{\hat{\pi}_{[k]}^{* n} \left( \That_n(0) \right) = \tree_0^* \right\}$, and the internal structure as characterized by the spatial Markov property in Lemma~\ref{lemma:spatialMarkovProperty}.
  In light of this result it is sufficient to show that the internal structure is preserved under the down-step if we further condition on the selected leaf $I$, the deleted leaf $\tilde{I}$, and in some cases additional information about label sets, ensuring that we uniquely characterize the collapsed tree after the down-step.
  Hence, we 
  \begin{enumerate}
    \item sample $\That$ from $\P \left( \That_n \in \cdot \ \middle \vert \ \hat{\pi}_{[k]}^{* n} \left( \That_n \right) = \tree_0^* \right)$,
    \item condition on $I = i$ and $\tilde{I} = \tildei$ (and in some cases additional information), delete $\tildei$, and relabel $\tildei + 1, \ldots, n$ by $\tildei, \ldots, n-1$ using the increasing bijection, obtaining $\That^\downarrow$, and
    \item project to $\T_{[k]}^{*(n-1)}$ to obtain $T^*$, and study the internal structures in $\That^\downarrow$.
  \end{enumerate}
  Even though the statement of the lemma formally involves the up-step, we will as an intermediary result show that
  \begin{align}
    &\P \left( \That_n^\downarrow = \that \ \middle \vert \ \hat{\pi}_{[k]}^{* n} \left( \That_n(0) \right) = \tree_0^*, \hat{\pi}_{[k]}^{* (n-1)} \left( \That^\downarrow \right) = \tree_1^{*\downarrow} \right) \nonumber \\
    = \quad &\P \left( \That_{n-1} = \that \ \middle \vert \ \hat{\pi}_{[k]}^{* (n-1)} \left( \That^\downarrow \right) = \tree_1^{*\downarrow} \right)
    \label{eq:intertwining_downsteppart_statement}
  \end{align}
  where $\tree_1^{* \downarrow}$ denotes $\tree_1^*$ with label $n$ removed.
  By definition of the down-up chain the statement of the lemma easily follows from~\eqref{eq:intertwining_downsteppart_statement}, since the up-step is simply defined by adding $\{ n \}$ to one of the $B_x$'s according to the same probabilities as in the decorated $(\alpha, \gamma)$-growth process, Definition~\ref{def:decorated_alphagamma_growth}.
  Now, from the spatial Markov Property (Lemma~\ref{lemma:spatialMarkovProperty}) we need mainly to focus on the internal structures of $\tree_0^*$ in $\That$ that are affected by the down-step.
  
  We now split up into cases based on $\tree_0^*$, $i$ and $\tildei$: \\

  \textbf{Case A1: $i \in B_x$ for external $x \in \edge (\tree[s])$, $y_x > 1$.}
  Note how in this case $I \in B_x$ implies that $\tilde{I} = \max \{ I, a, b \} \in B_x$ as well, where $a$ and $b$ are defined as in the semi-planar $(\alpha, \gamma)$-chain.
  This implies that the tree shape does not change during the down-step.  
  Hence conditioning on $I = i$ and $\tilde{I} = \tildei$ in $\tree_0^*$, is exactly the same as conditioning on the corresponding random variables $I' = \rank_{B_x}(i)$ and $\tilde{I}' = \rank_{B_x}(\tildei)$ in $\intstruct (x)$.
  Now, recall that {$\intstruct (x) \deq \That_{\# B_x}$}, and so Lemma~\ref{lemma:condind} ensures that, under this conditioning $\intstruct (x)$ after the down move is distributed as $\That_{\# B_x - 1}$.
  Since this is the only internal structure affected in the down-step all the other internal structures will still have the same distribution, and they will all still be conditionally independent as in Lemma~\ref{lemma:spatialMarkovProperty}.

  \textbf{Case A2: $i \in B_x$ for internal $x \in \edge (\tree[s])$.}
  This is very similar to Case A1, except that we have $I > k$ in this case.
  As above, note that conditioning on $I = i$ and $\tilde{I} = \tildei$, both being elements of $B_x$, is equivalent to conditioning on the corresponding random variables $I' = \rank_{B_x}(i)$ and $\tilde{I}' = \rank_{B_x}(\tildei)$ in $\intstruct (x)$.
  Note that $1 < I' \leq \tilde{I}' \leq \# B_x$, and recall that $\intstruct (x) \deq \That_{\# B_x}^\gamma$.
  Combining these observations with Corollary~\ref{cor:internal_downupchain_stationary_lemma_theorem} yields that $\intstruct (x)$ after the down-step has the same distribution as $\That_{\# B_x - 1}^\gamma$.

  \textbf{Case A3: $i \in B_x$ for $x \in \branchpoints (\tree[s])$.}
  This is the most complicated scenario due to the complexity of the internal structure of a branch point.
  But similarly to Cases A1 and A2, we do not change the tree shape since $I > k$ ensures that $\tilde{I} \in B_x$ as well.
  Hence we only need to argue that this internal structure has the correct (conditional) distribution after the down step, as in the previous cases.
  Similar to these cases, conditioning on $I = i$ and $\tilde{I} = \tildei$ in $\that_0^*$ is equivalent to conditioning on the corresponding random variables $I' = c + \rank_{B_x}(i)$ and $\tilde{I}' = c + \rank_{B_x}(\tildei)$ in $\intstruct(x)$, where $c$ is the number of children of $x$ in $\tree[s]$.
  
  From Lemma~\ref{lemma:spatialMarkovProperty}\ref{lemma:spatial5} we obtain that $\intstruct(x)$ has the same distribution as the $\# B_x$'th step of a $c$-order branch point $(\alpha, \gamma)$-growth process started from \\ $\rho (\intstruct (x), [c + \# B_x] \setminus [c])$.
  Hence Corollary~\ref{cor:korder_downupchain_stationary_lemma_theorem} immediately yields the results. \\

  Since no other ranked internal structures are altered in the above situation, and since we have now also conditioned on $\tildei$, we can easily relabel all the other label sets $B_{x'}$ for $x' \neq x$ if applicable. \\

  In all of the following cases, $v$ refers to the parent branch point of $i$ in $\tree[s]$.

\textbf{Case B1: $i \in B_x$ for external} $x \in \edge (\tree[s]), y_x = 1, y_v > 0, \tildei > k$.
  Firstly, we note that the restrictions on $i$ and $y_x$ imply that $i \leq k$.
  Now consider the semi-planar structure in the branch point $v$, and note how $y_x = 1$ and $y_v > 0$ imply that $\max \{i, a, b\} = \max \{i, b\}$, where $a$ and $b$ are defined as in the semi-planar $(\alpha, \gamma)$-chain.
Hence this case covers the events where (1) $i$ is located in a multifurcating branch point of $\That$ and the local search finds a label larger than $k$ to the left (or right, if $i$ is placed leftmost in $v$) of $i$, and (2) $i$ is located in a bifurcating branch point where there is a decorated mass in the branch point.
  The two situations are completely analogous so we will only provide the argument for the multifurcating case.
  Here, two internal structures are affected: $\intstruct(x)$ and $\intstruct(v)$.
  So condition on $I = i$ and $\tilde{I} = \tildei \in B_v$.
  By Lemma~\ref{lemma:spatialMarkovProperty}, conditional on the subtree of $v$ in $\That$ containing $\tildei$ having leaves labelled by $C_{\tildei}$, we will first show that $\intstruct_k^{T^\downarrow}(x)$ has the distribution of $\That_{\# C_{\tildei}}$.
  Consider how $\intstruct(v)$ is changed.
  First, conditional on $I = i, \tilde{I} = \tildei$ and $C_{\tildei}$, we will insert the leaves labelled by the elements of $C_{\tildei} \setminus \{\tildei\}$ into the subtree of $v$ with lowest label $\tildei$, and no other leaves will go into this subtree, so that what we extract as $\intstruct (x)$ in $\That^\downarrow$ is indeed distributed like $\That_{\# C_{\tildei}}$.
  Then, to form $\intstruct(v)$ in $\That$, we would still insert the leaves labelled by the elements of $B_v \setminus C_{\tildei}$ according to the same growth procedure that formed $\intstruct(v)$ in $\That$ to form the equivalent in $\That^\downarrow$; the leaves cannot (in this case) be inserted into either of the subtrees with smallest label $i$ and $\tildei$, respectively, nor in the gap between them.
  This proves that both the relevant internal structures have the correct conditional distribution after the down-step, given $I = i, \tilde{I}$ and $C_{\tildei}$.
  
\textbf{Case B2: $i \in B_x$ for external} $x \in \edge (\tree[s]), y_x = 1, y_v > 0, \tildei \leq k$.
  Following the same reasoning as above, this case only contains the event where the local search in $\That$ does not find a label larger than $k$ in a multifurcating branch point to the left (or right) of $i$.
  Hence $\tildei \in [k]$, and by swapping $i$ and $\tildei$ and subsequently deleting $\tildei$, we end up removing one of the leaves that defines the tree shape $\tree[s]$.
  The leaves $\tildei + 1, \ldots, k$ will be relabelled by $\tildei, \ldots, k-1$ using the increasing bijection, thus forming $\tree[s]' = \rho \left( \tau ( \tree[s], i, \tildei), \tildei \right)$, and $k$ will reappear in the following way.
  \begin{enumerate}
    \item If $k+1 \in B_{v^\prime}$ for some $v^\prime \in \branchpoints (\tree[s])$, the new tree shape will have $k$ as a child of $v^\prime$, i.e.\ the new tree shape is $\varphi \left(\tree[s]' , v^\prime, k \right)$.
    \item If $k+1 \in B_{e^\prime}$ for some $e^\prime \in \edge (\tree[s])$, $k$ will split the edge up into a new branch point and three edges, i.e.\ the new tree shape is $\varphi \left(\tree[s]' , e^\prime, k \right)$.
  \end{enumerate}
  Consequently, we need to focus on three internal structures in $\That$: $\intstruct(e_{\tildei \downarrow})$, $\intstruct(v)$ (where $v$ is the parent of $i$ in $\That$), and $\intstruct (x')$ where $k+1 \in B_{x'}$.
  
  The first of these remains the same, but becomes the internal structure on the leaf edge $i$ after the down-step.

  From arguments similar to Proposition~\ref{prop:generalresult_downstep_stationary_lemma_theorem} we obtain that $\intstruct(v)$ in $\That^\downarrow$ has the same distribution as the $y_v$'th step of a semi-planar $c-1$-order branch point $(\alpha, \gamma)$-process.

  Now observe how the relabelling naturally partitions $B_{x'}$ into either two blocks (if $x'$ is a branch point) or four blocks (if $x'$ is an edge), and which of the two situations we are in is determined solely by $\tree_0^*$.
  So if $x' = e_{v^\prime w} \in \edge(\tree[s])$ condition on $B_{e_{k \downarrow}}^\prime, B_{w^\prime}^\prime, B_{e_{v^\prime w^\prime}}^\prime$, and $B_{e_{w^\prime w}}^\prime$ partitioning $B_{x^\prime}$.
  On the other hand, if $x^\prime \in \branchpoints(\tree[s])$ condition on $B_{e_{k \downarrow}}^\prime$ and $B_{x^\prime}^\prime$ partitioning $B_{x^\prime}$.
  Considering how insertions of leaves were made into $\rho \left( T_n, [n] \setminus [k+1] \right)$ to obtain $T_n$, we deploy Lemma~\ref{lemma:spatialMarkovProperty} which yields that the internal structures corresponding to the branch points and edges outlined above have the correct distributions.

\textbf{Case B3: $i \in B_x$ for external} $x \in \edge (\tree[s])$, $y_x = 1$, $y_v = 0$, $\tildei > k$.
  This case is where we have a binary branch point, and there is decorated mass on the edge below, but is otherwise similar to Case B1.
  
\textbf{Case B4: $i \in B_x$ for external} $x \in \edge (\tree[s])$, $y_x = 1$, $y_v = 0$, $\tildei \leq k$.
  This case is where we have a binary branch point, and there is no decorated mass on the edge below, but is otherwise similar to Case B2. \\

  Since the above cases are exhaustive, and that we in all cases obtain the correct conditional distributions, this proves~\eqref{eq:intertwining_downsteppart_statement}.
\end{proof}
With Lemma~\ref{lemma:kemeny-snell} and Lemma~\ref{lemma:intertwining} in place, we are now ready to prove the final theorem of this exposition, from which Corollary~\ref{cor:consistency_stationarity} also follows easily.
\begin{thm}\label{thm:decorated_chain_projection_theorem}
  Fix $n \geq 2$.
  Let ${( \That_n(m) )}_{m \in \N_0}$ denote a semi-planar $(\alpha, \gamma)$-chain.
  For $1 \leq k \leq n$, let $T_{[k]}^{\bullet n}(m) := \hat{\pi}_{[k]}^{\bullet n} \left( \That_n(m) \right)$, $m \in \N_0$, denote the projection of the chain onto decorated $[k]$-trees of mass $n$.
  If $\That_n(0) \sim \hat{\Pi}_k^{\bullet n} (\tree^\bullet, \cdot)$ for some $\tree^\bullet \in \T_{[k]}^{\bullet n}$, then ${\left( T_{[k]}^{\bullet n}( m ) \right)}_{m \in \N_0}$ is a decorated $(\alpha, \gamma)$-chain on $\T_{[k]}^{\bullet n}$ started from $\tree^\bullet$.
\end{thm}
\begin{rem}
  Note that the case $k = n$ is essentially Theorem~\ref{thm:nonplanar_alphagamma_chain}.
\end{rem}
\begin{proof}
  This follows from the diagram in Figure~\ref{fig:commutativediagram_complete} commuting by Lemmas~\ref{lemma:kemeny-snell} and~\ref{lemma:intertwining}. 
  Sampling $\That_n(0)$ from $\hat{\Pi}_k^{\bullet n} (\tree^\bullet, \cdot)$ corresponds to first sampling a collapsed tree $T_k^{*n}(0)$ from $H$, defined in~\ref{eq:kernel_collapsed_decorated}, and subsequently sampling $\That_n(0)$ from $\hat{\Pi}_{k}^{* n} \left( T_k^{* n}(0), \cdot \right)$ by definition of the kernels.
  To obtain $T_{[k]}^{\bullet n}(1)$, one can now perform one step of the collapsed $(\alpha, \gamma)$-chain and then project the result to $\T_{[k]}^{\bullet n}$ using the function $h$.
  From Lemma~\ref{lemma:intertwining}, the intertwining criterion, Proposition~\ref{prop:intertwining_kemenysnell}\ref{prop:intertwining_criterion}, is satisfied for semi-planar $(\alpha, \gamma)$-chain, the function $\hat{\pi}_{[k]}^{* n}$ and the Markov kernel $\hat{\Pi}_{[k]}^{* n}$, and hence we obtain the collapsed $(\alpha, \gamma)$-chain on $\T_{[k]}^{* n}$ with initial distribution given by $T_k^{* n}(0)$.
  Finally, the Kemeny-Snell criterion, Proposition~\ref{prop:intertwining_kemenysnell}\ref{prop:kemenysnell_criterion}, and Lemma~\ref{lemma:kemeny-snell} ensures that projecting the collapsed $(\alpha, \gamma)$-chain with initial state $T_k^{* n}(0)$ to $\T_{[k]}^{\bullet n}$ using the function $h$ yields a Markov chain on $\T_{[k]}^{\bullet n}$ with transition kernel characterized in Definition~\ref{def:decorated_alphagamma_chain} started from $\tree^\bullet$.
  This proves the theorem.
\end{proof}
\begin{proof}[Proof of Theorem~\ref{thm:projection_decorated_chain}]
  This now follows immediately from Theorem~\ref{thm:decorated_chain_projection_theorem} and Theorem~\ref{thm:nonplanar_alphagamma_chain}.
\end{proof}

\bibliographystyle{abbrvnat}
\bibliography{bibliography}

\appendix

\end{document}